\documentclass[aap,preprint]{imsart}

\RequirePackage{amsthm,amsmath,amsfonts,amssymb,bbm,hyperref}
\RequirePackage[numbers]{natbib}
\RequirePackage{graphicx}
\graphicspath{ {./img/} }

\startlocaldefs
\numberwithin{equation}{section}

\theoremstyle{plain}
\newtheorem{lemma}{Lemma}
\newtheorem{theorem}{Theorem}
\newtheorem{cor}{Corollary}

\theoremstyle{remark}
\newtheorem{definition}{Definition}
\newtheorem{example}{Example}
\newtheorem{assump}{Assumption}
\newtheorem{remark}{Remark}
\newtheorem*{case}{Case}

\newcommand{\ind}[1]{\mathbbm{1}\left[#1\right]}
\newcommand{\Expect}[1]{\operatorname{\mathbb{E}}\left[#1\right]}

\newcommand{{\LPC}}{\textbf{LPC}}
\newcommand{\Prob}[1]{\operatorname{\mathbb{P}}\left[#1\right]}

\newcommand{\Reals}{\mathbb{R}}

\newcommand{\x}{\mathbf{x}}
\newcommand{\y}{\mathbf{y}}

\newcommand{\RBM}{\operatorname{RBM}}

\newcommand{\N}{\mathcal{N}}

\renewcommand{\d}{\operatorname{d}}

\newcommand{\floor}[1]{\lfloor #1 \rfloor}

\newcommand{\blow}{\underline{b}}
\newcommand{\rinv}{R^{-1}}
\newcommand{\sigmalow}{\underline{\sigma}}
\newcommand{\sigmahigh}{\overline{\sigma}}
\newcommand{\rinvk}[1]{\left(R|_{#1}\right)^{-1}}
\newcommand{\muk}[1]{\mu|_{#1}}
\newcommand{\bk}[1]{b^{(#1)}}
\newcommand{\ak}[1]{a^{(#1)}}
\newcommand{\Tk}[1]{T^{(#1)}}
\newcommand{\Lambdak}[1]{\Lambda^{(#1)}}
\newcommand{\blowk}[1]{\blow^{(#1)}}
\newcommand{\bnorm}[1]{\|#1 \|_{1, \beta}}
\newcommand{\supnorm}[1]{\|#1\|_{\infty}}
\newcommand{\bsupnorm}[2]{\|#1 \|_{\infty, #2}}
\newcommand{\thetak}[1]{\theta^{(#1)}}

\endlocaldefs

\begin{document}

\begin{frontmatter}
\title{Dimension-free local convergence and perturbations for reflected Brownian motions}
\runtitle{Dimension-free convergence rates for RBM}

\begin{aug}
\author[A]{\fnms{Sayan} \snm{Banerjee}\ead[label=e1]{sayan@email.unc.edu}},
\author[A]{\fnms{Brendan} \snm{Brown}\ead[label=e2,mark]{bb@live.unc.edu}}
\address[A]{Statistics and Operations Research, UNC-Chapel Hill, \printead{e1,e2}}

\end{aug}

\begin{abstract}
We describe and analyze a class of positive recurrent reflected Brownian motions (RBMs) in $\mathbb{R}^d_+$ for which local statistics converge to equilibrium at a rate independent of the dimension $d$. Under suitable assumptions on the reflection matrix, drift and diffusivity coefficients, dimension-independent stretched exponential convergence rates are obtained by estimating contractions in an underlying weighted distance between synchronously coupled RBMs. We also study the Symmetric Atlas model as a first step in obtaining dimension-independent convergence rates for RBMs not satisfying the above assumptions. By analyzing a pathwise derivative process and connecting it to a random walk in a random environment, we obtain polynomial convergence rates for the gap process of the Symmetric Atlas model started from appropriate perturbations of stationarity.
\end{abstract}



\end{frontmatter}


\section{Introduction}\label{sec:intro}

We say a continuous stochastic process $X$ is a solution to $\RBM(\Sigma, \mu, R)$ if it satisfies
\begin{equation}\label{eqn:rbm}
    X(x, t) = x + \mu t + D B(t) + R L(x, t)
\end{equation} 
for each $t > 0$ and $x \in \Reals^d_+ := \{x \in \Reals^d \> |\> x_i \ge 0, \ i = 1 \ldots d \}$. Here $\mu \in \Reals^d, D, R \in \mathbb{R}^{d \times d}$, $B$ is a $d$-dimensional Brownian motion and $\Sigma = DD^T$ is positive definite. We assume that $R = I - P^T$ for a matrix $P$ that is sub-stochastic (i.e. non-negative entries and row sums are bounded above by one) and transient (i.e. $P^n \rightarrow 0$ as $n \rightarrow \infty$). $L$ is the local time constraining $X$ to the positive orthant $\Reals^d_+$: For $x \in \Reals^d_+$, it is the non-decreasing, continuous process adapted to the natural filtration of the Brownian motion $B$ such that $X(x, t) \in \Reals^d_+$ for all $t \ge 0$ and
\begin{equation}\label{loctim}
L(x, 0) =0, \quad \quad \int_0^tX_i(x,s)dL_i(x, s)=0 \text{ for all } t >0, 1 \le i \le d.
\end{equation}
RBMs of the form \eqref{eqn:rbm} arise in a variety of situations, including heavy-traffic limits of queue-length processes in generalized Jackson networks with $d$ servers \cite{reiman_jackson_net,harrisonwilliams}, and gaps between $d+1$ competing particles in rank-based diffusions (e.g. \cite{karatzas_skewatlas,sarantsev}).

There is a large literature studying diffusions with oblique reflections, in cases both more specific and more general than \eqref{eqn:rbm}, and we give only a brief background describing previous work most relevant to the current article. The paper \cite{harrisonreiman} first proved \eqref{eqn:rbm} has a unique strong solution. More precisely, under the stated assumptions on the reflection matrix $R$, for each $x \in \Reals^d_+$, there is a unique pair of continuous stochastic processes $(X,L)$ satisfying \eqref{eqn:rbm}-\eqref{loctim}. Moreover, the collection $\{X(\cdot; x)\}_{\x \in \mathbb{R}^d_+}$ defines a strong Markov process (see \cite{harrisonwilliams}). The naturality of this assumption on $R$ stems from the fact that the routing matrix $P$ of any single-class open queueing network is sub-stochastic and transient \cite{harrisonreiman} which, in turn, translates to its heavy traffic limit described by equations of the form \eqref{eqn:rbm}-\eqref{loctim}. The conditions on $P$ in particular say that its spectral radius is strictly less than $1$.  The matrix $\Sigma = DD^T$ gives the covariance matrix associated with the diffusion term of \eqref{eqn:rbm}.

It was shown in \cite[Section 6]{harrisonwilliams} that \eqref{eqn:rbm} has a stationary distribution if and only if $\rinv \mu < 0$, and in that case the stationary distribution is unique. Intuitively, this stability condition can be understood by noting that the associated `noiseless system' (\eqref{eqn:rbm}-\eqref{loctim} taking $B\equiv 0$), which governs the long time stability properties of $\RBM(\Sigma, \mu, R)$, has $0$ as its unique attracting fixed point if $\rinv \mu < 0$ \cite{atarbudhiraja}. For the open queueing network whose heavy traffic limit gives $\RBM(\Sigma, \mu, R)$, this stability condition is equivalent to the traffic intensity at each server being less than its service rate, which is an `if and only if' condition for stability of the queueing network.

In this article, we are interested in the effect of dimension on convergence rates to stationarity for reflected Brownian motions (RBMs) from a variety of initial configurations. This is a natural consideration for steady state sampling and evaluating steady state performance for high dimensional RBMs. Towards this end, we will implicitly consider a family of processes $X^{(d)} \sim \RBM(\Sigma^{(d)}, \mu^{(d)}, R^{(d)})$ indexed by the dimension $d\ge 1$. For notational convenience, we will suppress the superscript $(d)$ in further discussion.

\subsection{Convergence rates for RBM: work till date}

To study convergence rates of $X$ to its stationary distribution, one can apply general methods like Harris' theorem via using appropriate Lyapunov functions and minorization conditions \cite{meyn2012markov}. For example, \cite{budhiraja_lee} uses this methodology to give exponentially fast convergence of $X(x,\cdot)$ to the stationary random variable in a weighted total variation norm starting from any $x \in \Reals_d^+$. However, the rate of convergence is not explicit, as is typical for such methods, and in particular has unknown dimension dependence. See also \cite{sarantsev_RBM_tail} for a similar treatment.

In \cite{blanchet-chen}, the authors obtained explicit dimension dependent convergence rates to stationarity in $L^1$-Wasserstein distance when the RBM satisfies `uniformity conditions in dimension' on the model parameters $\Sigma, \mu,R$ (discussed here in more detail in Example \ref{ex:bc}). Their key insight was to consider \emph{synchronous couplings} of the RBM $X$ (i.e. driven by the same Brownian motion) started from distinct points $x,y \in \mathbb{R}^d_+$, with $x \le y$ (co-ordinate wise ordering). They used the fact that synchronous couplings preserve ordering in time, that is, $X(x,t) \le X(y,t)$ for all $t \ge 0$. Moreover, there are contractions in $L^1$ distance between the synchronously coupled processes (under their uniformity assumptions) when the dominating process $X(y,\cdot)$ has hit all faces of the orthant $\Reals^d_+$. Building on this idea, \cite{banerjeebudhiraja} used a weighted Lyapunov function and excursion theoretic control of the synchronously coupled processes to give convergence rates in $L^1$-Wasserstein distance for the general process \eqref{eqn:rbm} which depend explicitly on $\mu, R, \Sigma,d$. In particular, this approach greatly improved the rates for the models considered in \cite{blanchet-chen} from polynomial in $d$ to logarithmic in $d$.

\subsection{Dimension-free local convergence for RBM} 

Typically, growing dimension slows down the rate of convergence for the whole system, as reflected in the bounds obtained in \cite{blanchet-chen,banerjeebudhiraja}, but one might observe a much faster convergence rate to equilibrium of \emph{local statistics} of the system. 
In Section \ref{sec:dfsec}, we describe and investigate a class of RBMs for which convergence rates of local statistics do not depend on the underlying dimension of the entire system. We call this phenomenon dimension-free local convergence. 

Mathematically, this is challenging as the local evolution is no longer Markovian and the techniques in \cite{blanchet-chen,banerjeebudhiraja} cannot be readily applied. We make a crucial observation that certain weighted $L^1$ distances (see $\|\cdot\|_{1,\beta}$ defined in Section \ref{wnormdef}) between synchronously coupled RBMs show dimension-free contraction rates. The evolution of such weighted distances are tracked in time for synchronously coupled RBMs $X(0,\cdot)$ and $X(x,\cdot)$ for $x \in \mathbb{R}^d_+$. It is shown that for this distance to decrease by a dimension-free factor of its original value, only a subset of co-ordinates of $X(x,\cdot)$, whose cardinality depends on the value of the original distance, need to hit zero. This is in contrast with the unweighted $L^1$ distance considered in \cite{blanchet-chen,banerjeebudhiraja} where all the coordinates need to hit zero to achieve such a contraction, thereby slowing down the convergence rate. Consequently, by tracking the hitting times to zero of a time dependent number of co-ordinates, one achieves dimension-free convergence rates in this weighted $L^1$ distance as stated in Theorem \ref{thm:main_fromx}. This, in turn, gives dimension-free local convergence as is made precise in \eqref{locstat}. In Section \ref{sec:examples}, Theorem \ref{thm:main_fromx} is applied to two important classes of RBM to obtain explicit convergence rates. 

\subsection{Perturbations from stationarity for the Symmetric Atlas Model}
As a first step in studying dimension-free convergence rates for RBMs which do not satisfy the assumptions of Section \ref{sec:dfsec}, we focus attention in Section \ref{sec:atlas_perturbation} on the Symmetric Atlas model. This is a rank-based diffusion comprising $d+1$ Brownian particles where the least ranked particle performs a Brownian motion with constant positive drift and the remaining particles perform standard Brownian motions. The gaps between the ordered particles collectively evolve as a RBM which converges in total variation distance to an explicit stationary measure \eqref{eqn:atlas_stationary} \cite{pitman_pal}. Interestingly, the gap process of the infinite-dimensional version of the Symmetric Atlas model obtained in \cite{pitman_pal} has infinitely many stationary measures \cite{sarantsev2017stationary}, only one of which is a weak limit of the stationary measure \eqref{eqn:atlas_stationary} of the $d$-dimensional system (thought of as a measure in $\mathbb{R}^{\infty}_+$) as $d \rightarrow \infty$. This leads to the heuristic that, for large $d$, the $d$-dimensional gap process with initial distribution `close' to the projection (onto the first $d$ co-ordinates) of one of the other infinite-dimensional stationary measures spends a long time near this projection before converging to \eqref{eqn:atlas_stationary}. From this heuristic, one expects that dimension-free convergence rates for associated statistics can only be obtained if the initial gap distribution is `close' to the stationary measure \eqref{eqn:atlas_stationary} in a certain sense. Evidence for this heuristic is provided in the few available results on `uniform in dimension' convergence rates of some rank-based diffusions \cite{jourdain2008propagation,jourdain2013propagation}. In both these papers, under strong convexity assumptions on the drifts of the particles, dimension-free exponential ergodicity was proven for the joint density of the particle system when the initial distribution is close to the stationary distribution as quantified by the Dirichlet energy functional (see \cite[Theorem 2.12]{jourdain2008propagation} and \cite[Corollary 3.8]{jourdain2013propagation}). The Symmetric Atlas model lacks such convexity in drift and hence, the dimension-free Poincar\'e inequality for the stationary density, that is crucial to the methods of \cite{jourdain2008propagation,jourdain2013propagation}, does not apply. We take a very different approach which involves analyzing the long term behavior of pathwise derivatives of the RBM in initial conditions. Using this analysis, we obtain polynomial convergence rates to stationarity in $L^1$-Wasserstein distance when the initial distribution of the gaps between particles is in an appropriate perturbation class (defined in Definition \ref{def:perturbation_class}) of the stationary measure. Although we do not yet have lower bounds on convergence rates, we strongly believe that the optimal rates are indeed polynomially decaying in time (see Remark \ref{powerrem}).

We mention here that \cite{blanchent2020efficient} has recently used the derivative process to study convergence rates for RBMs satisfying strong uniformity conditions in dimension (which do not hold for the Symmetric Atlas model). Our analysis of the derivative is based on a novel connection with a random walk in a random environment generated by the times and locations where the RBM hits faces of $\mathbb{R}^d_+$ (see Section \ref{pathrwre}). We believe our analysis can be combined with that of \cite{blanchent2020efficient} to study ergodicity properties of more general classes of RBM. This is deferred to future work.

We also mention the work of \cite{pal2019note} who obtained a dimension-free Talagrand type  transportation cost-information inequality for reflected Brownian motions. Such inequalities, however, are more useful in dimension-free concentration of measure phenomena as opposed to dimension-free rates of convergence to stationarity.

\subsection{Future work: Ergodicity of infinite-dimensional RBMs}
Although we only consider finite large $d$, our work sets the stage for obtaining convergence rates to stationarity for infinite-dimensional RBM, which we will address in future work. Such processes have appeared in numerous situations \cite{pitman_pal,ichiba_karatzas,sarantsev2017two,cabezas2019brownian,dembo2017equilibrium} but their ergodicity properties are far from being well understood. Notable advances in this direction are made in \cite{sarantsev_infinite,dembo2019infinite,banerjee2021domains}. \cite{sarantsev_infinite} partially characterizes weak limits of the gap process of the infinite-dimensional Symmetric Atlas model \cite[Section 4.3]{sarantsev_infinite}. \cite{dembo2019infinite} obtains general conditions on the initial configuration of the above gap process for it to converge weakly to $\bigotimes_{i = 1}^{\infty} \text{Exp}\left(2\right)$ \cite[Theorem 1.1]{dembo2019infinite}. In \cite{banerjee2021domains}, general conditions were given on the initial gap distribution of the infinite-dimensional Symmetric Atlas model for the time average of the gaps to converge to one of the (uncountably many) stationary measures given by $\bigotimes_{i = 1}^{\infty} \text{Exp}\left(2 + ia\right), \, a \ge 0$.

 The dimension-free convergence rates obtained here can be combined with taking a limit in dimension in an appropriate sense to obtain convergence rates in Wasserstein distance for infinite-dimensional RBMs starting from appropriate initial configurations. This is part of the first author's continuing program of studying ergodicity properties of infinite-dimensional systems \cite{banerjee2020rates,banerjee2020ergodicity,banerjee2021domains}. 

\subsection{Generic notation}\label{sec:notation}
Here we list notation for general concepts and conventions. Inequalities for vectors are evaluated element-wise.  For a square matrix $A$,  $A|_k$ is the $k\times k$ northwest quadrant. For a vector $v$, $v|_k$ is the projection of $v$ onto the first $k$ coordinates. Other conventions include $x\vee y = \max(x, y)$, $x \wedge y = \min(x, y)$, $\floor{x} = \max\{k \in \mathbb{Z} \> |\> k\le x\}$ and $x^+ = \max(0, x)$.

For $x \in \Reals^k$, we write the supremum norm as $\supnorm{x} = \max_{1 \le i \le k}|x_i|$ and the $\ell^1$ norm as $\|x\|_1 :=  \sum_{i=1}^k |x_i|$. For a fixed $\beta \in (0, 1)$, define a weighted $\ell^1$ norm by $\bnorm{x} = \sum_{i=1}^k \beta^i |x_i|$ and weighted supremum norm by $\bsupnorm{x}{\beta} = \max_{1 \le i \le k}\beta^i |x_i|$.

For $X$ a $\RBM(\Sigma, \mu, R)$ started at $x \in \Reals^d_+$ and any $k \in \{1  \ldots d\}$, we write $X(\infty)$ for the random variable with the stationary distribution. Write $X|_k(\cdot, x)$ for the process restricted to its first $k$ coordinates.

\section{Dimension-free local convergence rates for RBM}\label{sec:dfsec}

\subsection{A weighted norm governing dimension-free convergence}\label{wnormdef}
Our investigation of dimension-free convergence relies on the analysis of the weighted distance $\|X(x,\cdot) - X(X(\infty),\cdot)\|_{1,\beta}$ in time, for appropriate choices of $\beta \in (0,1)$.
Towards this end, we will analyze the following functionals:
\begin{align}
    u_{\beta}(x, t) &= \bnorm{\rinv\left(X(x, t) - X(0, t)\right)} := \sum_{i=1}^d \beta^i \left| \left[\rinv\left(X(x, t) - X(0, t)\right)\right]_i \right|,\label{eqn:weighted_norm}\\ 
    u_{\pi,\beta}(t) &= u_{\beta}(X(\infty), t), \quad \quad t \ge 0. \label{eqn:u_pi}
\end{align}
In the following, when $\beta$ is clear from context, we will suppress dependence on $\beta$ and write $u$ for $u_{\beta}$ and $u_{\pi}$ for $u_{\pi,\beta}$. The above functionals are convenient because the vector $\rinv\left(X(x, t) - X(0, t)\right)$ is co-ordinate wise non-negative and non-increasing in time (see Theorem \ref{thm:rbm_monotonicity} (iii)). This fact and the triangle inequality can be used to show for any $x \in \mathbb{R}^d_+, t \ge 0$ (see \eqref{eqn:main_fromx1}),
\begin{align*}
\bnorm{\left(X(x, t) - X(X(\infty), t)\right)} \le u(x, t) + u_\pi(t).
\end{align*}

We are interested in conditions under which there exists a $d$-independent $\beta \in (0,1)$ and a function $f: \Reals_+ \mapsto \Reals_+$ not depending on the dimension $d$ of $X$ such that $f(t) \to 0$ as $t \to \infty$ and, for any $x$ in an appropriate subset $\mathcal{S}$ of $\mathbb{R}_+^d$,
\begin{equation}\label{eqn:dimension_free}
    \Expect{\bnorm{\left(X(x, t) - X(X(\infty), t)\right)}} \le C f(t), \quad \quad t \ge t_0,
\end{equation}
where $C, t_0 \in (0, \infty)$ are constants not depending on $d$ (but can depend on $x$). This, in particular, gives dimension-free local convergence in the following sense: For any $k \in \{1,\dots,d\}$, consider any function $\phi: \mathbb{R}^k_+ \mapsto [0,\infty)$ which is $L^1$-Lipschitz, i.e., there exists $L_{\phi} >0$ such that
$$
|\phi(x) - \phi(y)| \le L_{\phi} \|x-y\|_1, \quad \quad x,y \in \mathbb{R}^k_+.
$$
Recall that the $L^1$-Wasserstein distance between two probability measures $\mu$ and $\nu$ on $\mathbb{R}^k_+$ is given by
$$
W_1(\mu, \nu) = \inf\left\lbrace\int_{\mathbb{R}^k_+ \times \mathbb{R}^k_+}\|\x-\y\|_1 \gamma(\operatorname{d}\x, \operatorname{d}\y) \ : \gamma \text{ is a coupling of } \mu \text{ and } \nu\right\rbrace.
$$
Denote the law of a random variable $Z$ by $\mathcal{L}[Z]$. Then, \eqref{eqn:dimension_free} implies for $x \in \mathcal{S}$,
\begin{align}\label{locstat}
W_1\left(\mathcal{L}[\phi(X|_k(x, t))], \mathcal{L}[\phi(X|_k(\infty))]\right) 
&\le \Expect{|\phi(X|_k(x, t)) - \phi(X|_k(X(\infty), t))|} \nonumber \\
&\le C\beta^{-k} L_{\phi} f(t), \quad \quad t \ge t_0.
\end{align}

\subsection{Parameters and Assumptions}
We now define the parameters that govern dimension-free local convergence which, in turn, are defined in terms of the original model parameters $(\Sigma,\mu,R)$ of the associated RBM. 
We abbreviate $\sigma_i = \sqrt{\Sigma_{ii}}, i=1,\ldots, d$. Define for $1 \le k \le d$,
\begin{align}\label{pardefcon}
\bk{k} &:= -\rinvk{k}\muk{k}, \quad \quad b = \bk{d}, \notag\\
\blowk{k} &:= \min_{1 \le i \le k} \bk{k}_i,  \quad \quad \ak{k} := \max_{1\le i \le k} \frac{1}{\bk{k}_i} \sum_{j=1}^k (\rinvk{k})_{ij}\sigma_j.
\end{align}
To get a sense of why these parameters are crucial, recall that our underlying strategy is to obtain contraction rates of $u(x,\cdot)$ defined in \eqref{eqn:weighted_norm} by estimating the number of times a subset of the co-ordinates of $X(x,\cdot)$, say $\{X_1(x,\cdot),\dots,X_k(x,\cdot)\}, k \le d$, hit zero. However, this subset does not evolve in a Markovian way. Thus, we use monotonicity properties of RBMs to couple this subset with a $\mathbb{R}^k_+$-valued reflected Brownian motion $\bar{X}(x|_k,\cdot)$, started from $x|_k$ and defined in terms of $\mu|_k, D|_k, R|_k$ and (a possible restriction of) the same Brownian motion driving $X(x,\cdot)$, such that $X_i(x,t) \le \bar{X}_i(x|_k,t)$ for all $1 \le i \le k$ (see Theorem \ref{thm:domination}). The analysis in \cite{banerjeebudhiraja} shows that the parameters defined in \eqref{pardefcon} with $k=d$ can be used to precisely estimate the minimum number of times all co-ordinates of $X(x,\cdot)$ hit zero by time $t$ as $t$ grows. Thus, for any $1 \le k \le d$, the parameters \eqref{pardefcon} can be used to quantify analogous hitting times for the process $\bar{X}(x|_k,\cdot)$ which, by the above coupling, gives control over corresponding hitting times of $\{X_1(x,\cdot),\dots,X_k(x,\cdot)\}$.

We list below two sets of assumptions on the model parameters $(\Sigma,\mu,R)$ which guarantee dimension-free local convergence.
\begin{assump}\label{assump:main}
There exist $d$-independent constants $\sigmalow, \sigmahigh, b_0 > 0$, $r^* \ge 0$, $M, C \ge 1$, $k_0 \in \{2, \ldots, d\}$ and $\alpha\in (0, 1)$ such that for all $d \ge k_0$, 
\begin{itemize}
    \item[I.] $(\rinv)_{ij} \le C \alpha^{j-i}$ for $1\le i \le j \le d$,
    \item[II.] $(\rinv)_{ij} \le M$ for $1 \le i, j \le d$,
    \item[III.] $\blowk{k} \ge b_0 k^{-r^*}$ for $k = k_0, \ldots, d$,
    \item[IV.] $\sigma_i \in [\sigmalow, \sigmahigh]$ for $1 \le i \le d$.
\end{itemize}
\end{assump}
We explain why Assumption \ref{assump:main} is `natural' in obtaining dimension-free local convergence. Since $P$ is a transient and substochastic, it can be associated to a killed Markov chain on $\{0\} \cup \{1,\dots,d\}$ with transition matrix $P$ on $\{1,\dots,d\}$ and killed at $0$ (i.e. probability of going from state $k \in \{1,\dots,d\}$ to $0$ is $1 - \sum_{l=1}^d P_{kl}$ and $P_{00} =1$). Moreover, since $P$ is transient and $R = I - P^T$, we have $\rinv = \sum_{n=0}^\infty (P^T)^n$. This representation shows that $(\rinv)_{ij}$ is the expected number of visits to site $i$ starting from $j$ of this killed Markov chain. For fixed $x \in \mathbb{R}^d_+$ and $k < < d$, consider a local statistic of the form $\phi(X|_k(x, t))$ as in \eqref{locstat}. For this statistic to stabilize faster than the whole system, we expect the influence of the far away co-ordinates $X|_j(x,\cdot), j>>k,$ to diminish in an appropriate sense as $k$ increases. This influence is primarily manifested through the oblique reflection arising out of the $R$ matrix in \eqref{eqn:rbm}. I of Assumption \ref{assump:main} quantifies this intuition by requiring that the expected number of visits to state $i$ starting from state $j>i$ of the associated killed Markov chain decreases geometrically with $j-i$. This is the case, for example, when this Markov chain started from $j>i$ has a uniform `drift' away from $i$ towards the cemetery state. See Example \ref{ex:skew_atlas}. In more general cases, one can employ Lyapunov function type arguments \cite{meyn2012markov} to the underlying Markov chain to check I. 

II above implies that the killed Markov chain starting from state $j$ spends at most $M$ expected time at any other site $i \in \{1,\dots,d\}$ before it is absorbed in the cemetery state $0$. This expected time, as our calculations show, is intimately tied to decay rates of $\bnorm{\left(X(x, \cdot) - X(X(\infty), \cdot)\right)}$. 

As noted in \cite{harrisonwilliams,blanchet-chen,banerjeebudhiraja}, the `renormalized drift' vector $b$ characterizes positive recurrence of the whole system. Through III above, we allow for a power law type co-ordinate wise lower bound of the renormalized drift vector $b^{(k)}$ of the projected system $X|_k(x, \cdot)$ as $k$ grows. In particular, if $\blowk{k}$ is uniformly lower bounded by $b_0$, we can take $r^*=0$.

IV above is a quantitative `uniform ellipticity' condition on the co-ordinates of the driving noise $DB(\cdot)$.

Note that we do not need to make any assumptions on the correlations of the driving noise, i.e. on $\sigma_{ij}/(\sigma_i,\sigma_j)$ for $i<j$. This can be understood upon inspection of our proof technique where the drift and the reflection `overpower' the diffusivity in long time contraction properties of $\bnorm{\left(X(x, \cdot) - X(X(\infty), \cdot)\right)}$.
The following assumption is a strengthening of Assumption \ref{assump:main} which, when satisfied, will lead to significantly better convergence rates to stationarity.
\begin{assump}\label{assump:bounded_row}
Suppose Assumption \ref{assump:main} holds. In addition assume $M$, which does not depend on $d$, may be chosen large enough that
\begin{itemize}
    \item[II'.] $\max_{1 \le i \le d} \sum_{j = 1}^d (\rinv)_{ij} \le M$. 
\end{itemize}
\end{assump}
This is satisfied, for example, when there exist positive $d$-independent constants $j_0,p_0$ such that the underlying killed Markov chain has jump size bounded by $j_0$ at each step, and a probability of at least $p_0$ of reaching $0$ in one step from any starting site in $\{1,\dots,d\}$. See Example \ref{ex:bc} for such a RBM.

\subsection{Main results}\label{sec:results}
Our first main result gives explicit bounds on the decay of expectation of the weighted distance $\|X(x,\cdot) - X(X(\infty),\cdot)\|_{1,\sqrt{\alpha}}$ ($\alpha$ defined in I of Assumption \ref{assump:main}) with time for RBMs satisfying Assumption \ref{assump:main} or \ref{assump:bounded_row}.
We first define some constants that will appear in Theorem \ref{thm:main_fromx}. They are needed to bound moments of weighted norms of the stationary random variable $X(\infty)$ and are derived in Lemma \ref{lemma:stat_moments}.

Suppose Assumption \ref{assump:main} holds, with $k_0 \in \{2 \ldots d\}$ and $\alpha \in (0, 1)$ defined therein. Set
\begin{equation}\label{l1def}
    L_1 := k_0^{r^* + 1} +  \sum_{i = k_0}^d i^{3 + r^*} \alpha^{i/8}.
\end{equation}
If in addition Assumption \ref{assump:bounded_row} holds, define 
\begin{equation}\label{l2def}
    L_2 := k_0^{r^*} +  \sum_{i = k_0}^d i^{2 + r^*} \alpha^{i/8}.
\end{equation}
Also, for $B \in (0,\infty)$, define the set
\begin{equation}\label{startset}
\mathcal{S}(b,B) := \left\lbrace x \in \mathbb{R}_+^d : \sup_{1 \le i \le d} \blow^{(i)} \supnorm{x|_i} \le B\right\rbrace.
\end{equation}
Theorem \ref{thm:main_fromx} directly implies dimension-free bounds on $t \mapsto \|X(x,t) - X(X(\infty),t)\|_{1,\sqrt{\alpha}}$ in the sense of \eqref{eqn:dimension_free} (see Remark \ref{rem1}) which, in turn, produce dimension-free local convergence rates as given by \eqref{locstat}.
\begin{theorem}\label{thm:main_fromx}
Suppose Assumption \ref{assump:main} holds for $X$, an $\RBM(\Sigma, \mu, R)$, with $\alpha \in (0, 1)$ defined therein. Recall the weighted distance $\|\cdot\|_{1,\sqrt{\alpha}}$ (taking $\beta = \sqrt{\alpha}$) defined in Section \ref{sec:notation}. 

Fix any $B \in (0,\infty)$. Then there exist constants $C_0, C_0', C_1 > 0$ not depending on $d, r^*$ or $B$ such that with $t_0' = t_0'(r^*) = C_0'\left(1+r^*\right)^{8+4r^*}$ and $L_1, L_2, \mathcal{S}(b,B)$ as defined in \eqref{l1def}-\eqref{startset}, we have for any $x \in \mathcal{S}(b,B)$ and any $d > t_0'^{1/(4+2r^*)}$,
\begin{multline}\label{eqn:main_fromx_result1}
    \Expect{\|X(x, t) - X(X(\infty), t)\|_{1,\sqrt{\alpha}}}\\
    \le
\begin{cases}
C_1 \left(L_1\sqrt{1+t^{1/(4+2r^*)} } + \supnorm{x}\exp\left\lbrace B/\sigmalow^2\right\rbrace\right) \> \exp\{-C_0 t^{1/(4+2r^*)} \},
  &   t_0' \le t < d^{4 + 2r^*}\\\\
C_1 \left(L_1\sqrt{1+t^{1/(4+2r^*)} } + \supnorm{x}\exp\left\lbrace B/\sigmalow^2\right\rbrace\right)\> \exp\left\lbrace -C_0\frac{t}{d^{3 + 2r^*}}\right\rbrace,   & t \ge d^{4 + 2r^*}.
\end{cases}
\end{multline}
If instead Assumption \ref{assump:bounded_row} holds, with $t_1' = t_1'(r^*) = C_0'\left(1+r^*\right)^{2+4r^*}$, we have for $d > t_1'^{1/(1+2r^*)}$,
\begin{multline}\label{eqn:main_fromx_result2}
  \Expect{\|X(x, t) - X(X(\infty), t)\|_{1,\sqrt{\alpha}}}\\
  \le
  \begin{cases}
  C_1 \left(L_2\sqrt{1+t^{1/(1+2r^*)} } + \supnorm{x}\exp\left\lbrace B/\sigmalow^2\right\rbrace\right) \> \exp\left\lbrace -C_0\frac{t^{1/(1+2r^*)}}{\log t}\right\rbrace,
 &   t_1' \le t < d^{1 + 2r^*}\\\\
C_1 \left(L_2\sqrt{1+t^{1/(1+2r^*)} } + \supnorm{x}\exp\left\lbrace B/\sigmalow^2\right\rbrace\right)\> \exp\left\lbrace - C_0 \frac{t}{d^{2r^*}\log d}\right\rbrace,
&   t \ge d^{1 + 2r^*}.
 \end{cases}
\end{multline}
\end{theorem}
\begin{remark}\label{rem1}
We note the following.
\begin{itemize}
\item[(i)] The bounds in Theorem \ref{thm:main_fromx} indeed imply dimension-free convergence in the $\|\cdot\|_{1,\sqrt{\alpha}}$ norm in the sense of \eqref{eqn:dimension_free}. To see this, note that under Assumption \ref{assump:main} the constant $r^*$ does not depend on dimension $d$. Thus, since $\frac{t}{d^{3 + 2r^*}} \ge t^{1/(4+2r^*)}$ for $t \ge d^{4 + 2r^*}$, \eqref{eqn:main_fromx_result1} implies the following dimension-free convergence rate bound,
\begin{align*}
&\Expect{\|X(x, t) - X(X(\infty), t)\|_{1,\sqrt{\alpha}}}\\
&\le C_1 \left(L_1\sqrt{1+t^{1/(4+2r^*)} } + \supnorm{x}\exp\left\lbrace B/\sigmalow^2\right\rbrace\right) \> \exp\{-C_0 t^{1/(4+2r^*)} \}\\
&\le C_1 \left(L_1 + \supnorm{x}\exp\left\lbrace B/\sigmalow^2\right\rbrace\right) \>\sqrt{1+t^{1/(4+2r^*)} } \exp\{-C_0 t^{1/(4+2r^*)} \}, \ \ t \ge t_0'.
\end{align*}
Similarly, the bound in the first part of \eqref{eqn:main_fromx_result2} continues to hold for all $t \ge t_1'$. The bounds in Theorem \ref{thm:main_fromx} are presented in the given form to emphasize that the weighted distances, in fact, decay exponentially with coefficients depending on $d$.
\item[(ii)] Bounds analogous to those in Theorem \ref{thm:main_fromx} hold using the norm $\bnorm{\cdot}$ for any $\beta \in (\alpha,1)$, with appropriately adjusted constants depending on $\beta$, and the choice $\beta = \sqrt{\alpha}$ is merely for simplicity of exposition. In fact, our proofs are in terms of two parameters $\beta \in (\alpha,1)$ and $\delta \in (\beta,1)$, which can be appropriately chosen for the specific RBM under consideration to optimize the obtained bounds.
\end{itemize}
\end{remark}

\subsection{Applications of Theorem \ref{thm:main_fromx}}\label{sec:examples}
Here, we present two examples of RBMs that arise in diverse applications, where we can apply Theorem \ref{thm:main_fromx} to obtain explicit dimension-free convergence rates.
\begin{example}[Asymmetric Atlas model]\label{ex:skew_atlas}
\textup{We consider Atlas-type models, which are interacting particle systems represented by the following SDE:
\begin{equation}\label{ascol}
Z_k(t) = Z_k(0) + \ind{k=1} t + B^*_k(t) + p L_{(k-1,k)}(t) - q L_{(k,k+1)}(t), \quad t \ge 0,
\end{equation}
for $1\le k \le d+1$, $p \in (0,1), q= 1-p$. Here, $L_{(0,1)}(\cdot) \equiv L_{(d+1,d+2)}(\cdot) \equiv 0$, and for $1 \le k \le d$, $L_{(k,k+1)}(\cdot)$ is a continuous, non-decreasing, adapted process that denotes the collision local time between the $k$-th and $(k+1)$-th co-ordinate processes of $Z$, namely $L_{(k,k+1)}(0)=0$ and $L_{(k,k+1)}(\cdot)$ can increase only when $Z_{k} = Z_{k+1}$.
$B^*_k(\cdot)$, $1\le k \le d+1$, are mutually independent standard one dimensional Brownian motions. Each of the $d+1$ ranked particles with trajectories given by $(Z_1(\cdot),\dots,Z_{d+1}(\cdot))$ evolves as an independent Brownian motion (with the particle $1$ having unit positive drift) when it is away from its neighboring particles, and interacts with its neighbors through possibly asymmetric collisions. The Symmetric Atlas model, namely the case $p=1/2$, was introduced in \cite{fernholz2002stochastic} as a mathematical model for stochastic portfolio theory. The Asymmetric Atlas model, namely the case $p \in (1/2,1)$, was introduced in \cite{karatzas_skewatlas}. It was shown that it arises as scaling limits of numerous well known interacting particle systems involving asymmetrically colliding random walks \cite[Section 3]{karatzas_skewatlas}. Since then, this model has been extensively analyzed: see \cite{karatzas_skewatlas,pal,ichiba_karatzas,sarantsev_infinite} and references therein.}

\textup{The gaps between the particles, defined by $X_i(\cdot) = Z_{i+1}(\cdot) - Z_i(\cdot), 1 \le i \le d$, evolve as an $\RBM(\Sigma, \mu, R)$ with $\Sigma$ given by $\Sigma_{ii} = 2$ for $i = 1 \ldots d$, $\Sigma_{ij}=-1$ if $|i-j|=1$, $\Sigma_{ij} = 0$ if $|i-j| >1$, $\mu$ given by $\mu_1 = -1, \mu_j = 0$ for $j = 2, \ldots, d$, and $R = I-P^T$, where
\begin{equation}\label{eqn:skew_atlas1}
    P_{ij} = \begin{cases}
    p \quad \quad j = i+1,\\
    1-p \quad \quad j = i-1, \\
    0 \quad \quad \text{otherwise}.
    \end{cases}
\end{equation}}
\textup{In this article, we are interested in the ergodicity of the gap process $X$. In the current example, we study the Asymmetric Atlas model. The Symmetric Atlas model is treated in Section \ref{sec:atlas_perturbation}.} 

\textup{Recall that the reflection matrix $R= I-P^T$ is associated with a killed Markov chain. For the Asymmetric Atlas model, this Markov chain has a more natural description as a random walk on $\{0,1,\dots,d+1\}$ which increases by one at each step with probability $p$ and decreases by one with probability $1-p$, and is killed when it hits either $0$ or $d+1$. Then for $1 \le i,j \le d$, $(\rinv)_{ij}$ is the expected number of visits to $i$ starting from $j$ by this random walk before it hits $0$ or $d+1$. Since $p > 1-p$, the random walk has a drift towards $d+1$, which suggests I, II of Assumption \ref{assump:main} hold. This is confirmed by direct computation, which gives for $q = 1-p$,
\begin{equation}\label{eqn:skew_atlas2}
    (\rinv)_{ij} = \begin{cases}
     \frac{\left(q/p \right)^{j-i}}{p-q}\frac{\left(1 - (q/p)^i \right)\left(1-(q/p)^{d+1-j} \right)}{1- (q/p)^{d+1}} \le \frac{\left(q/p \right)^{j-i}}{p-q} & 1 \le i \le j \le d,\\
     \frac{\left(p/q \right)^{i-j}}{p-q}\frac{\left((p/q)^j - 1 \right)\left((p/q)^{d+1-i} - 1 \right)}{(p/q)^{d+1} - 1} \le \frac{1}{p-q} & 1 \le j <  i \le d.
    \end{cases}
\end{equation}
Now I, II and IV of Assumption \ref{assump:main} holds with $M = C = \frac{1}{p-q}$, $\alpha = \frac{q}{p}$ and $\sigmalow=\sigmahigh=\sqrt{2}$. Furthermore, the restriction $P|_{k}$ is defined exactly as in \eqref{eqn:skew_atlas1} with $k$ in place of $d$. Thus $\rinvk{k}$ is given by \eqref{eqn:skew_atlas2} with $k$ in place of $d$, and $\bk{k} = -\rinvk{k} \mu|_k$ is the first column of $\rinvk{k}$. This entails,
\begin{align}\label{eqn:skew_atlas3}
    \bk{k}_i &= \frac{\left(p/q \right)^{i-1}}{p-q}\frac{\left((p/q) - 1 \right)\left((p/q)^{k+1-i} - 1 \right)}{(p/q)^{k+1} - 1} 
    \ge \frac{1}{q}\left(\frac{p}{q}\right)^{k-1}\frac{\left((p/q) - 1\right)}{(p/q)^{k+1}} \notag\\
    &= \frac{p-q}{p^2}  =: b_0 > 0, \quad \quad 1 \le i \le k, \> 1 \le k \le d.
\end{align}
Thus $\blowk{k} \ge b_0$ for all $1 \le k \le d$, uniformly in $d$. This shows that III of Assumption \ref{assump:main} holds with $b_0$ specified by \eqref{eqn:skew_atlas3} and $r^* = 0$. Moreover, it follows from the first equality in \eqref{eqn:skew_atlas3} that $\bk{k}_i  \le p/(p-q)$ for all $1 \le k \le d$ and $1 \le i \le k$. Therefore, recalling the definition of $\mathcal{S}(b,\cdot)$ from \eqref{startset}, for any $x \in \mathbb{R}^d_+$,
$$
x \in \mathcal{S}(b, p\supnorm{x}/(2p-1)).
$$
}
\textup{Finally we note Assumption \ref{assump:bounded_row} does not hold here. It can be checked from \eqref{eqn:skew_atlas2} that $\sum_{j = 1}^d (\rinv)_{ij}$ grows linearly in $i$ and hence, the row sums of $\rinv$ are not uniformly bounded by a dimension-independent constant. This stands in contrast with Example \ref{ex:bc}.}

\textup{The above observations result in the next theorem, which follows directly from Theorem \ref{thm:main_fromx}. As in Remark \ref{rem1}, the following bounds imply dimension-free convergence rates in the sense of \eqref{eqn:dimension_free} and \eqref{locstat}.}
\begin{theorem}\label{thm:skew_atlas}
Suppose $X$ is the RBM for the asymmetric Atlas model. Then there exist constants $\bar{C}, \bar{C}_0, t_0' > 0$ depending on $p$ but not on $d$ such that for $d > t_0'$,
\begin{multline}
    \Expect{\|X(x, t) - X(X(\infty), t)\|_{1,\sqrt{\frac{1-p}{p}}}}\\
    \le
\begin{cases}
\bar{C}\left(\sqrt{1+t^{1/4}} + \supnorm{x} e^{p\supnorm{x}/(4p-2)}\right)\> e^{-\bar{C}_0 t^{1/4}},
  &  t_0' \le t < d^4,\\\\
\bar{C}\left(\sqrt{1+t^{1/4}} + \supnorm{x} e^{p\supnorm{x}/(4p-2)}\right)\> e^{-\bar{C}_0 t/d^3},   & t \ge d^4.
\end{cases}
\end{multline}
\end{theorem}
\end{example}

\begin{example}[Blanchet-Chen type conditions]\label{ex:bc}
\textup{Here we consider $\RBM(\Sigma, \mu, R)$ with the system parameters satisfying certain `uniformity' assumptions in dimension similar to those of \cite{blanchet-chen}. In addition, we assume $P$ is a `band matrix' (see assumption a) below).}

\textup{With the notation of Assumption \ref{assump:main}: Suppose there exist $d$-independent constants $b_0, \sigmalow, \sigmahigh > 0$, $j_0 \in \{1, \ldots, d\}$, $k_0 \in \{ 2, \ldots, d\}$ and $\alpha' \in (0, 1)$, such that
\begin{itemize}
    \item[a)] $P_{ij} = 0$ \ for all $1 \le i, j \le d $ such that $|j-i| > j_0$.
    \item[b)] $\sum_{i = 1}^d P_{ij} \le \alpha '$ \ for all $1 \le j \le d$.
    \item[c)]  $\blowk{k} \ge b_0 $ \ for $k_0 \le k \le d$.
    \item[d)] $\sigma_i \in [\sigmalow, \sigmahigh]$ \ for $1 \le i \le d$.
\end{itemize}
We check that these conditions imply Assumption \ref{assump:bounded_row} with $r^* = 0$. Recall that the only difference between Assumption \ref{assump:main} and Assumption \ref{assump:bounded_row} is II in the former and II' in the latter. Note c) and d) immediately imply III, IV of Assumption \ref{assump:main} with $r^* = 0$ and $b_0, \sigmalow, \sigmahigh$ as above.}

\textup{Condition b) and induction imply
\begin{equation}\label{eqn:bc1}
    \max_{1 \le i, j \le d} P^n_{ij} \le \max_{1 \le j \le d}\>\sum_{i = 1}^d P^n_{ij} \le \left(\max_{1 \le i, l \le d} P^{n-1}_{il}\right)\max_{1 \le j \le d}\sum_{l = 1}^d P_{lj}  \le (\alpha ')^n, \quad \quad n \ge 1.
\end{equation}
Therefore, since $\rinv = \sum_0^\infty \left(P^T \right)^n$, condition II' holds with $M = 1/(1-\alpha')$. It remains only to show I of Assumption \ref{assump:main}. To simplify the proof we suppose $j_0 = 1$; the general case is similar. Consider $i, j$ such that $j > i$. Then, by part (a) of the above assumptions, $P_{ji}^{n} = 0$ for $n < j-i$. 
This fact and \eqref{eqn:bc1} give
\begin{equation}\label{eqn:bc2}
    (\rinv)_{ij} = \sum_{n=0}^\infty (P^T)^n_{ij} = \sum_{n=j-i}^\infty P^n_{ji} \le \sum_{n=j-i}^\infty \left(\alpha' \right)^n = \frac{(\alpha')^{j-i}}{1-\alpha'}.
\end{equation}
This proves I of Assumption \ref{assump:main} with $\alpha = \alpha'$ and $C = 1/(1-\alpha')$. The case where $j_0 > 1$ is proven similarly, with $\alpha = \left(\alpha'\right)^{1/j_0}$ and $C$ being a dimension-independent multiple of $1/(1-\alpha')$. Applying these facts to Theorem \ref{thm:main_fromx} in the case of Assumption \ref{assump:bounded_row} with $r^* = 0$ gives the following theorem.}

\begin{theorem}
Suppose $X$ satisfies a) to d) of Example \ref{ex:bc} and recall $\mathcal{S}(b,\cdot)$ from \eqref{startset}. Then there exist constants $\bar{C}, \bar{C}_0, t_0' > 0$ not depending on $d$ such that for any $B \in (0, \infty)$, $x \in \mathcal{S}(b,B)$ and $d > t_0'$,
\begin{align}
    \Expect{\|X(x, t) - X(X(\infty), t)\|_{1,\left(\alpha'\right)^{1/2j_0}}}
    \le
 \bar{C}\left(\sqrt{1+t} + \supnorm{x} e^{B/\sigmalow^2}\right)\>  e^{-\bar{C}_0 \frac{t}{\log (t \wedge d)}}, \nonumber \\
 &\qquad  t  \ge  t_0'.
\end{align}
\end{theorem}
\end{example}

\begin{remark}
A natural question in the above models is whether, for fixed $k \in \mathbb{N}$, our methods give dimension-free convergence for \emph{any set of $k$ of the $d$ coordinates} of $X$. For example, does dimension-free convergence hold for $(X_{d-k+1}(\cdot),\dots, X_d(\cdot))$ as $d$ ($\ge k$) grows? The answer is no in general. To see this, observe that in the Asymmetric Atlas model with $p>1/2$ (Example \ref{ex:skew_atlas}), the associated killed Markov chain starting from any $j \in\{1,\dots,d\}$ has a constant positive drift. Thus, although the expected number of visits to any $i<j$ decays like $\alpha^{j-i}$ for some $\alpha \in (0,1)$, the expected time spent at any $i>j$ is bounded below by a positive constant that is independent of $d, i-j$. Consequently, part I of Assumption \ref{assump:main} does not apply when analyzing the last $k$ coordinates. It is interesting to specify which (possibly $d$ dependent) subsets of $k$ coordinates exhibit dimension-free convergence and which do not. We leave this for future research.
\end{remark}

\section{Perturbations from stationarity for the Symmetric Atlas Model}\label{sec:atlas_perturbation}
This section is dedicated to the study of dimension free convergence for the Symmetric Atlas model, namely the model defined in \eqref{ascol} with $p=1/2$. We view this model as a first step to explore cases in which Assumption \ref{assump:main} fails to hold. As opposed to stretched exponential convergence rates obtained in Section \ref{sec:dfsec}, we obtain dimension-free convergence rates to stationarity for the process at a \emph{polynomial rate} if started from appropriate perturbations from stationarity.   

Recall that the gap process $X$ of the Symmetric Atlas model has the law of $\RBM(\Sigma, \mu, R)$ where $\mu = -(1, 0 \ldots, 0)$, $R = I - P^T$ and $\Sigma = 2R$ for
\begin{equation}\label{eqn:atlas1}
    P_{ij} = \begin{cases}
    1/2 \quad \quad j = i+1,\\
    1/2 \quad \quad j = i-1, \\
    0 \quad \quad \text{otherwise}.
    \end{cases}
\end{equation}
$\rinv$ is given by computation (e.g. \cite[Proof of Theorem 4]{banerjeebudhiraja}, or by taking $p \to 1/2$ in \eqref{eqn:skew_atlas2}:
\begin{equation}\label{eqn:atlas2}
    (\rinv)_{ij} = \begin{cases}
    2i\left(1 - \frac{j}{d+1}\right)\quad \quad 1 \le i \le j \le d,\\
    2j\left(1 - \frac{i}{d+1}\right) \quad \quad 1 \le j < i \le d.
    \end{cases}
\end{equation}
The above representation shows that $\rinv$ violates I, II of Assumption \ref{assump:main}, for example by considering $i = j = \floor{d/2}$. Nonetheless, $b = -\rinv \mu = \{(\rinv)_{i1} \}_{i=1}^d > 0$ and $\Sigma_{ii} = 2$ for all $i$. Therefore, there exists a stationary distribution. In fact, if $X(\infty)$ denotes the corresponding stationary distributed random variable, it holds that \cite{harrisonwilliams_expo,ichiba}
\begin{equation}\label{eqn:atlas_stationary}
    X(\infty) \sim \bigotimes_{i = 1}^d \text{Exp}\left(2\left(1 - \frac{i}{d+1}\right)\right).
\end{equation}
\subsection{Main result and applications}
Though Theorem \ref{thm:main_fromx} does not hold, we employ different methods to obtain dimension-free convergence rates to stationarity from initial conditions that perturb the stationary distribution by random variables in a 'perturbation class', which we now define. We direct the reader to Corollary \ref{cor:expo_perturbation} and Example \ref{ex:constant_perturbation} for concrete examples in this class of random variables.
\begin{definition}[Perturbation Class]\label{def:perturbation_class}
For $P_1,P_2,\delta \in (0,\infty)$, let $\mathcal{P}(P_1,P_2,\delta)$ denote the class of $\Reals^\infty$-valued random vectors $Y = (Y_1, Y_2,\dots)$ satisfying the following:
\begin{itemize}
\item[(i)] 
$
\Expect{\|Y\|_1^2} \le P_1.
$
\item[(ii)] 
$
\sup_{m \in \mathbb{N}}\Expect{\exp\left\lbrace \delta m^{-2}\|Y|_m\|_{\infty}\right\rbrace}  \le P_2.
$
\end{itemize}
\end{definition}
We will consider synchronously coupled processes, one starting from stationarity and the other starting from a perturbation of this stationary configuration by a random vector in $\mathcal{P}(P_1,P_2,\delta)$ for some $P_1,P_2,\delta \in (0,\infty)$.
Define for $Y \in \mathcal{P}(P_1,P_2,\delta)$ 
\begin{equation}\label{perdec}
\alpha^Y(n) := \Expect{\sum_{i=n+1}^{\infty} |Y_i|}, \ \ n \in \mathbb{N}.
\end{equation}
By assumption (i) above on the class $\mathcal{P}(P_1,P_2,\delta)$, note that for any $Y \in \mathcal{P}(P_1,P_2,\delta)$, $\alpha^Y(n) \rightarrow 0$ as $n \rightarrow \infty$.

\begin{theorem}\label{thm:atlas_perturbation} Fix any $P_1,P_2,\delta \in (0,\infty)$ and $Y \in \mathcal{P}(P_1,P_2,\delta)$. Let $X(\infty)$ be distributed as in \eqref{eqn:atlas_stationary} and define $X^{Y}(\infty) := \left(X(\infty) + Y|_d\right)_+$. 

Then there exist constants $t_0,t_0'', C_0, C_1 \in (0, \infty)$ not depending on $P_1,P_2,\delta$ such that for any $d \ge 1$ and any $n: \mathbb{R}_+ \rightarrow \mathbb{N}$ satisfying $\alpha^Y(n(t)) \rightarrow 0$ and $t^{-3/32}n(t) \rightarrow 0$ as $t \rightarrow \infty$,
\begin{multline}\label{eqn:atlas_perturbation_result1}
  \Expect{\|X(X^Y(\infty), t) - X(X(\infty), t) \|_1}\\
    \le
\begin{cases}
{\scriptstyle C_1\sqrt{P_1} n(t)t^{-3/32} + C_1\sqrt{P_1}\left(1+P_2^{1/4}\right)\exp\left\lbrace -C_0 \frac{\delta}{\delta + 4} t^{3/16} \right\rbrace + \alpha^Y(n(t))},
    & {\scriptstyle t^{(n)}_0 \le t < d^{16/3}},\\\\
{\scriptstyle C_1 \sqrt{P_2(d^2 + P_1)}\exp \left(-C_0 \frac{t}{d^6\log (2d)} \right)}, 
    & {\scriptstyle t\ge t_0'' \ d^4\log (2d)},
\end{cases}
\end{multline}
where $t^{(n)}_0 := \inf\{t \ge t_0 : t^{3/16} \ge 1 + 2n(t)\}$.
\end{theorem}

\begin{remark}
Note that the bounds in Theorem \ref{thm:atlas_perturbation} show polynomial decay when $t < d^{16/3}$ and exponential decay for $t > d^6 \log (2d)$. In particular, we do not obtain the `smooth patching' of the bounds as in the results of Section \ref{sec:dfsec}. This is mainly because the methods used for the two regimes $t < d^{16/3}$ and $t > d^6 \log (2d)$ in Theorem \ref{thm:atlas_perturbation} are starkly different. The `contractions' in $\|\cdot\|_{1,\beta}$ distance between the coupled RBMs upon certain events taking place in their trajectory, which was key to the results in Section \ref{sec:dfsec}, no longer holds here due to Assumption \ref{assump:main} not being satisfied. This is the main factor behind the discontinuous qualitative and quantitative transitions between the bounds in the two regimes in Theorem \ref{thm:atlas_perturbation}. See also Remark \ref{powerrem}.
\end{remark}

The choice of $n(\cdot)$ in Theorem \ref{thm:atlas_perturbation} has been intentionally kept flexible. One can choose $n(\cdot)$ in an `optimal' way so as to minimize $\max\{n(t)t^{-3/32},\alpha^Y(n(t))\}$. This, in turn, is intricately tied to the distributional behavior of the perturbation vector $Y$ as quantified by the function $\alpha^Y(\cdot)$. We mention the following two special cases as corollaries and choose $n(\cdot)$ in a case-specific way.

For perturbations from stationarity by finitely many coordinates in the following sense, one can take $n(\cdot)$ to be the (fixed) number of perturbed coordinates to obtain the following simplified bound.

\begin{cor}[Finite perturbations from stationarity]\label{cor:finite_perturbation}
Fix an integer $m \ge 1$ and a random vector $Z \in \Reals^m$ such that its extension to $\Reals^\infty$ given by $Y = (Z, 0, \ldots )$ is in the class $\mathcal{P}(P_1,P_2,\delta)$ of Definition \ref{def:perturbation_class} for some $P_1,P_2,\delta \in (0, \infty)$. Setting $n(t) = m$ for all $t$, we have for all $d > 1+2m$,
\begin{multline*}
  \Expect{\|X(X^Y(\infty), t) - X(X(\infty), t) \|_1}\\
    \le
\begin{cases}
{\scriptstyle C_1\sqrt{P_1} mt^{-3/32} + C_1\sqrt{P_1}\left(1+P_2^{1/4}\right)\exp\left\lbrace -C_0 \frac{\delta}{\delta + 4} t^{3/16} \right\rbrace},
    & {\scriptstyle t_0\vee(1 + 2m)^{16/3} \le t < d^{16/3}},\\\\
{\scriptstyle C_1 \sqrt{P_2(d^2 + P_1)}\exp \left(-C_0 \frac{t}{d^6\log (2d)} \right)}, 
    & {\scriptstyle  t\ge t_0'' \ d^4\log (2d)}.
\end{cases}
\end{multline*}
\end{cor}
The following corollary addresses the special case of perturbations from stationarity by independent exponential random variables.
\begin{cor}[Independent exponential perturbations]\label{cor:expo_perturbation}
Consider $Y = (Y_1, Y_2, \ldots)$ where $\{Y_i\}_{i \ge 1}$ are independent random variables with $Y_i \sim Exp(i^{1+\beta})$ (exponential with mean $i^{-(1+\beta)}$), for some $\beta > 0$. Then $Y \in \mathcal{P}(P_1,P_2,\delta)$ with $P_1 := \sum_{1}^\infty i^{-2(1+\beta)} + \left(\sum_{1}^\infty i^{-(1+\beta)}\right)^2$, $P_2 := 1 + \sum_{1}^\infty i^{-(1+\beta)}$ and $\delta:=1/2$. Setting $n(t) = \floor{t^{\frac{3}{32(1+\beta)}}}$, we have
\begin{multline*}
  \Expect{\|X(X^Y(\infty), t) - X(X(\infty), t) \|_1}\\
    \le
\begin{cases}
\left(C_1\sqrt{P_1} + \frac{2}{\beta} \right)t^{-\frac{\beta}{1 + \beta} \frac{3}{32}} + C_1\sqrt{P_1}\left(1+P_2^{1/4}\right)\exp\left\lbrace - \frac{C_0}{9} t^{3/16} \right\rbrace,
    & t_0' \le t < d^{16/3},\\\\
C_1 \sqrt{P_2(d^2 + P_1)}\exp \left(-C_0 \frac{t}{d^6\log (2d)} \right), 
    & t\ge t_0'' \ d^4\log (2d),
\end{cases}
\end{multline*}
where $t_0' \in (0, \infty)$ does not depend on $d$ or $\beta$.
\end{cor}
The proof of this corollary makes clear one could consider independent $Y_i \sim Exp(\lambda_i)$ for any sequence $\{ \lambda_i \}_{i \ge 1}$ such that $\|Y\|_1$ has finite expectation and variance. We choose $\lambda_i = i^{1+\beta}$ as it lends itself to simple and explicit calculations of the rates of convergence.
\begin{proof}[Proof of Corollary \ref{cor:expo_perturbation}] $Y \in \mathcal{P}(P_1,P_2,\delta)$ is the result of the following calculations:
\begin{align*}
\Expect{\|Y\|_1^2} &= 
\operatorname{Var}(\|Y\|_1) + \left(\Expect{\|Y\|_1}\right)^2 
= \sum_{i=1}^\infty i^{-2(1+\beta)} + \left(\sum_{i=1}^\infty i^{-(1+\beta)}\right)^2,\\
    \Expect{\exp \left\lbrace\frac{\supnorm{Y|_m}}{2m^2} \right\rbrace} 
    &\le 1 + m^{-2}\sum_{i=1}^m i^{-(1+\beta)} 
    \le 1 + \sum_{i=1}^\infty i^{-(1+\beta)}, \quad \text{for all} \ m \in \mathbb{N}.
\end{align*}
With $n(t) = \floor{t^{\frac{3}{32(1+\beta)}}}$ we have by basic calculus that $\alpha^Y(n(t)) \le \frac{2}{\beta} t^{-\frac{\beta}{1 + \beta} \frac{3}{32}}$ and $n(t)t^{-\frac{3}{32}} \le t^{-\frac{\beta}{1 + \beta} \frac{3}{32}}$ for $t \ge 2$. Applying Theorem \ref{thm:atlas_perturbation} gives the corollary.
\end{proof}
We close the series of applications with the most basic example, in which the perturbation $Y$ is a constant.
\begin{example}[Constant perturbations]\label{ex:constant_perturbation}
    Consider $Y = (Y_1, Y_2, \ldots)$ such that $Y$ is a constant vector satisfying $\|Y\|_1 < \infty$, which implies $\|Y\|_\infty < \infty$. Choose $n(t)$ to be any function such that $n(t) \ge 1$, $t \mapsto n(t)$ is non-decreasing for $t \ge 0$, and $n(t) t^{-3/32} \to 0$ as $t \to \infty$.  Then Theorem \ref{thm:atlas_perturbation} holds for any such $n(t)$ if we set $\delta = 1$, $P_1 = \|Y\|_1^2$ and $P_2 = \exp\left\{\|Y\|_\infty\right\}$. The rate of convergence then is determined by the function $t \mapsto \max\{n(t)t^{-3/32},\alpha^Y(n(t))\}$.

    In particular, Corollary \ref{cor:finite_perturbation} holds when for some $m \ge 1$ we have $Y_i = 0$ for $i \ge m+1$. 
\end{example}
\begin{remark}\label{powerrem}
In Theorem \ref{thm:atlas_perturbation} and Corollaries \ref{cor:finite_perturbation} and \ref{cor:expo_perturbation}, the upper bound has a polynomial decay in $t$ for large $d$ (for $t < d^{16/3}$) as opposed to the stretched exponential decay observed in Section \ref{sec:dfsec} when Assumption \ref{assump:main} applies. Although we do not currently have associated lower bounds, we strongly believe that the $L^1$-Wasserstein distance of the perturbed system (as defined in Theorem \ref{thm:atlas_perturbation}) from stationarity indeed shows polynomial decay for the Symmetric Atlas model. This belief stems from the dynamics of the associated killed Markov chain whose transition kernel is prescribed by $P$ (see discussion after Assumption \ref{assump:main}) which are shown throughout this article to govern convergence rates to stationarity. This Markov chain for the Symmetric Atlas model behaves as a simple random walk away from the cemetery state and thus lacks the `strong drift' towards the cemetery state characteristic of the models considered in Section \ref{sec:dfsec}. This results in the slower convergence rates.

The polynomial rates of convergence to stationarity obtained in \cite{banerjee2020rates} for the Potlatch process on $\mathbb{Z}^k$, which (for $k=1$) can be loosely thought of as a `Poissonian version' of the gap process of the infinite Symmetric Atlas model constructed in \cite{pitman_pal}, lends further evidence to this belief.  
\end{remark}
\subsection{A pathwise derivative approach towards convergence rates} \label{pathrwre}
The proof of Theorem \ref{thm:atlas_perturbation} is based on an analysis of the derivative process (derivative taken with respect to initial conditions) of the RBM $X$. The key observation made here is a representation of this derivative process in terms of a \emph{random walk in a certain random environment} constructed from the random order in which the RBM hits distinct faces of the orthant $\mathbb{R}^d_+$ (see \eqref{eqn:rwre_deriv}). This representation, in turn, is based on a succinct form  for the derivative process obtained in \cite[Theorem 1.2]{andres2009diffusion}. This is summarized in Theorem \ref{thm:derivative_process} below. This representation is interesting in its own right and we believe a systematic study of the derivative process is at the heart of obtaining convergence rates in more general cases where Assumption \ref{assump:main} does not hold. Moreover, as the relationship between the derivative process and the (random) transition kernel of the random walk in the random environment is an exact equality \eqref{eqn:rwre_deriv}, this representation should also lead to \emph{lower bounds for convergence rates}. We hope to report on this in future work. 

In the probability literature, random walks in random environments most commonly appear as random walks on graphs with jump probabilities given by i.i.d. random variables (see e.g. \cite{sznitman} or \cite{dembo_rwre} for a model with i.i.d. holding times). Since the process we will consider in Theorem \ref{thm:derivative_process} is substantially different, we take some care first to define it.
\begin{definition}[$RW(\mathbf{a},i_0)$]\label{def:rwre} Here we define a random walk on $\{0, \ldots, d+1 \}$, for $d\ge 1$, in a given fixed environment $\mathbf{a}$ and initial condition $i_0$. Call any sequence $\mathbf{a} := (l_k,t_k)_{k \ge 0}$ \emph{admissible} if 
\begin{itemize}
    \item[(i.)] $(l_k,t_k) \in \{1,\dots,d\} \times [0,\infty)$ for all $k \ge 0$,
    \item[(ii.)] $t_0 = 0$ and $\{t_k\}_{k \ge 0}$ is strictly increasing.
\end{itemize}
For any admissible sequence $\mathbf{a}$ and any $i \in \{1,\dots,d\}$, define the projected admissible sequence $\mathbf{a}_i = (l_k^i,t_k^i)_{k \ge 0} = (i,t_k^i)_{k \ge 0}$ to be the unique admissible sequence obtained from the elements of the set $\mathbf{a} \cap \left(\{i\} \times [0,\infty)\right)$. In words, this sequence consists of points in $\mathbf{a}$ with first coordinate equal to $i$ enumerated in ascending order of their second coordinates.

Define the random walk in environment $\mathbf{a}$ started from $i_0 \in \{0,\dots,d+1\}$, written as $RW(\mathbf{a},i_0)$, to be the time-inhomogeneous Markov process $W$ with state space $\{0,\dots,d+1\}$ whose law is uniquely characterized by the following:
\begin{itemize}
    \item[(i.)] $W(0) = i_0$,
    \item[(ii.)] $W$ is absorbed at $0$ and $d+1$,
    \item[(iii.)] Define the `jump times' $\{T_k\}_{k\ge 0} = \{T_k(\mathbf{a}, i_0)\}_{k\ge 1}$ as follows: $T_0 = 0, T_1 = t_1^{i_0}$ and
    $$T_{k+1} = \min\left\{t_j^i \> : \> i = W(T_k), \> t_j^i > T_k, \> (i, t_j^i) \in \mathbf{a} \right\}, \quad \quad k \ge 1.$$
   The transition probabilities of $W$ at the jump times are then given by
    \begin{align*}
        1/2 &= \mathbb{P}_{\mathbf{a},i_0}\left(W(T_{k+1}) = W(T_k) + 1 \mid \left(W(T_k), T_k \right)\right) \\
        &= \mathbb{P}_{\mathbf{a},i_0}\left(W(T_{k+1}) = W(T_k) - 1 \mid \left(W(T_k), T_k \right)\right).
    \end{align*}
    \item[(iv.)] $\mathbb{P}_{\mathbf{a},i_0}\left(W(t) = W(T_k) \mid \left(W(T_k), T_k \right)\right) = 1$ for $t \in [T_k, T_{k+1})$,  $k\ge 0$,
    \item[(v.)] for $0 \le t < t'$, 
    $$\mathbb{P}_{\mathbf{a},i_0}\left(W(t') = W(t) \mid W(t) = 0\right) = \mathbb{P}_{\mathbf{a},i_0}\left(W(t') = W(t) \mid W(t) = d+1\right) = 1.$$
\end{itemize}
In the above, we used the suffix in the probabilities to highlight the dependence of the law of $W$ on $\mathbf{a}$ and $i_0$. The process $W$ can be seen as a simple random walk absorbed at $0, d+1$ with jump times prescribed by the points in $\mathbf{a}$ encountered along its trajectory.

Finally, define
$$
J_{\mathbf{a},i_0}(t) := \#\left\{s \in [0,t] \>:\> W(s-) \neq W(s)\right\} = \#\left\{k \ge 1 : T_k \in [0, t]\right\},
$$
to be the number of jumps made by $RW(\mathbf{a},i_0)$ in the time interval $[0,t]$.
\end{definition}

We now define a few additional conventions and notations required to state the theorem. For two vectors $x, y \in \Reals^d$ we write $\langle x, y \rangle$ for the standard inner product, and $e_i, 1 \le i \le d$ for the standard basis vectors. For a $d\times d$ matrix $R$, write $R^{(i)}$ for the $i$-th column vector of $R$.

For $X$ started at $x \in \Reals_+^d$, $x > 0$, define a sequence of stopping times as follows: $\tau_0(x) = 0, \tau_1(x) = \inf\left\{t > 0 \> | \> X_i(x,t) = 0 \text{  for some  } i \right\}$ and for $k \ge 1$,
\begin{equation}\label{eqn:tauk}
    \tau_{k+1}(x) = \inf\{t > \tau_{k}(x) \> | \> X_i(x,t) = 0, \ X_j(x,\tau_{k}) = 0 \text{  for some } \>i,j \text{  such that  } j\neq i \}.
\end{equation}
Also define the sequence of integers $i_k(x)$ for $k \ge 0$ as follows: Fix any $i_0(x) \in \{1, \ldots, d \}$ and define the remaining $i_k(x)$ by $X_{i_k(x)}(x,\tau_k(x)) = 0$, i.e. $i_k(x)$ is the index of the coordinate hitting zero at time $\tau_k(x)$ for $k \ge 1$. In other words, $\{\tau_k(x)\}_{k \ge 1}$ represent the times when $X$ has crossed from one face of the orthant to another, and $i_k(x)$ tells which coordinate has hit zero at crossing time $\tau_k(x)$. We suppress dependence of $\tau_k, i_k$ on $x$ when there is no risk of confusion.

From \cite[Theorem 1.9]{sarantsev2015triple}, the Atlas model almost surely has no simultaneous collisions, which in the context of this paper means two different coordinates of $X$ do not hit $0$ at the same time. Thus, almost surely, for any $x \in \mathbb{R}^d_+, t > 0$, $X_i(x,t) = 0$ for at most one $i \in \{1,\dots,d\}$. Therefore, $i_k, \tau_k$ are well-defined and the sequence $\left\{(i_k, \tau_k)\right\}_{k \ge 0}$ is admissible in the sense of Definition \ref{def:rwre}. This fact is essential for the random walk representation below.

The following theorem gives a representation \eqref{andrep} of the derivative process of the RBM $X$, which is a specialization of \cite[Theorem 1.2]{andres2009diffusion} to the present context. This representation is then encoded in terms of a random walk in an admissible environment constructed from hitting times of faces of $\mathbb{R}^d_+$ by the RBM $X$. This connection is the main message of the theorem, and is key to proving Theorem \ref{thm:atlas_perturbation}.

\begin{theorem}\label{thm:derivative_process} For every $t \in [0,\infty)$ and every $x>0$, the map $y \mapsto X(y,t)$ is almost surely differentiable at $x$. For each $i_0 \in \{1, \ldots, d\}$ the process $$\eta^{i_0}(x, t) := \lim_{\varepsilon \to 0} \varepsilon^{-1} \left(X(x +  \varepsilon e_{i_0}, t) - X(x, t) \right)$$ has a right-continuous modification defined on $\left[0, \infty\right)$ such that
\begin{equation}\label{andrep}
\eta^{i_0}(x, t) = S_k^{i_0}(x), \quad \text{ for } t\in\left[\tau_k, \tau_{k+1}\right), \quad k \ge 0,
\end{equation}
where $\left\{S^{i_0}_k(x)\right\}_{k \ge 0}$ is a sequence of d-dimensional random vectors iteratively defined by
$$
\begin{cases} S_0^{i_0}(x) = e_{i_0} \\
    S_{k+1}^{i_0}(x) = S_k^{i_0}(x) - \langle S_k^{i_0}(x), e_{i_{k+1}} \rangle R^{(i_{k+1})}, \quad k \ge 0.
    \end{cases}
$$
Moreover, $\Theta(x) := \left\{ \left(i_k,\tau_k \right)\right\}_{k\ge0}$ is admissible and the derivative process has the following representation in terms of the law of $RW(\Theta(x), i_0)$:
\begin{equation}\label{eqn:rwre_deriv}
    \eta^{i_0}_j(t, x) = \mathbb{P}_{\Theta(x), i_0}(W(t) = j), \ j=1,\dots,d.
\end{equation}
\end{theorem}

We illustrate in Figure \ref{RBM_RW} the connection between the paths of the RBM $X$ and the random walk $W$ when $d=2$. In the figure, $i_1 = 1, i_2 = 2$ and $i_3 = 1$ corresponding to the index of the coordinates at times $\tau_i, i = 1, 2, 3$ when $X$ crosses faces of the orthant. The corresponding walk $W$, which begins at state $2$, does not jump at time $\tau_1$ because $W(\tau_1-) \neq i_1 = 1$. $W$ does jump at time $\tau_2$ since $W(\tau_2-) = i_2 = 2$ and thus $\tau_2$ is equal to the first jump time $T_1$. 
\begin{figure}
    \includegraphics[scale=2]{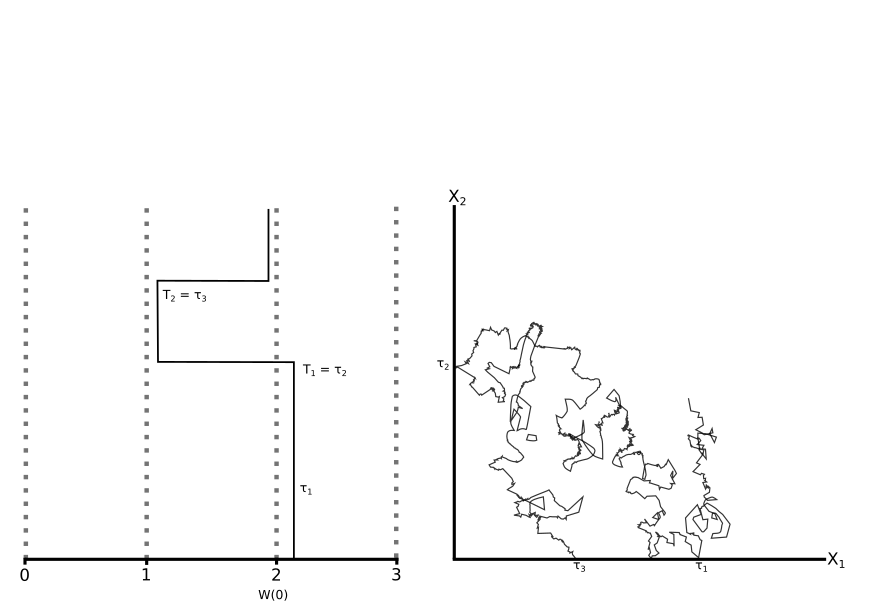}
    \caption{Illustration of the connection between RBM $X$ and the random walk $W$ for $d=2$.}
    \label{RBM_RW}
\end{figure}
\begin{remark}\label{rem:rwre_jumptimes}
We clarify the relationship between boundary-hitting times of the process $X$ started at $x>0$ and the jump times of $W \sim RW(\Theta(x),i_0)$, $i_0 \in \{1,\dots,d\}$.

Suppose $X$ begins at $x > 0$ and $W(t) = i \in \{1,\dots,d\}$ at some time $t \ge 0$. Then at the first time after $t$ that $X_i$ hits zero, $W$ will jump to $i-1$ or $i+1$ with equal probability.

Now suppose for a given time interval $[0, T]$ and integer $m \ge 2$ the random walk $W$ starting from $i_0$ remains in the set $\{1, \ldots, m-1 \} $. Suppose also that there are random times $0 = \eta^0_m <\eta^1_m < \eta^2_m< \dots < \eta^N_m < T$, with $\eta^{j+1}_m - \eta^{j}_m >1$ for each $j \in \{0,\ldots,N-1\}$, such that $X$ has hit each of the first $m$ coordinates in every interval $(\eta^j_m + 1, \eta^{j+1}_m], \, j \in \{0,\ldots,N-1\}$. Then the walk has made at least $N$ jumps in the time interval $[0, T]$. In particular, with $\N_{m}(x,T)$ defined as in \eqref{eqn:boundary_hit2} below,
\begin{multline}\label{eqn:rwre_jumptimes1}
   \left\{ \N_{m}(x, T) \ge N, \> W(s) \in \{1, \ldots, m-1\} \quad \text{for} \quad s\in [0, T]\right\} \\ \subseteq \left\{ J_{\Theta(x),i_0}(T) \ge N, \> W(s) \in \{1, \ldots, m-1\} \quad \text{for} \quad s\in [0, T]\right\}.
\end{multline}
This fact will be crucially used in the proof of Theorem \ref{thm:atlas_perturbation}.

We also note here that the process $W$ is non-standard in the sense that the number of jumps of $W$ in a certain time interval depends on the whole trajectory of $W$ in that interval, which makes its analysis challenging.
\end{remark}
\begin{remark}\label{rem:atlas_deriv_remark}We have stated Theorem \ref{thm:derivative_process} for the Symmetric Atlas model examined here, but an analogous result holds for any RBM \eqref{eqn:rbm} that almost surely does not hit intersections of faces (corners) of the orthant $\mathbb{R}^d_+$. In that case one-step transitions are given by the matrix $P$ (from $R = I - P^T$). See \cite{karatzas_skewatlas} for conditions guaranteeing when the gap process of an Atlas model (symmetric or asymmetric) does not hit corners, and \cite{sarantsev2015triple} for similar conditions for a general RBM.

For the general RBM \eqref{eqn:rbm}, even when corners are hit with positive probability, \cite{ramanan_mandelbaum} shows that the derivative process exists in an appropriate sense. However, in the general case we do not have a random walk representation as in Theorem \ref{thm:derivative_process}. \cite{blanchent2020efficient} has recently obtained an upper bound for the derivative process in terms of products of random matrices derived in terms of the boundary hitting times and locations of the RBM and the killed Markov process associated with $P$ (see \cite[Lemma 5]{blanchent2020efficient}). This presents an opportunity to generalize the methods used here, and we defer it to future work.
\end{remark}
The following corollary to Theorem \ref{thm:derivative_process} is the key tool in proving Theorem \ref{thm:atlas_perturbation}. 
\begin{cor}\label{cor:$L^1$_dist} Fix $x, \tilde{x} \in \Reals_+^d$ with $x > 0$ and let $\gamma(u) = x + u(\tilde{x} - x)$ for $u \in [0, 1]$. Then, writing $\tau^*_0 := \inf\{s \ge 0 : W(s) = 0\}$,
\begin{equation}\label{eqn:rwre_ineq_m}
    \|X(\tilde{x}, t) - X(x, t)\|_1 \le \sum_{i = 1}^d |(\tilde{x} - x)_i| \int_{[0, 1)}\mathbb{P}_{\Theta(\gamma(u)), i}(\tau^*_0 > t) \> du, \ \ t \ge 0.
\end{equation}
\end{cor}
\begin{proof}
For each $i = 1, \ldots, d$ and $t \ge 0$ define the function $f_{i, t}: [0, 1] \mapsto [0, \infty)$ as $f_{i, t}(u) = X_i(\gamma(u), t)$. As shown in the proof of \cite[Theorem 1]{harrisonreiman}, $x \mapsto X_i(x, t)$ is Lipschitz. Thus $f_{i, t}$ is absolutely continuous on $[0, 1]$ and we have for $t \ge 0$:
\begin{align*}
 \|X(\tilde{x}, t) - X(x, t)\|_1 &\le \sum_{j=1}^d\sum_{i=1}^d|(\tilde{x} - x)_i| \int_{[0, 1)}\mathbb{P}_{\Theta(\gamma(u)), i}(W(t)=j) \> du\\
 &= \sum_{i=1}^d|(\tilde{x} - x)_i| \int_{[0, 1)}\mathbb{P}_{\Theta(\gamma(u)), i}(W(t) \in \{1,\dots,d\}) \> du\\
 &\le \sum_{i = 1}^d |(\tilde{x} - x)_i| \int_{[0, 1)}\mathbb{P}_{\Theta(\gamma(u)), i}(\tau^*_0 > t) \> du.
\end{align*}
The first step above follows from absolute continuity and Theorem \ref{thm:derivative_process} for $\gamma(u) > 0$ for $u \in [0, 1)$. The second step follows by an interchange of summation.
\end{proof}

\section{Proofs: Dimension-free local convergence rates for RBM}\label{sec:dfproofs}
\subsection{Boundary-hitting times}
Before proceeding to the proofs, we define boundary hitting times for a solution $X$ to \eqref{eqn:rbm}, which we use throughout. For any $1 \le d' \le d$, we define a sequence of times between which $X$ hits $d'$ faces of $\Reals_+^d$ corresponding to $X_i = 0$ for $i = 1, \ldots, d'$. Set $\eta_{d'}^0(x) = 0$ and define inductively for $k \ge 1$ 
\begin{equation} \label{eqn:boundary_hit1}
    \xi_i^k(x) = \inf\{t > \eta_{d'}^{k-1}(x) + 1 \> | \> X_i(x, t) = 0\}, \quad \quad \eta_{d'}^k(x) = \max\{\xi_i^k(x)\> | \> i = 1,\ldots, d' \}
\end{equation} 
where we suppress the $d'$ dependence of $\xi_i^k$s for convenience. Also define
\begin{equation}\label{eqn:boundary_hit2}
    \N_{d'}(x, t) = \max \{k \> | \> \eta_{d'}^k(x) \le t\}.
\end{equation}
All the stopping times defined above are finite almost surely, which follows from the positive recurrence criterion $R^{-1}\mu<0$. It can also be deduced from Lemma \ref{lemma:boundary_hits} below.
\subsection{Fundamental properties of RBM}
The next two theorems record fundamental results related to this work from, respectively, \cite{kella} Theorem 1.1, and \cite{sarantsev} Theorem 3.1, Corollaries 3.5 and 3.6.

\begin{theorem}[Monotonicity under synchronous coupling]\label{thm:rbm_monotonicity}
For $X$ a solution to \eqref{eqn:rbm} and $x, \tilde{x} \in \Reals^d_+$ such that $x \ge \tilde{x}$, the following hold:
\begin{itemize}
    \item[(i)] $X(x, t) \ge X(\tilde{x}, t)$ for all $t > 0$.
    \item[(ii)]$t \mapsto L(x, t) - L(\tilde{x}, t)$ is non-positive, non-increasing and bounded below by  $- \rinv(x - \tilde{x})$.
    \item[(iii)] $t \mapsto \rinv \left(X(x, t) - X(\tilde{x}, t)\right) = \rinv(x - \tilde{x}) + L(x, t) - L(\tilde{x}, t)$ is non-negative and non-increasing.
\end{itemize}
\end{theorem}

\begin{theorem}[Stochastic domination of projected system]\label{thm:domination}
Suppose $X$ is a solution to \eqref{eqn:rbm} with parameters $(\Sigma, \mu, R)$ and corresponding local times $L$. For $x \in \Reals^d_+$ and an integer $1 \le k \le d$, define the process $Z(x|_k,t) := x|_k + \muk{k} t + (DB(t)) |_k, \ t \ge 0$, which uses the same driving Brownian motion $B$ as $X$. 
Define $\bar{X}$ to be the $\Reals^k_+$-valued process obtained as the solution to
$$
\bar{X}(x|_k, t) = Z(x|_k,t) + R|_k  \bar{L}(x|_k,t), \ t \ge 0,
$$
where $ \bar{L}(x|_k,\cdot)$ is the local time which constrains $\bar{X}$ to $\Reals^k_+$. Then $$ 
X|_k(x, t) \le \bar{X}(x|_k, t) \quad t \ge 0, \quad \quad L|_k(x, t) - L|_k(x, s) \ge \bar{L}(x|_k, t) - \bar{L}(x|_k, s) \quad 0 \le s \le t.
$$
\end{theorem}
\subsection{Proofs}
The following lemma provides a crucial local contraction estimate. It shows that for any $x \in \mathbb{R}^d_+$, the weighted distance between the coupled processes $X(x,\cdot)$ and $X(0,\cdot)$ as measured by $u(x, \cdot)$ in \eqref{eqn:weighted_norm} decreases by a constant factor if a subset of coordinates of $X(x,\cdot)$ (whose cardinality is determined by the initial distance) hit zero.
\begin{lemma}[Local contraction] \label{lemma:contraction}
Suppose I, II of Assumption \ref{assump:main} hold for $X$, an $\RBM(\Sigma, \mu, R)$. Fix an initial condition $X(x, 0) = x \ge 0$. With $\alpha$ as in Assumption \ref{assump:main}, fix $\beta \in (\alpha, 1)$ and $\delta \in (\beta, 1)$. Recall the weighted supremum norm $\bsupnorm{x}{\delta} = \max_{1 \le i \le d} \delta^i \> x_i$, and $u(x, \cdot)$ from \eqref{eqn:weighted_norm}. 

Fix $d' \in \{1, \ldots, d\}$. Recall the definition of $\eta_{d'}^1 = \eta_{d'}^1(x)$ from \eqref{eqn:boundary_hit1}.

There exist $C' > 0$ and $\lambda \in (1/2, 1)$ not dependent on $d$, $d'$ or $x$ such that, 
\begin{itemize}
    \item[(i)] if $1 \le d' \le d-1$,
    \begin{equation}\label{eqn:contraction_result}
    u(x, 0) \ge C'\bsupnorm{x}{\delta} \> (\beta/\delta)^{d'+1} \quad \Longrightarrow \quad u(x, \eta_{d'}^1) \le \lambda u(x, 0).
    \end{equation}
    \item[(ii)] if $d' = d$,
    \begin{equation}\label{eqn:contraction_result2}
    u(x, \eta_{d}^1) \le \lambda u(x, 0).
    \end{equation}
\end{itemize}
$C', \lambda$ may be chosen explicitly as functions of $\beta, \delta$ and the constants $\alpha, C, M$ from Assumption \ref{assump:main}.
\end{lemma}
\begin{proof}
Define the processes
\begin{eqnarray} \label{eqn:deltax}
\Delta X(t) &=& X(x, t) - X(0, t) \nonumber \\
\Delta L(t) &=& L(x, t) - L(0, t) \nonumber\\
Y(t) &=& \rinv \Delta X(t) = \rinv x + \Delta L(t)
\end{eqnarray}
From Theorem \ref{thm:rbm_monotonicity} we know that for all for all $t \ge 0$, $\Delta X(t) \ge 0$, $t \mapsto \Delta L(t)$ is non-positive, non-increasing and $t \mapsto Y(t)$ is non-negative, non-increasing. By definition, then, $t \mapsto u(x, t)$ is non-negative and non-increasing. We aim to show that $u$ indeed contracts by a fixed proportion $\lambda$ of its initial value at time $\eta^1_{d'}$.

The crucial fact is that if $X_i(x, \> \cdot)$ has hit zero before a time $t$, then $\Delta L_i(s) \le - x_i$ for all $s \ge t$. Indeed, setting $t_0>0$ to be the first hitting time of $X_i(x, \> \cdot)$ at $0$ and assuming $t_0 < t$,
\begin{align}\label{eqn:contraction1}
    0 = \Delta X_i(t_0) &= x_i + \left(R\Delta L(t_0)\right)_i \nonumber \\
    &= x_i + \Delta L_i(t_0) - \left(P^T\Delta L(t_0)\right)_i
    \ge x_i + \Delta L_i(t) \ge x_i + \Delta L_i(s),
\end{align}
for all $s \ge t$, where the first equality follows from $R= I - P^T$ and the last two inequalities follow from Theorem \ref{thm:rbm_monotonicity} (ii) and the non-negativity of $P$. By definition, at time $\eta^1_{d'} = \eta^1_{d'}(x)$ the first $d'$ coordinates of $X(x, \cdot)$ have already hit zero. \eqref{eqn:contraction1} then implies
\begin{align}\label{eqn:contraction2}
    u(x, \eta^1_{d'}) &= \sum_{i=1}^d \beta^i Y_i(\eta^1_{d'}) = u(x, 0) + \sum_{i=1}^d \beta^i \Delta L_i(\eta^1_{d'})\notag\\
    &\le u(x, 0) - \sum_{i=1}^{d'} \beta^i x_i + \ind{d' < d}\> \sum_{i=d'+1}^d \beta^i \Delta L_i(\eta^1_{d'})
    \le u(x, 0) - \sum_{i=1}^{d'} \beta^i x_i.
\end{align}
The last inequality follows once again from Theorem \ref{thm:rbm_monotonicity} (ii). To achieve the result \eqref{eqn:contraction_result}, we first bound $\sum_{i=d' + 1}^d \beta^i Y_i(0)$. In the following, the first inequality is a consequence of the definition of $\bsupnorm{x}{\delta}$ and the second inequality follows from I, II of Assumption \ref{assump:main}. Remaining statements follow from the fact that $\alpha < \beta < \delta < 1$. For $d' < d$,
\begin{align}\label{eqn:contraction3}
    \sum_{i=d' + 1}^d \beta^i Y_i(0) &= \sum_{i = d' + 1}^d \beta^i \sum_{j=1}^d (\rinv)_{ij }x_j \le \bsupnorm{x}{\delta} \> \sum_{i = d' + 1}^d \beta^i \sum_{j=1}^d (\rinv)_{ij }\delta^{-j}\notag\\
    &\le \bsupnorm{x}{\delta} \> \sum_{i = d' + 1}^d \beta^i \left( M \sum_{j=1}^i \delta^{-j} + C \sum_{j={i+1}}^d \alpha^{j-i} \delta^{-j}\right)\notag\\
    &\le \bsupnorm{x}{\delta} \> \sum_{i = d' + 1}^d (\beta/\delta)^i \left(  M \sum_{j=0}^{i-1} \delta^j + C\sum_{j={i+1}}^\infty (\alpha/\delta)^{j-i} \right)\notag\\
    &\le  \frac{\bsupnorm{x}{\delta} M}{1 - \delta} \sum_{i = d' + 1}^d (\beta/\delta)^i \nonumber \\
    &+  \frac{\bsupnorm{x}{\delta}\> C(\alpha/\delta)}{1 - \alpha/\delta}\sum_{i = d' + 1}^d (\beta/\delta)^i
    \le \tilde{C} \bsupnorm{x}{\delta} \> (\beta/\delta)^{d'+1},
\end{align}
with $\tilde{C} = \frac{M}{(1-\delta)(1-\beta/\delta)} + \frac{C(\alpha/\delta)}{(1 - \alpha/\delta)(1-\beta/\delta)}$, which by Assumption \ref{assump:main} does not depend on $d'$, $d$ or $x$.

Now recall that since $P$ is transient and $R = I - P^T$ we have $\rinv = \sum_{n=0}^\infty (P^T)^n$, which implies $Y(0) = \rinv x \ge x$. Using this and \eqref{eqn:contraction3}, we have for $1 \le d' \le d-1 $
\begin{equation}\label{eqn:contraction4}
    \sum_{i=1}^{d'} \beta^i x_i = \sum_{i=1}^{d} \beta^i x_i - \sum_{i=d'+1}^{d} \beta^i x_i \ge \sum_{i=1}^{d} \beta^i x_i - \sum_{i=d'+1}^{d} \beta^i Y_i(0) \ge \sum_{i=1}^{d}\beta^i x_i - \tilde{C}\bsupnorm{x}{\delta}\>(\beta/\delta)^{d'+1}.
\end{equation}
Furthermore, I, II of Assumption \ref{assump:main} and $1 > \beta > \alpha$ give 
\begin{align}\label{eqn:contraction5}
    u(x, 0) &= \sum_{i=1}^d \beta^i Y_i(0) = \sum_{j=1}^{d} \beta^j x_j \sum_{i = 1}^d (\rinv)_{ij} \beta^{i-j} \nonumber \\
    &\le \sum_{j=1}^{d} \beta^j x_j \left(C \sum_{i=1}^j (\alpha/\beta)^{j-i} + M \sum_{i=j+1}^d \beta^{i-j} \right) \nonumber \\
    &\le \left(\frac{C}{1 - \alpha/\beta} +\frac{M\beta}{1-\beta} \right)  \sum_{i=1}^{d}\beta^j x_j \le \tilde{C}' \sum_{i=1}^{d}\beta^j x_j,
\end{align}
where we have set $\tilde{C}' = 1\vee \left[C/(1 - \alpha/\beta) + M\beta/(1-\beta)\right]$. Combining \eqref{eqn:contraction4} and \eqref{eqn:contraction5},
\begin{equation}\label{eqn:contraction6}
    \sum_{i=1}^{d'} \beta^i x_i \ge \frac{1}{\tilde{C}'} u(x, 0) - \tilde{C}\bsupnorm{x}{\delta}\>(\beta/\delta)^{d'+1}.
\end{equation}
Finally, if $u(x, 0) \ge 2\tilde{C}'\tilde{C}\bsupnorm{x}{\delta}\>(\beta/\delta)^{d'+1}$ then \eqref{eqn:contraction6} gives
\begin{equation}\label{eqn:contraction7}
    \sum_{i=1}^{d'} \beta^i x_i \ge \frac{1}{2\tilde{C}'} u(x, 0).
\end{equation}
The result \eqref{eqn:contraction_result} now follows with $C' = 2\tilde{C}'\tilde{C}$ and $\lambda = 1-1/(2\tilde{C}')$ using \eqref{eqn:contraction7} and \eqref{eqn:contraction2}.
To prove \eqref{eqn:contraction_result2}, we use \eqref{eqn:contraction2} with $d' = d$ and \eqref{eqn:contraction5} as follows
\begin{equation}
    u(x, \eta^1_{d}) \le u(x, 0) - \sum_{i=1}^{d} \beta^i x_i \le \left(1-\frac{1}{\tilde{C}'}\right)u(x, 0) \le \lambda u(x, 0).
\end{equation}
\end{proof}

\begin{cor}\label{cor:contraction1}
Retain the assumptions of Lemma \ref{lemma:contraction} and recall $\beta, \delta$ chosen there. Recall the definition of $\N_{d'}(x, t)$ from \eqref{eqn:boundary_hit2}. Define the stopping times with $C'$ as in \eqref{eqn:contraction_result},
\begin{equation}
    \tau(x, d') := \inf \left\{s > 0 \,\,\,| \,\,\, u(x, s) \le C' \bsupnorm{x}{\delta}\> (\beta/\delta)^{d'+1} \right\}, \quad \quad \text{for }\> x \in \Reals^d_+, \ 1\le d' \le d-1.
\end{equation}
Then for any $q > 0$,
\begin{itemize}
    \item[(i)] if $1 \le d' \le d-1$,
    \begin{equation}\label{eqn:cor1_contraction_result}
    u(x, t) \> \ind{\tau(x, d') > t, \> \N_{d'}(x, t) \ge q} \le \lambda^{\floor{q}} u(x, 0).
    \end{equation}
    \item[(ii)] if $d' = d$,
    \begin{equation}\label{eqn:cor1_contraction_result2}
    u(x, t) \> \ind{\N_{d}(x, t) \ge q} \le \lambda^{\floor{q}} u(x, 0).
    \end{equation}
\end{itemize}
\end{cor}
\begin{proof}
First, by Theorem \ref{thm:rbm_monotonicity} (iii) and the definition \eqref{eqn:weighted_norm} of $u(x, t)$ we have $u(x, t) \le u(x, 0)$ for all $t > 0$. Therefore, it suffices to show for each $k \ge 1$
\begin{eqnarray}\label{eqn:cor1_contraction1}
    u(x, \eta_{d'}^k) \ge C'(\beta/\delta)^{d'+1} \quad \Longrightarrow \quad && u(x, \eta_{d'}^{k+1}) \le \lambda u(x, \eta_{d'}^k), \quad \quad \text{if} \>\> 1 \le d' \le d-1, \nonumber \\
    \text{and } \quad && u(x, \eta_{d}^{k+1}) \le \lambda u(x, \eta_{d}^k).
\end{eqnarray}
To do so, we note that the argument proving Lemma \ref{lemma:contraction} remains valid if we replace $u(x, \eta_{d'}^1)$ with $u(x, \eta_{d'}^{k+1})$, $u(x, 0)$ with $u(x, \eta_{d'}^k)$ and $\Delta X(0) = x$ with $\Delta X(\eta_{d'}^k)$ throughout---so long as \eqref{eqn:contraction3} is replaced by $\sum_{i=d' + 1}^d \beta^i Y_i(\eta_{d'}^k) \le \sum_{i=d' + 1}^d \beta^i Y_i(0) \le \tilde{C} \bsupnorm{x}{\delta} \> (\beta/\delta)^{d'+1}$
in the case where $1\le d' \le d-1$.
This follows directly from \eqref{eqn:contraction3} and Theorem \ref{thm:rbm_monotonicity} (iii) which gives $Y_i(\eta_{d'}^k) \le Y_i(0)$ for $i = 1, \ldots, d$.
\end{proof}

In the following lemma, we obtain estimates on tail probabilities for $\N_{d'}(x, t)$, defined in \eqref{eqn:boundary_hit2}, using results from \cite{banerjeebudhiraja} and the stochastic domination recorded in Theorem \ref{thm:domination}. Recall $k_0$ from Assumption \ref{assump:main}, which by definition was such that $d \ge k_0$.
\begin{lemma}[Boundary-hitting estimates]\label{lemma:boundary_hits}
Fix $d' \in \{k_0, \ldots, d\} $. Suppose $\blow^{(d')}>0$ and IV of Assumption \ref{assump:main} holds, and recall the definition of $\ak{d'}$ from \eqref{pardefcon}. Define the $d'$-dependent quantities
\begin{eqnarray}\label{eqn:hitting_pars}
\Tk{d'} = 1 + \left(\ak{d'}\right)^2\log(2d'), \ \
\Lambdak{d'} = \left(\ak{d'}\right)^{-2}. 
\end{eqnarray}
There exist positive constants $\delta', C''$ and $A_0 \ge 1$ not dependent on $d', d,\mu,R,\Sigma$, such that for any $x \in \Reals^d_+$, $A \ge A_0$ and $t \ge 4\Tk{d'}/\delta'$,
\begin{align}\label{eqn:boundary_hits_tail}
    \Prob{\N_{d'}(x, t) <  \delta' t/(4\Tk{d'})} &\le \exp\left(-t\frac{\delta' C''}{\Tk{d'}} \right) \notag \\
    &\qquad + \exp\left(-t  \frac{C''\Lambdak{d'}}{A}\right) \left\{1 + \exp \left(\frac{\supnorm{x|_{d'}}}{A\sigmalow \ak{d'}} \right) \right\}.
\end{align}
\end{lemma}
\begin{proof} Define $\bar{X}$ as in Theorem \ref{thm:domination} with $k = d'$. The theorem states $\bar{X}$ dominates $X|_{d'}$, the projection of the $d$-dimensional RBM with parameters $(\Sigma, \mu, R)$ onto the first $d'$ coordinates. Therefore, a coordinate of $X|_{d'}$ hits zero whenever the same coordinate of $\bar{X}$ hits zero. In other words, $\N_{d'}(x, t)$ dominates the corresponding quantity for $\bar{X}$, for all $x, t$.

By hypothesis of the lemma, $\bk{d'}>0$. As in \cite{banerjeebudhiraja}, for any $v \in \mathbb{R}^{d'}_+$ satisfying $R^{-1}v \le \bk{d'}, v>0,$ and any $y \in \mathbb{R}^{d'}_+$, define
\begin{align*}
    \| y\|^{\star}_{\infty, v} &:= \sup_{1 \le i \le d'} v_i \sigma_i^{-2}y_i, \quad \Lambda(v) := \inf_{1 \le i \le d'}\sigma_i^{-2}v_i^2,\\
     T(v) &:= \left(1 + \frac{\log\left(2\sum_{i=1}^{d'} v_i^2\sigma_i^{-2}/\Lambda(v))\right)}{\Lambda(v)}\right).
\end{align*}
With these definitions, recalling the stochastic domination noted in the previous paragraph, \cite[Proof of Lemma 8, Equations (33) and (41)]{banerjeebudhiraja} applied to the process $\bar{X}$ give positive constants $\delta', A_0$, not depending on $d', d,\mu,R,\Sigma$, such that for each $x \in \Reals^d_+$, $A \ge A_0$ and $t \ge 4T(v)/\delta'$,
\begin{align}\label{inter}
    &\Prob{\N_{d'}(x, t) <  \delta' t/(4T(v))} \nonumber \\
    &\le \exp\left(-\frac{\delta' t}{128 T(v)} \right) 
    + \exp\left(- \frac{\Lambda(v) t}{16 A}\right) \left\{1 + \exp \left(A^{-1} \| x|_{d'}\|^{\star}_{\infty, v} \right) \right\}.
\end{align}
From certain optimality properties of rates of convergence obtained in \cite{banerjeebudhiraja} (see \cite[Section 8]{banerjeebudhiraja}), we take $v = v^*$ where $v^*_i =  \left(\ak{d'}\right)^{-1} \sigma_i, \ 1 \le i \le d'$. Noting that $T(v^*) = \Tk{d'}$, $\Lambda(v^*) = \Lambdak{d'}$ and $\| x|_{d'}\|^{\star}_{\infty, v^*} \le \supnorm{x|_{d'}}/(\sigmalow \ak{d'})$, the lemma follows from \eqref{inter}.
\end{proof}
The following lemma combines the local contraction estimates obtained in Lemma \ref{lemma:contraction} and the probability estimates on number of times subsets of coordinates hit zero by time $t$, obtained in Lemma \ref{lemma:boundary_hits}, to furnish upper bounds on $ \Expect{u(x, t)}$, $x \in \mathbb{R}^d_+, t \ge 0$.
\begin{lemma}\label{lemma:convergence_intermediate}
Suppose Assumption \ref{assump:main} holds. Fix $d' \in \{k_0, \ldots, d\}$ and $x \in \mathbb{R}^d_+$. Recall $u(x, \cdot)$ from \eqref{eqn:weighted_norm}, the quantities $\lambda, \beta, \delta, C'$ in Lemma \ref{lemma:contraction}, and $A_0, \Lambdak{d'},\Tk{d'}, \delta', C''$ in Lemma \ref{lemma:boundary_hits}. Define
\begin{equation}\label{eqn:lambda_t}
    \lambda(t) = \lambda^{\floor{t\delta'/(4\Tk{d'})}}.
\end{equation}
Then for any $A \ge A_0$ and $t \ge 4\Tk{d'}/\delta'$,
\begin{align}\label{eqn:convergence_intermediate_result}
    \Expect{u(x, t)}
    &\le u(x, 0) \left[ \exp\left(-t\frac{\delta' C''}{\Tk{d'}} \right) + \exp\left(-t  \frac{C''\Lambdak{d'}}{A}\right) \left\{1 + \exp \left(\frac{\supnorm{x|_{d'}}}{A \ak{d'}\sigmalow} \right) \right\} \right] \notag\\
   & \qquad +   u(x, 0) \lambda(t) + C'\bsupnorm{x}{\delta}\>(\beta/\delta)^{d' + 1}.
\end{align}
In the case $d' = d$, \eqref{eqn:convergence_intermediate_result} holds without the $C'\bsupnorm{x}{\delta}\>(\beta/\delta)^{d' + 1}$ term in the bound.
\end{lemma}
\begin{proof}
With $\tau(x, d')$ as in Corollary \ref{cor:contraction1}, we have for any $A \ge A_0$ and $t \ge 4\Tk{d'}/\delta'$,
\begin{align}\label{eqn:convergence_intermediate_3}
    \Expect{u(x, t)} &\le \Expect{u(x, t)\ind{\tau(x, d') > t}} +  C'\bsupnorm{x}{\delta}\>(\beta/\delta)^{d' + 1} \notag \\
    &= \Expect{u(x, t)\ind{\tau(x, d') > t, \> \N_{d'}(x, t) < t\delta'/4\Tk{d'}}} \notag \\
    &\qquad + \Expect{u(x, t)\ind{\tau(x, d') > t, \> \N_{d'}(x, t) \ge t\delta'/4\Tk{d'}}} + C'\bsupnorm{x}{\delta}\>(\beta/\delta)^{d' + 1} \notag \\
    &\le u(x,0)  \Prob{\N_{d'}(x, t) <  \delta' t/(4\Tk{d'})} + \lambda(t)u(x,0) + C'\bsupnorm{x}{\delta}\>(\beta/\delta)^{d' + 1} \notag\\
    &\le u(x, 0) \left[ \exp\left(-t\frac{\delta' C''}{\Tk{d'}} \right) + \exp\left(-t  \frac{C''\Lambdak{d'}}{A}\right) \left\{1 + \exp \left(\frac{\supnorm{x|_{d'}}}{A \ak{d'}\sigmalow} \right) \right\} \right] \notag\\
    &\qquad + \lambda(t)u(x,0) + C'\bsupnorm{x}{\delta}\>(\beta/\delta)^{d' + 1},
\end{align}
where the second inequality follows from the monotonicity of $u$ and Corollary \ref{cor:contraction1}, and the last inequality follows from Lemma \ref{lemma:boundary_hits}. 

When $d'=d$, by Corollary \ref{cor:contraction1},
\begin{equation*}\label{eqn:convergence_intermediate_6}
    u(x, t) \> \ind{N_{d}(x, t) \ge t\delta'/4\Tk{d}} \le \lambda(t) u(x, 0).
\end{equation*}
Thus, again using the monotonicity of $u$ and applying Lemma \ref{lemma:boundary_hits} with $d' = d$, we have for any $A \ge A_0$ and $t \ge 4\Tk{d'}/\delta'$,
\begin{multline}\label{eqn:convergence_intermediate_7}
    \Expect{u(x, t)} \le \Expect{u(x, t) \> \ind{N_{d}(x, t) < t\delta'/4\Tk{d}}} + \Expect{u(x, t) \> \ind{N_{d}(x, t) \ge t\delta'/4\Tk{d}}} \\
    \le u(x, 0)\left[ \exp\left(-t\frac{\delta' C''}{\Tk{d}} \right) + \exp\left(-t  \frac{C''\Lambdak{d}}{A}\right) \left\{1 + \exp \left(\frac{\supnorm{x}}{A\ak{d}\sigmalow} \right) \right\} \right] + \lambda(t) u(x, 0).
\end{multline}
The lemma follows from \eqref{eqn:convergence_intermediate_3} and \eqref{eqn:convergence_intermediate_7}.
\end{proof}
For any $x \in \mathbb{R}^d_+$ and $d' \in \{k_0, \dots,d-1\}$, Lemma \ref{lemma:convergence_intermediate} shows that one can track the number of times the first $d'$ co-ordinates of $X(x,\cdot)$ hit zero by time $t$ to achieve exponential contraction in time $t$ of the weighted distance $u(x,\cdot)$ between $X(x,\cdot)$ and $X(0,\cdot)$, till $u(x,\cdot)$ hits $C'\bsupnorm{x}{\delta}\>(\beta/\delta)^{d' + 1}$. Thus, to ensure that this exponential contraction holds till $u(x,\cdot)$ is small, $d'$ should be close to $d$. However, for large $d$, choosing a large $d'$ slows down the convergence rate as it takes a long time for the $d'$ co-ordinates to hit zero. This is manifested in the large value of $\Tk{d'}$ which makes the exponential contraction coefficient in \eqref{eqn:convergence_intermediate_result} small. In the next lemma, we take an \emph{adaptive approach where the number of co-ordinates tracked increases with time}. Suppose Assumption \ref{assump:main} holds. With $r^* \ge 0$ as in III of Assumption \ref{assump:main}, set
\begin{equation}\label{eqn:ellt}
    \ell(t) = \begin{cases}
    d\wedge \floor{t^{1/(3 + 2r^*)}} &\quad \text{under Assumption } \ref{assump:main},\\
    d \wedge  \floor{t^{1/(1+2r^*)}} &\quad \text{under Assumption } \ref{assump:bounded_row}.\\
    \end{cases}
\end{equation}
$\ell(\cdot)$ represents the time varying number of coordinates of the process $X(x,\cdot)$ that must hit zero to achieve a desired contraction. The choice of $\ell(\cdot)$ is obtained by optimizing bounds on the exponents appearing in \eqref{eqn:convergence_intermediate_result} which depend on the assumptions.

\begin{lemma}[Decay rate of $\Expect{u(x, \> \cdot)}$]\label{lemma:convergence_mean_u} Fix an initial condition $X(x, 0) = x \ge 0$. With $\delta, \beta$ as in Lemma \ref{lemma:contraction}, recall the weighted supremum norm $\bsupnorm{x}{\delta}$ and the process $u(x, \cdot)$ as in \eqref{eqn:weighted_norm}. Define $\ell(\cdot)$ as in \eqref{eqn:ellt}.

If Assumption \ref{assump:main} holds, there exist constants $C_0, C_1 > 0$ not depending on $d, x$, $r^*$ such that, with $k_0' = k_0'(r^*) = k_0 \vee \left(\frac{8(3+2r^*)}{C_0' e} \right)^2$, we have for $d > k_0'$ and any $A \ge A_0$ ($A_0$ defined in Lemma \ref{lemma:boundary_hits}),
\begin{equation} \label{eqn:convergence_meanu_result_main}
    \Expect{u(x, t)}
    \le 
    \begin{cases}{\scriptstyle
   C_1 \left(u(x, 0)e^{\frac{\supnorm{x|_{\ell(t)}}}{A\sigmalow\ak{\ell(t)}}} + \bsupnorm{x}{\delta}\right) \> e^{-\frac{C_0}{A} t^{1/(3+2r^*)}}
    + C_1u(x, 0)\> e^{-C_0 \frac{t^{1/(3+2r^*)}}{\log t}}, 
    \quad  k_0' \le \ell(t) < d,
    }
    \\
    {\scriptstyle
    C_1 u(x, 0)e^{\frac{\supnorm{x}}{A\sigmalow\ak{d}}} \> e^{-C_0\frac{t}{Ad^{2(1+r^*)}}} + C_1 u(x, 0)e^{- C_0 \frac{t}{d^{2(1+r^*)}\log d}},
    \quad \ell(t) = d.
    }
    \end{cases}
\end{equation}
If Assumption \ref{assump:bounded_row} holds, we have using the same constants $k_0', C_0, C_1$,
\begin{equation} \label{eqn:convergence_meanu_result_bounded_row}
    \Expect{u(x, t)}
    \le 
    \begin{cases}
    {\scriptstyle
    C_1 \left(u(x, 0)e^{\frac{\supnorm{x|_{\ell(t)}}}{A\sigmalow\ak{\ell(t)}}} + \bsupnorm{x}{\delta}\right) \> e^{-\frac{C_0}{A}t^{1/(1+2r^*)}} + C_1u(x, 0)\> e^{-C_0 \frac{t^{1/(1+2r^*)}}{\log t}}, \quad  k_0' \le \ell(t) < d,
    }
    \\
    {\scriptstyle
    C_1 u(x, 0)e^{\frac{\supnorm{x}}{A\sigmalow\ak{d}}} \> e^{-C_0\frac{t}{Ad^{2r^*}}} + C_1 u(x, 0)e^{- C_0 \frac{t}{d^{2r^*}\log d}},
    \quad \ell(t) = d.
    }
    \end{cases}
\end{equation}
\end{lemma}
\begin{proof} We will employ Lemma \ref{lemma:convergence_intermediate} with $d' = \ell(t)$. We will consider two cases: $k_0 \le \ell(t) < d$ and $\ell(t) = d$.

In the work below, all constants depend on $\alpha, M, C, r^*,b_0, \sigmalow,\sigmahigh$ in the notation of Assumptions \ref{assump:main} and \ref{assump:bounded_row}, and $\beta \in (\alpha, 1)$ of Lemma \ref{lemma:contraction}.
\begin{case}[$k_0 \le \ell(t) < d$]
First suppose Assumption \ref{assump:main} holds. Set $d' = \ell(t)$ where for now we suppress the dependence on $t$. To employ the bound in Lemma \ref{lemma:convergence_intermediate}, we consider bounds on the quantities $\Tk{d'}, \Lambdak{d'}$  and $\ak{d'}$. III of Assumption \ref{assump:main} implies $\blowk{d'} \ge b_0 (d')^{-r^*}$ for some $b_0 >0$ not depending on $d$. This along with II, IV of Assumption \ref{assump:main} gives
\begin{equation}\label{eqn:convergence_mean_u7}
    \ak{d'} \le \sigmahigh \> \max_{1 \le i \le d'} \frac{1}{\bk{d'}_i}\sum_{j = 1}^{d'} (\rinv)_{ij} \le \frac{d' \sigmahigh M}{\blowk{d'}} \le (d')^{1+r^*} \frac{\sigmahigh M }{b_0}.
\end{equation}
Here, we have used $((R|_{d'})^{-1})_{ij} \le (\rinv)_{ij}$ for $1 \le i,j \le d'$ in the first inequality, which is a consequence of $P^T$ having non-negative entries.
From \eqref{eqn:convergence_mean_u7} and the definitions in \eqref{eqn:hitting_pars}, setting $A \ge A_0$ and recalling $d' = \ell(t) = d\wedge \floor{t^{1/(3 + 2r^*)}}$, there exists $C_0' > 0$ not dependent on $d', d, r^*$ such that for all $t \ge 2$,
\begin{eqnarray}\label{eqn:convergence_mean_u8}
-t \frac{\Lambdak{d'}}{A} &=& -t\frac{1}{A(\ak{d'})^2} \le -t\frac{b_0^2}{A(d')^{2(1+r^*)} \left(\sigmahigh M\right)^2}
 \le -\frac{C_0'}{A} t^{1/(3+2r^*)}, \nonumber \\
-t\frac{\delta'}{\Tk{d'}} &\le& -t\frac{\delta'}{1 +   (d')^{2(1+r^*)}\left(\frac{\sigmahigh M }{b_0}\right)^2\> \log\left(2 d' \right)} \le -C_0'\frac{t^{1/(3+2r^*)}}{\log t}.
\end{eqnarray}
$C_0'$ in the above can be taken to be $\left(2\left(\frac{\sigmahigh M }{b_0}\right)^2 + 2\log 2\right)^{-1}$. Recalling $\lambda(t) = \lambda^{\floor{t\delta'/4\Tk{d'}}}$, \eqref{eqn:convergence_mean_u8} also gives
\begin{equation}\label{eqn:convergence_mean_u10}
    \lambda(t) \le \lambda^{\frac{C_0'}{4}\frac{t^{1/(3+2r^*)}}{\log t} - 1}.
\end{equation}
In addition, \eqref{eqn:convergence_mean_u8} implies $\frac{4\Tk{d'}}{\delta'} \le \frac{4}{C_0'} t^{1-1/(3+2r^*)} \log t.$ 
Since $t^{-1/2(3+2r^*)}\log t$ as a function of $t$ is upper-bounded by $\frac{2(3+2r^*)}{e}$, we have $\frac{4\Tk{d'}}{\delta'} \le t$ for $\left(\frac{8(3+2r^*)}{C_0' e}\right)^{2(3 + 2r^*)} \le t$. These calculations show the condition $t \ge \frac{4\Tk{d'}}{\delta'}$ in Lemma \ref{lemma:convergence_intermediate} holds when  $\ell(t) \ge \left(\frac{8(3+2r^*)}{C_0' e}\right)^2$.

We now apply \eqref{eqn:convergence_mean_u8}, \eqref{eqn:convergence_mean_u10} to \eqref{eqn:convergence_intermediate_result} in Lemma \ref{lemma:convergence_intermediate}, with $A \ge A_0$. Recalling $k_0' = k_0'(r^*) = k_0 \vee \left(\frac{8(3+2r^*)}{C_0' e}\right)^2$ we have
\begin{multline}\label{eqn:convergence_mean_u11}
    \Expect{u(x, t)}
    \le u(x, 0)\left[1 + \exp \left(\frac{\supnorm{x|_{\ell(t)}}}{A\sigmalow \ak{\ell(t)}} \right) \right] \> e^{-\frac{C'' C_0'}{A}t^{1/(3+2r^*)}} + u(x, 0)e^{-C'' C_0'\frac{t^{1/(3+2r^*)}}{\log t}}\\
    + u(x, 0) \, \lambda^{\frac{C_0'}{4}\frac{t^{1/(3+2r^*)}}{\log t} - 1} +  C'\bsupnorm{x}{\delta}\>(\beta/\delta)^{t^{1/(3+2r^*)}}, \quad \quad \text{for} \quad   k_0' \le \ell(t) < d,
\end{multline}
where we used in the second line $d' + 1 = \ell(t) + 1 \ge t^{1/(3+2r^*)}$. This proves the first case in \eqref{eqn:convergence_meanu_result_main} with
\begin{eqnarray}\label{eqn:convergence_mean_u12}
C_0 &=& C'' C_0' \wedge \frac{C_0'}{4}\log \frac{1}{\lambda} \wedge \log \frac{\delta}{\beta}, \nonumber \\
C_1 &=& \left(2 + \frac{1}{\lambda}\right)\vee C'.
\end{eqnarray}
If Assumption \ref{assump:bounded_row} holds, we set $d' = \ell(t) = d\wedge \floor{t^{1/(1+2r^*)}}$. Instead of \eqref{eqn:convergence_mean_u7} we have
\begin{equation}\label{eqn:convergence_mean_u7_bc}
    \ak{d'} \le \sigmahigh \> \max_{1 \le i \le d'} \frac{1}{\bk{d'}_i}\sum_{j = 1}^{d'} (\rinv)_{ij} \le \frac{\sigmahigh M}{\blowk{d'}} \le (d')^{r^*} \frac{\sigmahigh M }{b_0}.
\end{equation}
Proceeding in the same way as \eqref{eqn:convergence_mean_u8}, we use \eqref{eqn:convergence_mean_u7_bc} to show
\begin{eqnarray}\label{eqn:convergence_mean_u8_bc}
-t \frac{\Lambdak{d'}}{A} &=& -t\frac{1}{A(\ak{d'})^2} \le -t\frac{b_0^2}{A(d')^{2r^*} \left(\sigmahigh M\right)^2}
 \le  -\frac{C_0'}{A} t^{1/(1+2r^*)}, \nonumber \\
-t\frac{\delta'}{\Tk{d'}} &\le& -t\frac{\delta'}{1 + (d')^{2r^*}\left(\frac{\sigmahigh M }{b_0}\right)^2\> \log\left(2 d' \right)} \le -C_0' \frac{t^{1/(1+2r^*)}}{\log t}.
\end{eqnarray}
Arguing as in \eqref{eqn:convergence_mean_u10} but using \eqref{eqn:convergence_mean_u8_bc} instead of \eqref{eqn:convergence_mean_u8}, we have
\begin{equation}\label{eqn:convergence_mean_u10_bc}
    \lambda(t) \le \lambda^{\frac{C_0'}{4}\frac{t^{1/(1+2r^*)}}{\log t} - 1}
\end{equation}
Using \eqref{eqn:convergence_mean_u8_bc}, we have $\frac{4\Tk{d'}}{\delta'} \le \frac{4}{C_0'} t^{1 - 1/(1+2r^*)} \log t$. As in the argument after \eqref{eqn:convergence_mean_u10}, $\frac{4\Tk{d'}}{\delta'} \le t$ for $t$ such that $\ell(t) \ge \left(\frac{8(1+2r^*)}{C_0' e} \right)^2$, under which Lemma \ref{lemma:convergence_intermediate} is valid. We now apply \eqref{eqn:convergence_mean_u8_bc}, \eqref{eqn:convergence_mean_u10_bc} to \eqref{eqn:convergence_intermediate_result}:
\begin{multline}\label{eqn:convergence_mean_u11_bc}
    \Expect{u(x, t)}
    \le u(x, 0)\left[1 + \exp \left(\frac{\supnorm{x|_{\ell(t)}}}{A\ak{\ell(t)} \sigmalow} \right) \right] \> e^{-\frac{C'' C_0'}{A} t^{1/(1+2r^*)}} + u(x, 0)e^{-C'' C_0' \frac{t^{1/(1+2r^*)}}{\log t}}\\
    + u(x, 0) \, \lambda^{\frac{C_0'}{4}\frac{t^{1/(1+2r^*)}}{\log t} - 1} +  C'\bsupnorm{x}{\delta}\>(\beta/\delta)^{t^{1/(1+2r^*)}}, \quad \quad \text{for} \quad k_0' \le \ell(t) < d.
\end{multline}
This proves the first case in \eqref{eqn:convergence_meanu_result_bounded_row} with $C_0, C_1$ as in \eqref{eqn:convergence_mean_u12}. Since $\left(\frac{8(1+2r^*)}{C_0' e} \right)^2 < \left(\frac{8(3+2r^*)}{C_0' e} \right)^2$, we use the same $k_0'$ in  \eqref{eqn:convergence_mean_u11} and \eqref{eqn:convergence_mean_u11_bc}.
\end{case}
\begin{case}[$\ell(t) = d$] First we consider the scenario of Assumption \ref{assump:main}, in which case $\ell(t) = d$ implies $t \ge d^{3+2r^*}$. We follow the same basic recipe: We use Lemma \ref{lemma:convergence_intermediate}, this time in the case $d' = d$, and bound the quantities $\ak{d}, \Tk{d}, \Lambdak{d}$.

The bound on $\ak{d'}$ in \eqref{eqn:convergence_mean_u7} continues to hold with $d' = d$, and using this with we have
\begin{eqnarray}\label{eqn:convergence_mean_u15}
-t \frac{\Lambdak{d}}{A} &=& -t\frac{1}{A(\ak{d})^2} \le -t\frac{Ab_0^2}{d^{2(1+r^*)} \left(\sigmahigh M\right)^2}
 \le -C_0' \frac{t}{d^{2(1+r^*)}}, \nonumber \\
-t\frac{\delta'}{\Tk{d}} &\le& -t\frac{\delta'}{1 +   d^{2(1+r^*)}\left(\frac{\sigmahigh M }{b_0}\right)^2\> \log\left(2 d \right)} \le -C_0'\frac{t}{d^{2(1+r^*)}\log d}.
\end{eqnarray}
Now \eqref{eqn:convergence_mean_u15} implies
\begin{equation}\label{eqn:convergence_mean_u16}
    \lambda(t) \le \lambda^{\frac{C_0'}{4} \frac{t}{d^{2(1+r^*)}\log d} - 1}.
\end{equation}
Since the lemma statement has imposed $d > k_0'$ we have $t \ge [k_0'(r^*)]^{3+2r^*}$. Applying the argument preceding \eqref{eqn:convergence_mean_u11}, this implies $t \ge \frac{4 \Tk{d}}{\delta'}$ and thus Lemma \ref{lemma:convergence_intermediate} holds for all $t \ge [k_0'(r^*)]^{3+2r^*}$ in the case $\ell(t) = d$. Using \eqref{eqn:convergence_mean_u15}, \eqref{eqn:convergence_mean_u16} in \eqref{eqn:convergence_intermediate_result} (without the $C'\bsupnorm{x}{\delta}\>(\beta/\delta)^{d' + 1}$ term) we have
\begin{equation}\label{eqn:convergence_mean_u17}
    \Expect{u(x, t)}
    \le C_1u(x, 0)e^{\frac{\supnorm{x}}{A\ak{d} \sigmalow}} \> e^{-C_0\frac{t}{Ad^{2(1+r^*)}}} + C_1 u(x, 0)e^{- C_0 \frac{t}{d^{2(1+r^*)}\log d}},
    \quad \quad \text{for} \quad \ell(t) = d,
\end{equation}
where we use $C_0, C_1$ from \eqref{eqn:convergence_mean_u12}. This is the second line in \eqref{eqn:convergence_meanu_result_main}.

When Assumption \ref{assump:bounded_row} holds, $\ell(t) = d$ implies $t \ge d^{1+2r^*}$. The second line of \eqref{eqn:convergence_meanu_result_bounded_row} is proven in identical fashion to \eqref{eqn:convergence_mean_u17}, after accounting for the stronger assumptions in the same way as we did in \eqref{eqn:convergence_mean_u8_bc} and \eqref{eqn:convergence_mean_u11_bc}.
\end{case}
\end{proof}
For any $\beta \in (\alpha,1)$ and $x \in \mathbb{R}^d_+$, Lemma \ref{lemma:convergence_mean_u} gives quantitative estimates for the decay rate of the $\bnorm{\cdot}$ distance between $X(x,\cdot)$ and $X(0,\cdot)$. To use this in furnishing rates of convergence to stationarity in $\bnorm{\cdot}$ distance starting from any $x \in \mathbb{R}^d_+$, namely Theorem \ref{thm:main_fromx}, we use Theorem \ref{thm:rbm_monotonicity} to make the following simple observation. Recalling $u$ in \eqref{eqn:weighted_norm} and $u_\pi$ in \eqref{eqn:u_pi}, we have by the triangle inequality,
\begin{align}\label{eqn:main_fromx1}
\bnorm{\left(X(x, t) - X(X(\infty), t)\right)} &\le \bnorm{\left(X(x, t) - X(0, t)\right)} + \bnorm{\left(X(X(\infty), t) - X(0, t)\right)}\notag\\
&\le \bnorm{\rinv\left(X(x, t) - X(0, t)\right)} \notag \\
&+ \bnorm{\rinv\left(X(X(\infty), t) - X(0, t)\right)}\notag\\
 & = u(x, t) + u_\pi(t).
\end{align}
To bound the expectation of the final two terms in \eqref{eqn:main_fromx1}, we apply Lemma \ref{lemma:convergence_mean_u} to bound $\Expect{u(x, t)}$. To bound $\Expect{u_\pi(t)}$, we will use a slightly altered version of Lemma \ref{lemma:convergence_intermediate} and Lemma \ref{lemma:convergence_mean_u} conditional on $x = X(\infty)$ followed by taking expectation in the law of $X(\infty)$. This will require quantitative control over moments of several functionals of $X(\infty)$. This is the objective of the following lemma.
\begin{lemma}[Moments under stationarity]\label{lemma:stat_moments}
Suppose Assumption \ref{assump:main} holds, with $\alpha \in (0, 1)$ set therein. Fix $\beta \in (\alpha, 1)$ and define $u(x, 0) = \bnorm{x}$ as in \eqref{eqn:weighted_norm}. Fix $\delta \in (\beta, 1)$.

Recall the random variable $X(\infty)$ distributed as the stationary distribution for the process \eqref{eqn:rbm}. Fix $d' \in \{k_0 \ldots d\}$. Then there exists a constant $C'''\ge 1$ not depending on $d', d$ or $r^*$ (see III of Assumption \ref{assump:main}) such that
\begin{eqnarray}
\Expect{\exp\left(2 \frac{\supnorm{X|_{d'}(\infty)}}{A\sigmalow\ak{d'}} \right)} &\le& 1+d'\quad \quad \text{for }\> A \ge 2 d'\frac{\sigmahigh M}{\sigmalow}, \label{eqn:stat_moment_result1}\\
\Expect{\bsupnorm{X(\infty)}{\delta}} \le \Expect{\bsupnorm{X(\infty)}{\sqrt{\delta}}} &\le& C''' L_1(\delta), \label{eqn:stat_moment_result2}\\
\Expect{u\left(X(\infty), 0\right)} \le \sqrt{\Expect{u^2\left(X(\infty), 0\right)}} &\le& C''' L_1(\delta). \label{eqn:stat_moment_result3}
\end{eqnarray}
where
$L_1(\delta) := \left(k_0^{r^* + 1} +  \sum_{i = k_0}^d i^{3 + r^*} \delta^{i/2} \right)$. 
If in addition Assumption \ref{assump:bounded_row} holds, we have
\begin{eqnarray}
\Expect{\exp\left(2 \frac{\supnorm{X|_{d'}(\infty)}}{A\sigmalow\ak{d'}} \right)} &\le& 1+d'\quad \quad \text{for }\> A \ge 2 \frac{\sigmahigh M}{\sigmalow}, \label{eqn:stat_moment_result1_bc}\\
\Expect{\bsupnorm{X(\infty)}{\delta}} \le \Expect{\bsupnorm{X(\infty)}{\sqrt{\delta}}} &\le& C''' L_2(\delta), \label{eqn:stat_moment_result2_bc}\\
\Expect{u\left(X(\infty), 0\right)} \le \sqrt{\Expect{u^2\left(X(\infty), 0\right)}} &\le& C''' L_2(\delta). \label{eqn:stat_moment_result3_bc}
\end{eqnarray}
where
$L_2(\delta) := \left(k_0^{r^*} +  \sum_{i = k_0}^d i^{2 + r^*} \delta^{i/2} \right)$.
\end{lemma}

\begin{proof}
For $k \in \{k_0, \ldots, d\}$, write $\bar{X}^{(k)}$ for the process $\bar{X}$ defined in Theorem \ref{thm:domination}.

III of Assumption \ref{assump:main} imposes $- \rinvk{d'} \muk{d'} = \bk{d'} > 0$. Thus $\bar{X}^{(d')}$ has a stationary distribution (\cite{harrisonwilliams}, Section 6). We write $\bar{X}^{(d')}(\infty)$ for the random variable with this distribution. From \cite[Lemma 4, Section 6]{harrisonwilliams} and \cite[Lemma 12 and its proof, Section 6]{harrisonwilliams}, for any $\thetak{d'} \in \Reals^{d'}$ such that 
\begin{equation}\label{eqn:stat_moments1}
    \thetak{d'} > 0, \quad \quad \rinvk{d'} \thetak{d'} \le \bk{d'},
\end{equation}
we have
\begin{align}\label{eqn:stat_moments2}
   \Prob{\rinvk{d'}\bar{X}^{(d')}(\infty) \le \rinvk{d'} z} &= \liminf_{t \to \infty}\Prob{\rinvk{d'}\bar{X}^{(d')}(0, t) \le \rinvk{d'} z} \notag \\
   & \ge 1 - \sum_{j = 1}^{d'} \exp \left(- 2z_j\frac{\thetak{d'}_j}{\sigma_j} \right) , \quad \quad z \in \Reals^{d'}_+.
\end{align}
In other words, the distribution of $\rinvk{d'}\bar{X}^{(d')}(\infty)$ has exponential tails. This is the key fact in proving the lemma, and the remainder of the argument is in choosing $\thetak{d'}$ appropriately to achieve the desired dependence on the parameters and dimension. Recalling the quantity $\ak{d'}$ from Assumption \ref{assump:main} we set \begin{equation}\label{eqn:stat_moments3}
    \thetak{d'}_i = \frac{\sigmalow}{\ak{d'}}, \quad \quad 1 \le i \le d'.
\end{equation}
By definition of $\ak{d'}$, for each $1 \le i \le d'$,
\begin{align}
    \sum_{\ell=1}^{d'}(\rinvk{d'})_{i \ell}\thetak{d'}_{\ell} 
    &\le \sum_{\ell=1}^{d'}(\rinvk{d'})_{i \ell}\left(\frac{\bk{d'}_{i}\sigmalow}{\sum_{j=1}^{d'}(\rinvk{d'})_{i j}\sigma_j}\right) \nonumber \\ 
    &\le \sum_{\ell=1}^{d'}(\rinvk{d'})_{i \ell}\left(\frac{\bk{d'}_{i}\sigma_{\ell}}{\sum_{j=1}^{d'}(\rinvk{d'})_{i j}\sigma_j}\right) = \bk{d'}_{i},
\end{align}
and hence, $\thetak{d'}$ satisfies \eqref{eqn:stat_moments1}.

We now prove the exponential moments \eqref{eqn:stat_moment_result1} and \eqref{eqn:stat_moment_result1_bc}. Since we consider a fixed $d'$ here, we write $\bar{X}(\infty) = \bar{X}^{(d')}(\infty)$ to lighten notation.  Note Theorem \ref{thm:domination} implies $\bar{X}(\infty)$ stochastically dominates $X|_{d'}(\infty)$. Hence, since $(\rinvk{d'})_{ij} \ge 0$ we have for any $z \in \Reals^{d'}_+$,
\begin{equation} \label{eqn:stat_moments4}
 \Prob{\left(\rinvk{d'} X|_{d'}(\infty)\right)_i \le z_i\>, \quad 1 \le i \le d'}
    \ge \Prob{\left(\rinvk{d'} \bar{X}(\infty)\right)_i \le z_i\>, \quad 1 \le i \le d'}.
\end{equation}
For arbitrary $z_0 \ge 1$, setting $z_i = (\log z_0) \frac{A \ak{d'}}{2}$ for each $i = 1, \ldots, d'$ in \eqref{eqn:stat_moments4}, 
\begin{align}\label{eqn:stat_moments5}
   & \Prob{\exp \left(2\frac{\supnorm{X|_{d'}(\infty)} }{A \ak{d'}}\right)\le z_0}
    = \Prob{X_i(\infty) \le (\log z_0) \frac{A \ak{d'}}{2}\>, \quad 1 \le i \le d'} \notag\\
   & \ge \Prob{\left(\rinvk{d'} X|_{d'}(\infty)\right)_i \le (\log z_0) \frac{A \ak{d'}}{2}\>, \quad 1 \le i \le d'}\notag\\
   & \ge \Prob{\left(\rinvk{d'} X|_{d'}(\infty)\right)_i \le (\log z_0) \frac{A \ak{d'}}{2} \frac{\sum_{j = 1}^{d'}(\rinvk{d'})_{ij}}{\max_{1 \le k \le d'}\sum_{j = 1}^{d'}(\rinvk{d'})_{kj}}\>, \quad 1 \le i \le d'} \notag\\
   & \ge 1 - \sum_{j = 1}^{d'} \exp \left(- (\log z_0) \frac{A \sigmalow}{\sigma_j \max_{1 \le k \le d'}\sum_{j = 1}^{d'}(\rinvk{d'})_{kj}} \right)
    \ge 1 - d' \exp \left(- (\log z_0) \frac{A \sigmalow}{\sigmahigh M d'} \right).
\end{align}
We used in the second line $\rinvk{d'}x \ge x, \ \forall \ x \in \mathbb{R}^{d'}_+$. For the third inequality, we used \eqref{eqn:stat_moments2} with the $d'$-dimensional vector $(\log z_0) \frac{A \ak{d'}}{2}\left(\max_{1 \le k \le d'}\sum_{j = 1}^{d'}(\rinvk{d'})_{kj}\right)^{-1}(1,\dots,1)^T$ in place of $z$. For the last inequality, we used II and IV of Assumption \ref{assump:main}. From \eqref{eqn:stat_moments5} we obtain \eqref{eqn:stat_moment_result1} as follows:
\begin{align}\label{eqn:stat_moments6}
    \Expect{\exp\left(2 \frac{\supnorm{X|_{d'}(\infty)}}{A \sigmalow \ak{d'}} \right)}
    &\le 1 + \int_1^\infty \Prob{\exp \left(2\frac{\supnorm{X|_{d'}(\infty)} }{A \ak{d'}}\right) > z_0} \> dz_0 \notag\\ 
    &\le 1 + d' \int_1^\infty  \exp \left(- (\log z_0) \frac{A \sigmalow}{\sigmahigh M d'} \right) \le 1 + d',
\end{align}
for $A \ge 2 d' \frac{\sigmahigh M }{\sigmalow}.$ If instead Assumption \ref{assump:bounded_row} holds, then 
\begin{equation}
  \max_{1 \le k \le d'}\sum_{j = 1}^{d'}(\rinvk{d'})_{kj} \le \max_{1 \le k \le d'}\sum_{j = 1}^{d'}(\rinv)_{kj} \le M.  
\end{equation}
Instead of \eqref{eqn:stat_moments5} we have, 
\begin{align}\label{eqn:stat_moments5_bc}
    &\Prob{\exp \left(2\frac{\supnorm{X|_{d'}(\infty)} }{A \ak{d'}}\right)\le z_0}
    = \Prob{X_i(\infty) \le (\log z_0) \frac{A \ak{d'}}{2}\>, \quad 1 \le i \le d'} \notag \\
    &\ge 1 - \sum_{j = 1}^{d'} \exp \left(- (\log z_0) \frac{A \sigmalow}{\sigma_j \max_{1 \le k \le d'}\sum_{j = 1}^{d'}(\rinvk{d'})_{kj}} \right) \notag\\
    &\ge 1 - \sum_{j = 1}^{d'} \exp \left(- (\log z_0) \frac{A \sigmalow}{\sigma_j M} \right) 
    \ge 1 - d' \exp \left(- (\log z_0) \frac{A \sigmalow}{\sigmahigh M} \right).
\end{align}
This proves \eqref{eqn:stat_moment_result1_bc} by proceeding exactly as in \eqref{eqn:stat_moments6}, using \eqref{eqn:stat_moments5_bc} in place of \eqref{eqn:stat_moments5}. 

We turn to \eqref{eqn:stat_moment_result2}, recalling the notation $\bar{X}^{(k)}(\infty)$ from the start of this proof. By Theorem \ref{thm:domination}, $X_i(\infty) \le \bar{X}^{(k_0)}_i(\infty)$ for $i = 1, \ldots, k_0$ and $X_i(\infty) \le \bar{X}^{(i)}_i(\infty)$ for $i = k_0, \ldots, d$. This implies \begin{equation}\label{eqn:stat_moments7}
    \Prob{\bsupnorm{X(\infty)}{\sqrt{\delta}} > z_0} 
    \le
    \Prob{\bsupnorm{\bar{X}^{(k_0)}(\infty)}{\sqrt{\delta}} > z_0} + \sum_{i=k_0}^d \Prob{\delta^{i/2}\bar{X}^{(i)}_i(\infty) > z_0}.
\end{equation}
In preparation to handle the first probability of the right-hand side in \eqref{eqn:stat_moments7}, we note that by I, II of Assumption \ref{assump:main} and $(\rinvk{k_0})_{ij} \le (\rinv)_{ij}$,
\begin{align}\label{eqn:stat_moments8}
    \sum_{j = 1}^{k_0}(\rinvk{k_0})_{ij}\delta^{-j/2} 
    &\le M\sum_{j = 1}^{i}\delta^{-j/2} + \ind{i < k_0} C \sum_{j = i+1}^{k_0} \alpha^{j-i} \delta^{-j/2} \notag\\
    &\le \delta^{-i/2}\frac{M}{1 - \sqrt{\delta}}
    + \ind{i < k_0} C \delta^{-i/2}\sum_{j = 1}^{k_0-i} \left(\alpha / \sqrt{\delta} \right)^{j} \le C' \delta^{-i/2},
\end{align}
for $C' = \frac{M}{1-\sqrt{\delta}} + C \frac{\alpha/\sqrt{\delta}}{1-(\alpha/\sqrt{\delta})}$, recalling that $0 < \alpha < \sqrt{\delta} < 1$. In the following, we set $\thetak{i}_j = \frac{\sigmalow}{\ak{k_0}}$ for $k_0 \le i \le d, 1 \le j \le i$. Note that,
\begin{align}\label{eqn:stat_moments9}
    &\Prob{\bsupnorm{\bar{X}^{(k_0)}(\infty)}{\sqrt{\delta}} \le z_0} \nonumber\\
    &\ge \Prob{\left(\rinvk{k_0}\bar{X}^{(k_0)}(\infty)\right)_i \le z_0 \delta^{-i/2}, \quad i = 1, \ldots, k_0}\notag\\
    &\ge  \Prob{\left(\rinvk{k_0}\bar{X}^{(k_0)}(\infty)\right)_i \le z_0 \frac{\sum_{j=1}^{k_0}(\rinvk{k_0})_{ij}\delta^{-j/2}}{C'}, \quad i = 1, \ldots, k_0} \notag \\
    &\ge 1 - \sum_{j = 1}^{k_0} \exp \left(- 2z_0\delta^{-j/2}\frac{ \thetak{k_0}_j}{C'\sigma_j} \right) 
    \ge  1 - \sum_{j = 1}^{k_0} \exp \left(- 2z_0\delta^{-j/2}\frac{ \sigmalow}{\ak{k_0}C'\sigmahigh} \right).
\end{align}
In the first line we used $\left(\rinvk{k_0} \bar{X}^{(k_0)}_i(\infty)\right)_i \ge \bar{X}^{(k_0)}_i(\infty), \ i =1,\dots,k_0$. The second line uses \eqref{eqn:stat_moments8}. The last line applies \eqref{eqn:stat_moments2} with $k_0$ in place of $d'$ and with $z_j = (C')^{-1}z_0\delta^{-j/2}, \ 1 \le j \le k_0$, and uses IV of Assumption \ref{assump:main}. 

Now we bound $\Prob{\delta^{i/2}\bar{X}^{(i)}_i(\infty) > z_0}$ for $i = k_0,\dots,d$ required to bound the second term of the right hand side in \eqref{eqn:stat_moments7}. In the following equations we use II of Assumption \ref{assump:main}, which says $(\rinvk{i})_{kj} \le M$, in the third line and \eqref{eqn:stat_moments2} to show
\begin{align}\label{eqn:stat_moments10}
    \Prob{\delta^{i/2}\bar{X}^{(i)}_i(\infty) \le z_0} 
    &\ge \Prob{\left(\rinvk{i}\bar{X}^{(i)}(\infty)\right)_i \le z_0 \delta^{-i/2}}\notag\\
    &\ge \Prob{\left(\rinvk{i}\bar{X}^{(i)}(\infty)\right)_k \le z_0 \delta^{-i/2}, \quad k = 1, \ldots, i} \notag\\
    &\ge \Prob{\left(\rinvk{i}\bar{X}^{(i)}(\infty)\right)_k \le z_0 \delta^{-i/2} \frac{\sum_{j = 1}^i(\rinvk{i})_{kj}}{iM}, \quad k = 1, \ldots, i} \notag\\
    &\ge 1 - \sum_{j = 1}^{i} \exp \left(- 2z_0\delta^{-i/2}\frac{ \thetak{i}_j}{iM\sigma_j} \right)
    \ge 1 - i \exp \left(- 2z_0\delta^{-i/2}\frac{ \sigmalow}{i\ak{i}M\sigmahigh}\right),
\end{align}
for $i = k_0, \ldots, d$, where in the first line we used $\left(\rinvk{k_0} \bar{X}^{(i)}(\infty)\right)_i \ge \bar{X}^{(i)}_i(\infty)$. Applying \eqref{eqn:stat_moments9}, \eqref{eqn:stat_moments10} to \eqref{eqn:stat_moments7} we get
\begin{align}\label{eqn:stat_moments11}
    \Prob{\bsupnorm{X(\infty)}{\sqrt{\delta}} > z_0}  
    &\le \sum_{i = 1}^{k_0} \exp \left(- 2z_0\delta^{-i/2}\frac{ \sigmalow}{\ak{k_0}C'\sigmahigh} \right) \nonumber \\
    &+ \sum_{i = k_0}^d i \exp \left(- 2z_0\delta^{-i/2}\frac{ \sigmalow}{i\ak{i}M\sigmahigh}\right).
\end{align}
As a result,
\begin{align}\label{eqn:stat_moments12}
    &\Expect{\bsupnorm{X(\infty)}{\sqrt{\delta}}} = \int_0^\infty \Prob{\bsupnorm{X(\infty)}{\sqrt{\delta}} > z_0}\> dz_0 \notag\\
    &\le \ak{k_0}\frac{\sigmahigh C'}{2\sigmalow} \sum_{i = 1}^{k_0} \delta^{i/2} + \frac{M \sigmahigh}{2\sigmalow}\sum_{i = k_0}^d i^2 \ak{i} \delta^{i/2}
    \le \ak{k_0}\frac{\sigmahigh C'\sqrt{\delta}}{2\sigmalow(1-\sqrt{\delta})}  + \frac{M \sigmahigh}{2\sigmalow}\sum_{i = k_0}^d i^2 \ak{i} \delta^{i/2}\notag\\
    &\le k_0^{r^* + 1} \frac{MC'\sigmahigh^2 \sqrt{\delta}}{2b_0\sigmalow(1-\sqrt{\delta})} + \frac{M^2 \sigmahigh^2}{2b_0\sigmalow}\sum_{i = k_0}^d i^{3 + r^*} \delta^{i/2} 
    \le C'' \left(k_0^{r^* + 1} +  \sum_{i = k_0}^d i^{3 + r^*} \delta^{i/2} \right),
\end{align}
with $C'' = \frac{M^2 \sigmahigh^2}{2b_0\sigmalow} \vee \frac{MC'\sigmahigh^2 \sqrt{\delta}}{2b_0\sigmalow(1-\sqrt{\delta})}$. In the final line we used the fact that $\ak{i} \le \frac{iM\sigmahigh}{\blowk{i}}$ by definition of $\ak{i}$ and Assumption \ref{assump:main}, and $\blowk{i} \ge b_0 i^{-r^*}$ using III of Assumption \ref{assump:main}. If instead Assumption \ref{assump:bounded_row} holds, $\ak{i} \le \frac{M\sigmahigh}{\blowk{i}} \le i^{r^*}\frac{M\sigmahigh}{b_0} $. Substituting this fact in the final line of \eqref{eqn:stat_moments12}, but otherwise proceeding in exactly the same way, produces \eqref{eqn:stat_moment_result2_bc} with the same choice of $C''$.

Now we show \eqref{eqn:stat_moment_result3} and \eqref{eqn:stat_moment_result3_bc}. We need prove only the second inequality in \eqref{eqn:stat_moment_result3}, \eqref{eqn:stat_moment_result3_bc}. Using Jensen's inequality in the first line below, we have 
\begin{align}\label{eqn:stat_moments13}
    u^2(X(\infty), 0) 
    &= \left(\sum_{i=1}^d\beta^i \sum_{j=1}^d(\rinv)_{ij}(\infty) X_j(\infty)\right)^2
    \le \frac{\beta}{1-\beta} \sum_{i=1}^d\beta^i \left( \sum_{j=1}^d(\rinv)_{ij}(\infty) X_j(\infty)\right)^2 \notag\\
    &\le \bsupnorm{X(\infty)}{\sqrt{\delta}}^2 \frac{\beta}{1-\beta} \sum_{i=1}^d\beta^i \left( \sum_{j=1}^d(\rinv)_{ij}(\infty) \delta^{-j/2}\right)^2 \notag\\
    &\le \bsupnorm{X(\infty)}{\sqrt{\delta}}^2 \frac{(C')^2\beta}{1-\beta} \sum_{i=1}^d\beta^i \delta^{-i} 
    \le \bsupnorm{X(\infty)}{\sqrt{\delta}}^2 \frac{(C')^2\beta}{1-\beta} \frac{\beta / \delta}{1 - \beta/\delta}.
\end{align}
In the second line, we used \eqref{eqn:stat_moments8} with $d$ in place of $k_0$ and $C'$ set therein. For the final line recall $\beta \in (\alpha, \delta)$. Using \eqref{eqn:stat_moments11} to bound the quadratic moment on the right-hand side of \eqref{eqn:stat_moments13},
\begin{align}\label{eqn:stat_moments14}
    \Expect{u^2(X(\infty), 0)} 
    &\le \frac{(C')^2\beta}{1-\beta} \frac{\beta / \delta}{1 - \beta/\delta} \> \Expect{\bsupnorm{X(\infty)}{\sqrt{\delta}}^2} \notag\\
    &= \frac{(C')^2\beta}{1-\beta} \frac{\beta / \delta}{1 - \beta/\delta}\> \int_0^\infty \Prob{\bsupnorm{X(\infty)}{\sqrt{\delta}} > \sqrt{z_0}} \> dz_0 \notag\\
    &\le \frac{(C')^2\beta}{1-\beta} \frac{\beta / \delta}{1 - \beta/\delta}\>   \int_0^\infty \sum_{i = 1}^{k_0} \exp \left(- 2\sqrt{z_0}\delta^{-i/2}\frac{ \sigmalow}{\ak{k_0}C'\sigmahigh} \right) dz_0 \notag \\
    &+ \frac{(C')^2\beta}{1-\beta} \frac{\beta / \delta}{1 - \beta/\delta}\>   \int_0^\infty \sum_{i = k_0}^d i \exp \left(- 2\sqrt{z_0}\delta^{-i/2}\frac{ \sigmalow}{i\ak{i}M\sigmahigh}\right) dz_0 \notag \\
    &= 2 \frac{(C')^2\beta}{1-\beta} \frac{\beta / \delta}{1 - \beta/\delta} \> 
    \left(\left(\frac{\ak{k_0}C'\sigmahigh}{2\sigmalow}\right)^2 \sum_{i=1}^{k_0}  \delta^i 
    + \left(\frac{M\sigmahigh}{2\sigmalow}\right)^2\sum_{i=k_0}^{d} \delta^i i^3  \left(\ak{i}\right)^2 \right).
\end{align}
Under Assumption \ref{assump:main} we have $\ak{i} \le \frac{iM\sigmahigh}{\blowk{i}} \le i^{1+r^*}\frac{M \sigmahigh}{b_0}$, and applying this to \eqref{eqn:stat_moments14} gives
\begin{equation}\label{eqn:stat_moments15}
    \Expect{u^2(X(\infty), 0)}  \le (C''')^2 \left(k_0^{2(r^* + 1)} +  \sum_{i = k_0}^d i^{5 + 2r^*} \delta^{i} \right),
\end{equation}
where we have chosen $C''' \ge C''$ to be large enough that both \eqref{eqn:stat_moments15} and \eqref{eqn:stat_moments12} are satisfied,
\begin{equation}\label{eqn:stat_moments16}
    C''' = C'' \vee  \sqrt{\frac{(C')^2\beta}{2(1-\beta)} \frac{\beta / \delta}{1 - \beta/\delta} \left(\frac{M \sigmahigh}{b_0} \right)^2\left(\left(\frac{C'\sigmahigh}{\sigmalow}\right)^2\frac{\delta}{1-\delta} + \left(\frac{M\sigmahigh}{\sigmalow}\right)^2\right)}.
\end{equation}
Under Assumption \ref{assump:bounded_row} we have $\ak{i} \le \frac{M\sigmahigh}{\blowk{i}} \le i^{r^*}\frac{M \sigmahigh}{b_0}$, so by \eqref{eqn:stat_moments14}
\begin{equation}\label{eqn:stat_moments17}
    \Expect{u^2(X(\infty), 0)}  \le (C''')^2 \left(k_0^{2r^*} +  \sum_{i = k_0}^d i^{3 + 2r^*} \delta^{i} \right).
\end{equation}
After taking square roots and using $\sum_1^m x_i^2 \le \left(\sum_1^m x_i\right)^2$ for any non-negative numbers $x_1 \ldots x_m$, \eqref{eqn:stat_moments15} proves \eqref{eqn:stat_moment_result3} and \eqref{eqn:stat_moments17} proves \eqref{eqn:stat_moment_result3_bc}.
\end{proof}

Now we bound $\Expect{u_\pi(t)}$. We would like simply to use Lemma \ref{lemma:convergence_mean_u} conditional on $x = X(\infty)$ followed by taking expectation in the law of $X(\infty)$. We will do so to prove \eqref{eqn:main_result_bounded_row} under Assumption \ref{assump:bounded_row}, but this is not desirable under Assumption \ref{assump:main} for the following reason.

If one tries this approach under Assumption \ref{assump:main}, terms of the form $\Expect{\exp\left(2 \frac{\supnorm{X|_{\ell(t)}(\infty)}}{A\sigmalow\ak{\ell(t)}} \right)}$ (where $\ell(\cdot)$ is defined in \eqref{eqn:ellt}) appear in the bound and $A$ should be chosen large enough so that this expectation is finite. Lemma \ref{lemma:stat_moments} shows this requires $A$ to be of order $\ell(t)$. However, such choice of $A$ implies that $e^{-\frac{C_0}{A} t^{1/(3+2r^*)}}$ is bounded below by a positive dimension independent constant as $t \rightarrow \infty$, thereby lending the bounds obtained via Lemma \ref{lemma:convergence_mean_u} trivial. 

Thus, under Assumption \ref{assump:main}, we proceed by choosing a higher number of coordinates of $X(x,\cdot)$ that must hit zero in order to achieve a desirable contraction in $\Expect{u_\pi(\cdot)}$. Namely, instead of $\ell(\cdot)$ of Lemma \ref{lemma:convergence_mean_u}, we define
\begin{equation}\label{eqn:dt}
    d(t) = \begin{cases}
    d \wedge  \floor{t^{1/(4+2r^*)}} &\quad \text{under Assumption } \ref{assump:main},\\
    d \wedge  \floor{t^{1/(1+2r^*)}} &\quad \text{under Assumption } \ref{assump:bounded_row},\\
    \end{cases}
\end{equation}
with $r^* \ge 0$ as in III of Assumption \ref{assump:main}.

\begin{lemma}[Decay rate of $\Expect{u_{\pi}(\cdot)}$]\label{statconv}
Suppose Assumption \ref{assump:main} holds for $X$, an $\RBM(\Sigma, \mu, R)$, with $\alpha \in (0, 1)$ defined therein. Fix $\beta \in (\alpha, 1)$, $\delta \in (\beta,1)$, and recall the weighted distance $u_\pi(\cdot)$ from \eqref{eqn:u_pi}. Recall $L_1(\delta), L_2(\delta)$ from Lemma \ref{lemma:stat_moments}.

There exist constants $\bar{C}_0, \bar{C}_1, C_0' > 0$ not depending on $d$ or $r^*$ such that, with $k_0'' = k_0''(r^*) = \max\left\lbrace k_0, \frac{A_0\sigmalow}{2\sigmahigh M}, \left(\frac{8(4+2r^*)}{C_0' e} \right)^2\right\rbrace$ ($A_0$ defined in Lemma \ref{lemma:boundary_hits}), we have for $d > k_0''$
\begin{equation}\label{eqn:main_result}
    \Expect{u_\pi(t)}
    \le 
    \begin{cases}
    \bar{C}_1 L_1(\delta)\sqrt{1+d(t)} \> e^{-\bar{C}_0 t^{1/(4+2r^*)}} + \bar{C}_1 L_1(\delta) \> e^{-\bar{C}_0\frac{t^{1/(2+r^*)}}{\log t}}, \quad \quad k_0'' \le d(t) < d,\\
    \bar{C}_1 L_1(\delta)\sqrt{1+d}\> e^{-\bar{C}_0\frac{t}{d^{3 + 2r^*}}} + \bar{C}_1 L_1(\delta) e^{-\bar{C}_0 \frac{t}{d^{2(1+r^*)}\log d}},
    \quad \quad d(t) = d.
    \end{cases}
\end{equation}
If instead Assumption \ref{assump:bounded_row} holds, retaining $k_0'', \bar{C}_1, \bar{C}_0$ but switching $d(t)$ according to \eqref{eqn:dt}, we have
\begin{equation}\label{eqn:main_result_bounded_row}
    \Expect{u_\pi(t)}
    \le 
    \begin{cases}
    \bar{C}_1 L_2(\delta)\sqrt{1+d(t)} \> e^{-\bar{C}_0 t^{1/(1+2r^*)}} + \bar{C}_1 L_2(\delta) \> e^{-\bar{C}_0 \frac{t^{1/(1+2r^*)}}{\log t}}, \quad \quad k_0'' \le d(t) < d,\\
    \bar{C}_1 L_2(\delta)\sqrt{1+d}\> e^{-\bar{C}_0\frac{t}{d^{2r^*}}} + \bar{C}_1 L_2(\delta) e^{-\bar{C}_0 \frac{t}{d^{2r^*}\log d}},
    \quad \quad d(t) = d.
    \end{cases}
\end{equation}
\end{lemma}

\begin{proof}
The proof technique is similar to that of Lemma \ref{lemma:convergence_mean_u}, so we merely sketch the common parts of the argument. 

Suppose Assumption \ref{assump:main} holds, and recall \eqref{eqn:convergence_mean_u7} holds for arbitrary $d' \in \{k_0'', \ldots d\}$. Set $d' = d(t)$, with $d(t)$ as in \eqref{eqn:dt}, and consider first the case $d' < d$.  Setting $A = 2d'\frac{\sigmahigh M}{\sigmalow}\ge A_0$ (by choice of $k_0''$), we apply \eqref{eqn:convergence_mean_u7} exactly as in \eqref{eqn:convergence_mean_u8} to show,
\begin{eqnarray}\label{eqn:main_proof1}
-t \frac{\Lambdak{d'}}{A} &=& -t\frac{1}{2d' \frac{\sigmahigh M}{\sigmalow}(\ak{d'})^2} \le -t\frac{b_0^2 \sigmalow}{2(d')^{1 + 2(1+r^*)} \left(\sigmahigh M\right)^3}
 \le -C_0' t^{1/(4+2r^*)}, \nonumber \\
-t\frac{\delta'}{\Tk{d'}} &\le& -t\frac{\delta'}{1 +   (d')^{2(1+r^*)}\left(\frac{\sigmahigh M }{b_0}\right)^2\> \log\left(2 d' \right)} \le -C_0'\frac{t^{1/(2+r^*)}}{\log t} ,
\end{eqnarray}
for a constant $C'_0 > 0$ that does not depend on $d, d', r^*$. We note the discrepancy of orders in the first and second line of \eqref{eqn:main_proof1} comes from the extra $d'$-dependence in the first term, which was not present in \eqref{eqn:convergence_mean_u8}.

Fix $x \in \mathbb{R}^d_+$. Arguments preceding \eqref{eqn:convergence_mean_u11} remain valid here: Apply Lemma \ref{lemma:convergence_intermediate} with $A = 2d'\frac{\sigmahigh M}{\sigmalow}$, using \eqref{eqn:main_proof1} instead of \eqref{eqn:convergence_mean_u8}, to obtain
\begin{align}\label{eqn:main_proof2}
    \Expect{u(x, t)}
    \le u(x, 0)\left[1 + \exp \left(\frac{\supnorm{x|_{\d(t)}}}{2d(t)\sigmahigh M \ak{\d(t)}} \right) \right] \> e^{-C'' C_0' t^{1/(4+2r^*)}} \nonumber \\
    + u(x, 0)e^{-C'' C_0'\frac{t^{1/(2+r^*)}}{\log t}} 
    + u(x, 0) \, \lambda^{\frac{C_0'}{4}\frac{t^{1/(2+r^*)}}{\log t} - 1} +  C'\bsupnorm{x}{\delta}\>(\beta/\delta)^{t^{1/(4+2r^*)}},
\end{align}
for $k_0'' \le \d(t) < d.$ As in the proof of \eqref{eqn:convergence_mean_u11}, $d(t) \ge \left(\frac{8(4+2r^*)}{C_0' e}\right)^2$ implies $t \ge \frac{4\Tk{d(t)}}{\delta}$.

Applying \eqref{eqn:main_proof2} conditional on $x = X(\infty)$, taking expectations and applying Lemma \ref{lemma:stat_moments} to bound the expectations of associated functionals of $X(\infty)$ produces
\begin{align}\label{eqn:main_proof3}
  &\Expect{u_\pi(t)} = \Expect{u(X(\infty), t)} \notag\\
    &\le \Expect{u(X(\infty), 0)\left[1 + \exp \left(\frac{\supnorm{X|_{\d(t)}(\infty)}}{2d(t)\sigmahigh M \ak{\d(t)}} \right) \right]} \> e^{-C'' C_0' t^{1/(4+2r^*)}} \notag\\
    &+ \Expect{u(X(\infty), 0)}e^{-C'' C_0'\frac{t^{1/(2+r^*)}}{\log t}} 
    + \Expect{u(X(\infty), 0)} \, \lambda^{\frac{C_0'}{4}\frac{t^{1/(2+r^*)}}{\log t} - 1} \notag\\
    &+  C'\Expect{\bsupnorm{X(\infty)}{\delta}}\>(\beta/\delta)^{t^{1/(4+2r^*)}} \notag \\
    &\le \Expect{u(X(\infty), 0)\left[1 + \exp \left(\frac{\supnorm{X|_{\d(t)}(\infty)}}{2d(t)\sigmahigh M \ak{\d(t)}} \right) \right]} \> e^{-C'' C_0' t^{1/(4+2r^*)}} \notag\\
    &\qquad + C''' L_1(\delta) \left( e^{-C'' C_0'\frac{t^{1/(2+r^*)}}{\log t}}
    + \lambda^{\frac{C_0'}{4}\frac{t^{1/(2+r^*)}}{\log t} - 1} +  C'\>(\beta/\delta)^{t^{1/(4+2r^*)}} \right) \notag\\
   & \le 2 C''' L_1(\delta)\sqrt{1 + d(t)}\> e^{-C'' C_0' t^{1/(4+2r^*)}} \notag\\
    &\qquad + C''' L_1(\delta) \left( e^{-C'' C_0'\frac{t^{1/(2+r^*)}}{\log t}}
    + \lambda^{\frac{C_0'}{4}\frac{t^{1/(2+r^*)}}{\log t} - 1} +  C'\>(\beta/\delta)^{t^{1/(4+2r^*)}} \right),
\end{align}
for $k_0'' \le \d(t) < d,$, where $L_1(\delta)$ is defined in Lemma \ref{lemma:stat_moments}. The second inequality above follows from \eqref{eqn:stat_moment_result2} and \eqref{eqn:stat_moment_result3}. In the final line we used the Cauchy-Schwarz inequality, the observation that $(1 + e^z)^2 \le 4 e^{2z}$ for $z \ge 0$, and \eqref{eqn:stat_moment_result1} and \eqref{eqn:stat_moment_result3}. This proves the first line in \eqref{eqn:main_result}, with
\begin{equation}\label{eqn:main_proof4}
    \bar{C}_1 = C''' \left( (2 + C') \vee (1 + \lambda^{-1})\right)
\end{equation}
and 
\begin{equation}\label{eqn:main_proof5}
    \bar{C}_0 = C''C_0' \wedge \frac{C'_0}{4}\log \frac{1}{\lambda}  \wedge \log \frac{\delta}{\beta}.
\end{equation}
We now consider $d(t) = d$, which implies $t \ge d^{4+2r^*}$. Setting $A = 2d\frac{\sigmahigh M}{\sigmalow} \ge A_0$, once again we use \eqref{eqn:convergence_mean_u7} with $d' = d$ to show,
\begin{eqnarray}\label{eqn:main_proof6}
-t \frac{\Lambdak{d'}}{A}  &=& -t\frac{1}{2d \frac{\sigmahigh M}{\sigmalow}(\ak{d})^2} \le -t\frac{b_0^2 \sigmalow}{2d^{1 + 2(1+r^*)} \left(\sigmahigh M\right)^3} \le -C_0' \frac{t}{d^{1 + 2(1+r^*)}}, \nonumber \\
-t\frac{\delta'}{\Tk{d}} &\le& -t\frac{\delta'}{1 +   d^{2(1+r^*)}\left(\frac{\sigmahigh M }{b_0}\right)^2\> \log\left(2 d \right)} \le -C_0' \frac{t}{d^{2(1+r^*)} \log d}.
\end{eqnarray}
The second line of \eqref{eqn:main_result} now follows using Lemma \ref{lemma:convergence_intermediate} and Lemma \ref{lemma:stat_moments} with $A = 2d\frac{\sigmahigh M}{\sigmalow}$ via calculations exactly like \eqref{eqn:main_proof3}, using \eqref{eqn:main_proof6} instead of \eqref{eqn:main_proof1}.

To prove \eqref{eqn:main_result_bounded_row}, i.e. supposing Assumption \ref{assump:bounded_row} holds, we simply use Lemma \ref{lemma:convergence_mean_u}: Set $x = X(\infty)$ and $A = \max\left\lbrace2\frac{\sigmahigh M}{\sigmalow},A_0\right\rbrace$ in \eqref{eqn:convergence_meanu_result_bounded_row} then take expectations with respect to $X(\infty)$. Result \eqref{eqn:main_result_bounded_row} now follows in a manner perfectly analogous to \eqref{eqn:main_proof3}, using \eqref{eqn:stat_moment_result1_bc}, \eqref{eqn:stat_moment_result2_bc}, \eqref{eqn:stat_moment_result3_bc} instead of \eqref{eqn:stat_moment_result1}, \eqref{eqn:stat_moment_result2}, \eqref{eqn:stat_moment_result3}.
\end{proof}

With Lemma \ref{lemma:convergence_mean_u} and Lemma \ref{statconv} in hand, we are now ready to prove Theorem \ref{thm:main_fromx} via \eqref{eqn:main_fromx1}.
\begin{proof}[Proof of Theorem \ref{thm:main_fromx}]  
Fix any $\beta \in (\alpha, 1)$ and $\delta \in (\beta,1)$. Fix $B \in (0,\infty)$ and fix any $x\in \mathcal{S}(b,B)$. First we consider the case in which Assumption \ref{assump:main} holds. Since $d(t)$ of \eqref{eqn:main_result} differs slightly from $\ell(t)$ of Lemma \ref{lemma:convergence_mean_u}, we must take a little care to match the convergence rates appropriately.

Recall $\ell(t)$ of Lemma \ref{lemma:convergence_mean_u} is given as $\ell(t) = d\wedge \floor{t^{1/(3+2r^*)}}$. Recall from the statement of that lemma the term $k_0' = k_0'(r^*) = k_0 \vee \left(\frac{8(3+2r^*)}{C_0'e}\right)^2$. Then $\left(k_0 \vee \left(\frac{8(3+2r^*)}{C_0'e}\right)^2 + 1\right)^{3+2r^*} \le t < d^{3+2r^*}$ implies $k_0' \le \ell(t) < d$. As a result we have directly from the first line of \eqref{eqn:convergence_meanu_result_main}, using $A=A_0$ ($A_0$ defined in Lemma \ref{lemma:boundary_hits})
\begin{align}\label{eqn:main_fromx2}
    \Expect{u(x, t)} &\le C_1 \left(u(x, 0)e^{\frac{\supnorm{x|_{\ell(t)}}}{A_0\sigmalow\ak{\ell(t)}}} + \bsupnorm{x}{\delta}\right) \> e^{-\frac{C_0}{A_0} t^{1/(3+2r^*)}}
    + C_1u(x, 0)\> e^{-C_0 \frac{t^{1/(3+2r^*)}}{\log t}}\notag\\
   & \le C'_1 \supnorm{x} \left(e^{\frac{\supnorm{x|_{\ell(t)}}}{A_0\sigmalow\ak{\ell(t)}}} + 1\right) \> e^{-\frac{C_0}{A_0}t^{1/(3 + 2r^*)}}
    + C'_1 \supnorm{x}\> e^{-C_0 \frac{t^{1/(3 + 2r^*)}}{ \log t}} \notag\\
   & \le C'_1 \supnorm{x} \left(e^{B/\sigmalow^2} + 1\right) \> e^{-\frac{C_0}{A_0}t^{1/(3 + 2r^*)}}
    + C'_1 \supnorm{x}\> e^{-C_0 \frac{t^{1/(3 + 2r^*)}}{ \log t}},
\end{align}
for $\left(k_0 \vee \left(\frac{8(3+2r^*)}{C_0'e} \right)^2 + 1\right)^{3+2r^*} \le t < d^{3+2r^*}$, where $C'_1$ is a constant not depending on $d, x$. The second-last line above follows from the observation $u(x, 0) \le \frac{\beta}{1-\beta} \supnorm{x}$. In the final line we have used $A_0 \ge 1$, $x \in \mathcal{S}(b,B)$ and $\sum_{j = 1}^{\ell(t)} (\rinvk{\ell(t)})_{ij} \ge 1$ for $1 \le i \le \ell(t)$ to get from the definition of $\ak{k}$,
$$
\frac{\supnorm{x|_{\ell(t)}}}{A_0\sigmalow\ak{\ell(t)}} \le \sigmalow^{-2} \> \blow^{(\ell(t))}\supnorm{x|_{\ell(t)}}  \le \sigmalow^{-2}B.
$$

Now recall $d(t) = d \wedge \floor{t^{1/(4 + 2r^*)}}$, which by definition gives $d(t) \le \ell(t)$. Therefore, recalling $k_0''$ from Lemma \ref{statconv}, $\left(k_0''^2 + 1\right)^{4+2r^*} \le t < d^{3+2r^*}$ implies $k_0 \vee \left(\frac{8(4+2r^*)}{C_0'e}\right)^2 \le d(t) \le \ell(t) < d$. Hence, we combine \eqref{eqn:main_fromx1}, \eqref{eqn:main_fromx2} with the first line of \eqref{eqn:main_result} to obtain
\begin{align}\label{eqn:main_fromx3}
    &\Expect{\bnorm{\left(X(x, t) - X(X(\infty), t)\right)}} \notag\\
    &\qquad\le \Expect{u_\pi(t)}
    + C'_1 \supnorm{x} \left(e^{B/\sigmalow^2} + 1\right) \> e^{-\frac{C_0}{A_0}t^{1/(3 + 2r^*)}}
    + C'_1 \supnorm{x}\> e^{-C_0 \frac{t^{1/(3 + 2r^*)}}{ \log t}} \notag\\
    &\qquad\le \bar{C}_1 L_1(\delta)\sqrt{1+t^{1/(4+2r^*)}} \> e^{-\bar{C}_0 t^{1/(4+2r^*)} } + \bar{C}_1 L_1(\delta) \> e^{-\bar{C}_0\frac{t^{1/(2+r^*)}}{\log t}} \notag\\
    &\qquad\qquad+ C'_1 \supnorm{x} \left(e^{B/\sigmalow^2} + 1\right) \> e^{-\frac{C_0}{A_0}t^{1/(3 + 2r^*)}}
    + C'_1 \supnorm{x}\> e^{-C_0 \frac{t^{1/(3 + 2r^*)}}{ \log t}},
\end{align}
for $\left(k_0''^2 + 1\right)^{4+2r^*} \le t < d^{3+2r^*}$. Now if $d^{3+2r^*} \le t < d^{4+2r^*}$ the bound on $\Expect{u_\pi(t)}$ in the first line of \eqref{eqn:main_result} continues to hold, and the bound on $\Expect{u(x, t)}$ from the second line of \eqref{eqn:convergence_meanu_result_main} is now valid. Thus, we have
\begin{align}\label{eqn:main_fromx4}
    &\Expect{\bnorm{\left(X(x, t) - X(X(\infty), t)\right)}} \notag \\
    &\qquad \le \bar{C}_1 L_1(\delta)\sqrt{1+t^{1/(4+2r^*)}} \> e^{-\bar{C}_0 t^{1/(4+2r^*)} } + \bar{C}_1 L_1(\delta) \> e^{-\bar{C}_0\frac{t^{1/(2+r^*)}}{\log t}} \notag\\
    &\qquad\qquad + C'_1 \supnorm{x} \left(e^{B/\sigmalow^2} + 1\right) \> e^{-\frac{C_0}{A_0}\frac{t}{d^{2(1+r^*)}}} 
    + C'_1 \supnorm{x}\> e^{-C_0 \frac{t}{d^{2(1+r^*)} \log d}} \notag\\
    &\qquad \le \bar{C}_1 L_1(\delta)\sqrt{1+t^{1/(4+2r^*)}} \> e^{-\bar{C}_0 t^{1/(4+2r^*)} } + \bar{C}_1 L_1(\delta) \> e^{-\bar{C}_0\frac{t^{1/(2+r^*)}}{\log t}} \notag\\
    &\qquad\qquad + C'_1 \supnorm{x} \left(e^{B/\sigmalow^2} + 1\right) \> e^{-\frac{C_0}{A_0}t^{1/(3+2r^*)} } 
    + C'_1 \supnorm{x}\> e^{-3 C_0 \frac{t^{1/(3+2r^*)} }{ \log t}},
\end{align}
for $d^{3+2r^*} \le t < d^{4+2r^*}$, where the final inequality follows from $\frac{t}{d^{2(1+r^*)}} \ge  t^{1/(3+2r^*)}$ and $\log d \le \frac{\log t}{3 + 2r^*} \le \frac{\log t}{3}$. The first line in \eqref{eqn:main_fromx_result1} follows from \eqref{eqn:main_fromx3}, \eqref{eqn:main_fromx4} by taking $\beta = \sqrt{\alpha}$ and $\delta = \alpha^{1/4}$ after keeping only leading-order terms in the above bounds, for simplicity.

To prove the second line in \eqref{eqn:main_fromx_result1}: Note the second lines of \eqref{eqn:main_result} and \eqref{eqn:convergence_meanu_result_main} remain valid for all $t \ge d^{4 + 2r(d)}$. Applying those results to \eqref{eqn:main_fromx1} and otherwise proceeding as in the lead-up to \eqref{eqn:main_fromx4}
\begin{align}\label{eqn:main_fromx5}
    &\Expect{\bnorm{\left(X(x, t) - X(X(\infty), t)\right)}} \notag\\
    & \qquad\le C_1 \supnorm{x}e^{\frac{\supnorm{x}}{A_0\sigmalow\ak{d}}} \> e^{-\frac{C_0}{A_0}\frac{t}{d^{2(1+r^*)}}} 
    + C_1 \supnorm{x}e^{- C_0 \frac{t}{d^{2(1+r^*)}\log d}} \notag\\
    &\qquad\qquad + C_1 L_1(\delta)\sqrt{1+d}\> e^{-\bar{C}_0\frac{t}{d^{3 + 2r^*}}}
    + C_1 L_1(\delta) e^{-C_0 \frac{t}{d^{2(1+r^*)}\log d}} \notag\\
    & \qquad \le C_1 \supnorm{x}e^{B/\sigmalow^2} \> e^{-\frac{C_0}{A_0}\frac{t}{ d^{2(1+r^*)}}}
    + C_1 \supnorm{x}e^{- C_0 \frac{t}{d^{2(1+r^*)}\log d}} \notag\\
    &\qquad\qquad + C_1 L_1(\delta)\sqrt{1+t^{1/(4 + 2r^*)}}\> e^{-\bar{C}_0\frac{t}{d^{3 + 2r^*}}} 
    + C_1 L_1(\delta) e^{-C_0 \frac{t}{d^{2(1+r^*)}\log d}},
\end{align}
for $t \ge d^{4 + 2r^*}$ and constants $C_0, C_1 > 0$ not depending on $d, r^*$ or $B$. The second line in \eqref{eqn:main_fromx_result1} follows from \eqref{eqn:main_fromx5} by taking $\beta = \sqrt{\alpha}$ and $\delta = \alpha^{1/4}$ and by keeping only leading-order terms in \eqref{eqn:main_fromx5}.

\eqref{eqn:main_fromx_result2} follows in identical fashion, using \eqref{eqn:convergence_meanu_result_bounded_row} instead of \eqref{eqn:convergence_meanu_result_main} and \eqref{eqn:main_result_bounded_row} instead of \eqref{eqn:main_result}. We therefore omit the proof.
\end{proof}

\section{Proofs: Perturbations from stationarity for the Symmetric Atlas Model}\label{sec:proofsam}
Theorem \ref{thm:derivative_process} is a simple specialization of \cite[Theorem 1.2]{andres2009diffusion}. We provide a proof nonetheless to fix notation for the special case considered here, in particular for the fact that \cite[Theorem 1.2]{andres2009diffusion} considers possibly state-dependent drift, whereas solutions to \eqref{eqn:rbm} have constant drift coefficients. To assist a reader in relating the specialized Theorem \ref{thm:derivative_process} to the reference, we make the following comparisons between the notation of \eqref{eqn:rbm} and that of the referenced theorem: The drift coefficient marked $b(\cdot)$ in the reference is the constant $\mu$ here, the domain $G$ is $\Reals_+^d$, the directions of reflection $v_i, i = 1 \ldots d$ in the reference are the columns of the reflection matrix $R$, and $w$ of the reference is the Brownian motion $DB$.
\begin{proof}[Proof of Theorem \ref{thm:derivative_process}]
The almost sure existence and representation \eqref{andrep} of the derivative is a consequence of \cite[Theorem 1.2]{andres2009diffusion}, as we now show. The cited theorem proves the almost sure existence of the derivative process up to the first time $X$ hits a corner (intersection of two or more faces) of the orthant $\mathbb{R}^d_+$. Since the Atlas model does not hit corners by \cite[Theorem 1.9]{sarantsev2015triple}, the derivative $\eta^{i_0}(x, t)$ exists almost surely for any $t \in [0, \infty)$.

For $1 \le i \le d$, the vector $v_i$ of \cite[Theorem 1.2]{andres2009diffusion} is the $i$th column of $R$ here, denoted $R^{(i)}$, and the $i$th inward normal $n_i$ of $\Reals_+^d$ is the standard basis vector $e_i$. Terms $\frac{\partial }{\partial x_j}b(X(x, t))$ of \cite[Theorem 1.2]{andres2009diffusion} are all zero here, since the drift $b(X(x, t)) = \mu t$ does not depend on $x$. For $1 \le i \le d$, define vectors $\left(R^{(i)}\right)^\perp$ and $e_{i}^\perp$, orthogonal to $R^{(i)}$ and $e_{i}$ respectively, by equation (1.1) of \cite{andres2009diffusion} such that these vectors lie in $\operatorname{span}\{R^{(i)},e_{i}\}$. For $d \ge 3$, extend $e_{i},e_{i}^\perp$ by the vectors $\{n^j_i\}_{3 \le j \le d}$ to an orthonormal basis of $\mathbb{R}^d_+$.

From \cite[Theorem 1.2]{andres2009diffusion}, writing $S_{k}^{i_0}(x) = \eta^{i_0}(x, \tau_k)$ for $k \ge 0$,
\begin{equation}\label{eqn:derivative_process_1}
    S^{i_0}_{k+1} = \left\langle S^{i_0}_k(x),  \left(R^{(i_{k+1})}\right)^\perp \right\rangle e_{i_{k+1}}^\perp + \sum_{j=3}^d\left\langle S^{i_0}_k(x), n_{i_{k+1}}^j \right\rangle n_{i_{k+1}}^j,
\end{equation}
and $\eta^{i_0}(x, t)$ is constant on $t \in [\tau_k, \tau_{k+1})$. Moreover,
\begin{equation}\label{eqn:derivative_process_2}
    S^{i_0}_{k} = \left\langle S^{i_0}_k(x), e_{i_{k+1}} \right\rangle e_{i_{k+1}} + \left\langle S^{i_0}_k(x), e_{i_{k+1}}^\perp \right\rangle e_{i_{k+1}}^\perp + \sum_{j=3}^d\left\langle S^{i_0}_k(x), n_{i_{k+1}}^j \right\rangle n_{i_{k+1}}^j.
\end{equation}
In the above representations, the sum $\sum_{j=3}^d$ is taken to be zero if $d =1,2$. From \eqref{eqn:derivative_process_1}, \eqref{eqn:derivative_process_2} and \cite[Lemma 1.7]{andres2009diffusion}, 
\begin{align}\label{eqn:derivative_process_3}
S^{i_0}_{k+1} - S^{i_0}_{k} &= \left\langle S^{i_0}_k(x),  \left(R^{(i_{k+1})}\right)^\perp \right\rangle e_{i_{k+1}}^\perp -\left\langle S^{i_0}_k(x), e_{i_{k+1}} \right\rangle e_{i_{k+1}} - \left\langle S^{i_0}_k(x), e_{i_{k+1}}^\perp \right\rangle e_{i_{k+1}}^\perp \nonumber \\
&= -\left\langle S^{i_0}_k(x), e_{i_{k+1}} \right\rangle R^{(i_{k+1})}
\end{align} 
which proves \eqref{andrep}.

It remains only to prove the random walk representation \eqref{eqn:rwre_deriv}. Define the $\mathbb{R}^{d+2}_+$-valued functions $u(\cdot)$ and $v(\cdot)$ as follows: $v_j(t) := \mathbb{P}_{\Theta(x), i_0}(W(t) = j), \ j \in \{0,\dots,d+1\}$. Set $u_j(t) := \eta^{i_0}_j(x, t)$ for $j = 1, \ldots, d$ and define $u_0(\cdot), u_{d+1}(\cdot)$ iteratively by $u_0(\tau_{k+1}) = u_0(\tau_k) + \frac{1}{2}u_1(\tau_k) \ind{i_{k+1} = 1}$, $u_{d+1}(\tau_{k+1}) = u_{d+1}(\tau_{k}) + \frac{1}{2}u_d(\tau_{k}) \ind{i_{k+1} = d}$, with $u_0(\cdot)$ and $u_{d+1}(\cdot)$ constant on $t \in [\tau_k, \tau_{k+1})$ for $k \ge 0$. 

Using \eqref{andrep} for $u(\cdot)$ and the defining properties of $RW(\Theta(x), i_0)$ for $v(\cdot)$, note that for any $k \ge 0$, $u(t) = u(\tau_k)$ and $v(t) = v(\tau_k)$ for all $t \in [\tau_k, \tau_{k+1})$. Hence, we only need to show that $u(\tau_k) = v(\tau_k), k \ge 0$. This follows from the fact that both $\{u(\tau_k)\}_{k \ge 0}$ and $\{v(\tau_k)\}_{k \ge 0}$ are solutions to the recursive equation in $\{w(k)\}_{k \ge 0}$: $w(0) = e_{i_0}$ and for $k \ge 0$, with the fixed integer sequence $\{i_k\}_{k\ge 0}$,
\begin{align}\label{eqn:derivative_process_4}
w_j(k+1) &= \left(w_j(k) + \frac{1}{2}w_{j-1}(k)\right)\ind{i_{k+1} = j-1} \nonumber \\
&\qquad  + \left(w_j(k) + \frac{1}{2}w_{j+1}(k)\right)\ind{i_{k+1} = j+1} + w_j(k)\ind{i_{k+1} \neq j, j \pm 1},
\end{align}
for $1 \le j \le d,$ and
\begin{align}\label{eqn:derivative_process_5}
w_0(k+1) &= w_0(k) + \frac{1}{2} w_1(k)\ind{i_{k+1} = 1}, \nonumber \\
w_{d+1}(k+1) &= w_{d+1}(k) + \frac{1}{2} w_d(k)\ind{i_{k+1} = d}.
\end{align}
Note an inductive argument implies $\sum_{j=0}^{d+1} w_j(k) = 1$ for all $k \ge 0$ for any solution to \eqref{eqn:derivative_process_4}, \eqref{eqn:derivative_process_5} with $w(0) = e_{i_0}$.

\eqref{eqn:derivative_process_4}, \eqref{eqn:derivative_process_5} hold for $\{u(\tau_k)\}_{k \ge 0}$ by \eqref{andrep} and for $\{v(\tau_k)\}_{k \ge 0}$ by the definition of $RW(\Theta(x), i_0)$. Since $ u(0) - v(0) = 0$, the sequence $\{i_k\}_{k\ge 0}$ is common to $u$ and $v$, and \eqref{eqn:derivative_process_4}, \eqref{eqn:derivative_process_5} are linear recursive equations in $w(\cdot)$, we have $u(\tau_k) - v(\tau_k) = 0$ for all $k \ge 0$. This proves \eqref{eqn:rwre_deriv}.
\end{proof}
\begin{proof}[Proof of Theorem \ref{thm:atlas_perturbation}]
The proof consists of analyzing two regimes: $t < d^{16/3}$ and $t > t_0'' d^4 \log (2d)$. In the former regime, we show that the probability of any of the first $m(t)$ coordinates of $X$ not hitting zero sufficiently often is well-controlled by Lemma \ref{lemma:boundary_hits}, for appropriately chosen time-dependent integer $m(t)$. On the other hand, if each of the first $m(t)$ coordinates of $X$ makes a large number of visits to zero, then the random walk $W$ in the derivative representation of Theorem \ref{thm:derivative_process} makes a large number of jumps, and consequently, has a higher chance of getting absorbed in $0$ or $d+1$ by time $t$. In this case, we bound the right hand side of Corollary \ref{cor:$L^1$_dist} using the probability that a simple random walk does not hit $0$ within a certain number of steps. For $t > t_0 '' d^4 \log (2d)$, we use the approach of \cite{banerjeebudhiraja} via contractions in $L^1$ distance between the synchronously coupled RBMs.

Note that the Atlas model $X$ satisfies $b^{(d')} = -\rinvk{d'}\muk{d'} =  \left\{\rinvk{d'}_{i1}\right\}_{i=1}^{d'} > 0$ for every $d' \in \{1, \ldots, d \}$, and IV of Assumption \ref{assump:main} holds with $k_0 = 1$ and $\sigmalow = \sigmahigh = \sqrt{2}$. Therefore we may apply Lemma \ref{lemma:boundary_hits}, and in preparation we first calculate the quantities $\ak{d'}, \Tk{d'}, \Lambdak{d'}$ for $d' \in \{1, \ldots, d \}$.

Recalling that $\bk{d'}$ is the first column of $\rinvk{d'}$ and computing the row sums of $\rinvk{d'}$ from \eqref{eqn:atlas2} with $d'$ in place of $d$ gives 
\begin{equation}\label{eqn:atlas_perturbation1}
    \ak{d'} = \max_{1\le i \le d'} \frac{1}{\bk{d'}_i} \sum_{j=1}^{d'} (\rinvk{d'})_{ij}\sigma_j = \max_{1\le i \le d'} \frac{\sqrt{2}i(d' + 1 - i)}{2\left(1 - \frac{i}{d' + 1}\right)} = \frac{d'(d'+1)}{\sqrt{2}}. 
\end{equation}
Plugging this into the definitions of $\Tk{d'}, \Lambdak{d'}$ in \eqref{eqn:hitting_pars} and applying Lemma \ref{lemma:boundary_hits}, we obtain $A_0 \ge 1$ not depending on $d,d'$ such that for any $d' \in \{ 1,\ldots, d\}$, $A \ge A_0$ and $t \ge 4\left(1 + \frac{1}{2}(d'(d'+1))^2 \log (2d') \right)/\delta'$,
\begin{align}\label{eqn:atlas_perturbation2}
    \Prob{\N_{d'}(x, t) <   t \frac{\delta'}{4\left(1 + \frac{1}{2}(d'(d'+1))^2 \log (2d')\right)}} \nonumber\\
    \le \exp\left(-t\frac{\delta' C''}{1 + \frac{1}{2}(d'(d'+1))^2 \log (2d')} \right) \nonumber \\
    + \exp\left(-t  \frac{2C''}{A(d'(d'+1))^2}\right) \left\{1 + \exp \left(\frac{ \supnorm{x|_{d'}}}{A d'(d'+1)} \right) \right\},
\end{align}
where $\delta', C''>0$ and $A_0\ge 1$ do not depend on $d',d$.
We now consider $d' = m(t)$, where $m(t) \in \{ 1, \ldots, d\}$ will be a time-dependent integer to be determined later. Recall $\tau^*_0 := \inf\{s \ge 0 : W(s) = 0\}$. For any integer $n(t)$ such that $1 \le n(t) < m(t)$ for $t$ large enough that \eqref{eqn:atlas_perturbation2} holds (a time $t_0$ to be determined below) and with $N(t) = t \frac{\delta'}{4\left(1 + \frac{1}{2}(d'(d'+1))^2 \log (2d')\right)}$ we have for $i \in \{1, \ldots, n(t) \}$,
\begin{align}\label{eqn:atlas_perturbation3}
   & \Expect{\mathbb{P}_{\Theta(x),i}\left(\tau^*_0 > t, \>\max_{0 \le s \le t} W(s) < m(t) \right) }\notag\\
    &\le \Expect{\mathbb{P}_{\Theta(x), i} \left(\tau^*_0 > t, \>\max_{0 \le s \le t} W(s) < m(t)\right) \ind{\N_{m(t)}(x, t) \ge N(t)}}
    + \Prob{\N_{m(t)}(x, t) < N(t)}\notag\\
    &\le 12 \frac{n(t)}{\sqrt{N(t)}} + \Prob{\N_{m(t)}(x, t) < N(t)} \notag\\
    &\le 12 \frac{n(t)}{\sqrt{N(t)}} + \exp\left(- 4C''N(t) \right) \notag\\
   & + \exp\left(-8 C'' N(t) \frac{1 + \frac{1}{2}(m(t)(m(t)+1))^2 \log (2m(t))}{\delta'A(m(t)(m(t)+1))^2}\right) \left\{1 + \exp \left(\frac{ \supnorm{x|_{m(t)}}}{A m(t)(m(t)+1)} \right) \right\}\notag\\
    &\le 12 \frac{n(t)}{\sqrt{N(t)}} + \exp\left(-4C''N(t) \right) 
    + \exp\left(-\frac{4 C''}{\delta' A} N(t) \right) \left\{1 + \exp \left(\frac{ \supnorm{x|_{m(t)}}}{A m(t)(m(t)+1)} \right) \right\}.
\end{align}
The second inequality above follows from \eqref{eqn:rwre_jumptimes1} with $m = m(t)$ and a standard bound on the probability that a simple random walk started from $i \in \{1, \ldots, n(t) \}$ has not hit $0$ after $N(t)$ steps (e.g. \cite{mcmt} Theorem 2.17). The third inequality applies \eqref{eqn:atlas_perturbation2} with $d' = m(t)$ and $t = N(t)(\delta')^{-1}4\left(1 +\frac{1}{2} (m(t)(m(t)+1))^2 \log (2m(t))\right)$.

Now for $i \in \{1, \ldots, n(t) \}$ such that $W(0) = i$, the event $\{\tau^*_0 > t, \> \max_{0 \le s \le t} W(s) \ge m(t)\}$ implies the walk $W$ has taken at least $m(t) - n(t)$ steps without hitting $0$ or $d+1$, where it is absorbed. Thus for all $i \in \{1, \ldots, n(t) \}$, 
\begin{equation}\label{eqn:atlas_perturbation4}
    \mathbb{P}_{\Theta(x), i}\left(\tau^*_0 > t, \> \max_{0 \le s \le t} W(s) \ge m(t)\right)
    \le 12 \frac{n(t)}{\sqrt{m(t) - n(t)}}.
\end{equation}
We now set $m(t)$ so that the bounds in \eqref{eqn:atlas_perturbation3}, \eqref{eqn:atlas_perturbation4} are of the same order. Fix $\epsilon \in (0, 1/4)$ to be chosen later. Set $m(t) = d \wedge \floor{t^{ 1/4 - \epsilon}}$. There exists a $t_0(\epsilon) > 0$ not depending on $d$ such that \begin{equation}\label{eqn:atlas_perturbation5}
    N(t) = \frac{t \delta'}{1 + \frac{1}{2}\left(m(t)(m(t) + 1)\right)^2 \log (2m(t))}
    \ge \frac{t \delta'}{t^{1-4\epsilon} \log (2t^{1/4 - \epsilon})}\\
    \ge t^{3\epsilon},
\end{equation}
for $t \ge t_0(\epsilon)$. From this, we conclude that if $t$ is chosen such that $d \ge \floor{t^{ 1/4 - \epsilon}} \ge \floor{t_0(\epsilon)^{ 1/4 - \epsilon}}$ and $n(t) \le m(t)/2$, the dominating term in \eqref{eqn:atlas_perturbation3} is of order $n(t) t^{-\frac{3}{2} \epsilon}$ and the dominating term in \eqref{eqn:atlas_perturbation4} is of order $n(t) t^{-\frac{1}{8} + \frac{\epsilon}{2}}$. 

Setting $\epsilon = \frac{1}{16}$ matches these orders, at $n(t) t^{-\frac{3}{32}}$. Therefore we set $t_0 = t_0(1/16)$ and define \begin{equation}\label{eqn:atlas_perturbation6}
     m(t) = d \wedge \floor{t^{ 1/4 - \epsilon}} = d \wedge \floor{t^{3/16}}.
\end{equation}
We are now ready to prove \eqref{eqn:atlas_perturbation_result1}. Choose and fix any $n(\cdot)$ as in the statement of the theorem, and recall the definition of $t_0^{(n)}$ given there. We have by \eqref{eqn:atlas_perturbation3}, \eqref{eqn:atlas_perturbation4}, \eqref{eqn:atlas_perturbation5} for any $1 \le i \le n(t)$ and $t \ge t_0^{(n)}$, which implies $2n(t) \vee \floor{t_0^{3/16}} \le m(t) \le d$,
\begin{align}\label{eqn:atlas_perturbation7}
    \Expect{\mathbb{P}_{\Theta(x),i}\left(\tau^*_0 > t\right) }
     &= \Expect{\mathbb{P}_{\Theta(x),i}\left(\tau^*_0 > t, \> \max_{0 \le s \le t} W(s) \ge m(t) \right) }\notag\\
     &+ \Expect{\mathbb{P}_{\Theta(x),i}\left(\tau^*_0 > t, \> \max_{0 \le s \le t} W(s) < m(t) \right)}\notag\\
    & \le 12 \frac{n(t)}{\sqrt{m(t) - n(t)}}
     + 12 \frac{n(t)}{\sqrt{N(t)}} + \exp\left(-4C''N(t) \right)\notag\\
    &\qquad + \exp\left(-\frac{4 C''}{\delta' A} N(t) \right) \left\{1 + \exp \left(\frac{ \supnorm{x|_{m(t)}}}{A m(t)(m(t)+1)} \right) \right\}\notag\\
    &\le 12(1+\sqrt{2})\frac{n(t)}{\sqrt{m(t)}} 
    + \exp\left(-4C''N(t) \right) \nonumber \\
    &+ \exp\left(-\frac{4 C''}{\delta' A} N(t) \right) \left\{1 + \exp \left(\frac{ \supnorm{x|_{m(t)}}}{A m(t)(m(t)+1)} \right) \right\}.
\end{align}
This holds for any $A \ge A_0$ given in \eqref{eqn:atlas_perturbation2}. In the final inequality we used $m(t) \ge 2n(t)$ implies $\sqrt{m(t) - n(t)} \ge 2^{-1/2}\sqrt{m(t)}$, and  $N(t) \ge t^{3/16} \ge m(t)$ by \eqref{eqn:atlas_perturbation6} with the chosen $\epsilon = 1/16$. 

For $x, \tilde{x} \in \Reals_+^d$ with $x > 0$ and $t \ge 0$, by Corollary \ref{cor:$L^1$_dist}, with $\gamma(u) = x + u(\tilde{x} - x), u \in [0,1]$,
\begin{align}\label{eqn:atlas_perturbation8}
    \|X(\tilde{x}, t) - X(x, t)\|_1 
    \le \sum_{i = 1}^{n(t)} |(\tilde{x} - x)_i| \int_{[0, 1)}\mathbb{P}_{\Theta(\gamma(u)), i}(\tau^*_0>t) \> du + \sum_{i=n(t)+1}^d |(\tilde{x} - x)_i|.
\end{align}
Applying \eqref{eqn:atlas_perturbation7} to \eqref{eqn:atlas_perturbation8} and using $N(t) \ge m(t)$, we have for $2n(t) \vee \floor{t_0^{3/16}} \le m(t) \le d$,
\begin{align}\label{eqn:atlas_perturbation9}
    &\Expect{\|X(\tilde{x}, t) - X(x, t)\|_1}\notag\\
    &\le \left[12(1+\sqrt{2}) \|x - \tilde{x} \|_1 \right] \> \frac{n(t)}{\sqrt{m(t)}}
     + \|x - \tilde{x} \|_1 \> \exp\left(-4C''m(t) \right) \notag\\
   &+ \left[\|x - \tilde{x} \|_1\int_{[0, 1)}\left\{1 + \exp \left(\frac{ \supnorm{\gamma(u)|_{m(t)}}}{A m(t)(m(t)+1)} \right) \right\} \> du \right] \> \exp\left(-\frac{4 C''}{\delta' A} m(t) \right) \nonumber \\
   &+ \sum_{i=n(t)+1}^d |(\tilde{x} - x)_i|.
\end{align}
Fix any $Y \in \mathcal{P}(P_1,P_2,\delta)$. Recall $X^Y(\infty) := \left(X(\infty) + Y|_d\right)_+$ and $\alpha^Y(\cdot)$ from \eqref{perdec}. Using \eqref{eqn:atlas_perturbation9} conditioned on $x = X(\infty), \tilde{x} = X^Y(\infty)$, then taking expectations, and using the fact $\|X(\infty)-X^Y(\infty)\|_1 \le \sum_1^d|Y_i| \le \sum_1^\infty|Y_i| = \|Y\|_1$, we have for $2n(t) \vee \floor{t_0^{3/16}} \le m(t) \le d$,
\begin{align}\label{eqn:atlas_perturbation10}
   &\Expect{\|X(X^Y(\infty), t) - X(X(\infty), t) \|_1} - \alpha^Y(n(t)) \notag\\
    &\le \left[12(1+\sqrt{2})\Expect{\|Y\|_1}\right] \> \frac{n(t)}{\sqrt{m(t)}}
    + \Expect{\|Y\|_1} \> \exp\left(-4C''m(t) \right)\notag\\
     &\qquad + \Expect{
    \|Y\|_1\int_{[0, 1)}\left\{1 + \exp \left(\frac{ \supnorm{\gamma(u)|_{m(t)}}}{A m(t)(m(t)+1)} \right) \right\} \> du } \> \exp\left(-\frac{4 C''}{\delta' A} m(t) \right)\notag\\
   & \le \left[12(1+\sqrt{2})\Expect{\|Y\|_1}\right] \> \frac{n(t)}{\sqrt{m(t)}}
    + \Expect{\|Y\|_1} \> \exp\left(-4C''m(t) \right)\notag\\
    &\qquad + \sqrt{\Expect{\|Y\|_1^2}}\left[
    1 + \sqrt{\Expect{\exp \left(\frac{2 \supnorm{Y|_{m(t)}(\infty)}}{A m(t)(m(t)+1)} \right)\exp \left(\frac{2 \supnorm{X|_{m(t)}(\infty)}}{A m(t)(m(t)+1)} \right)}}\> \right]\notag\\
    &\hspace{10cm} \times \exp\left(-\frac{4 C''}{\delta' A} m(t) \right)\notag\\
    & \le \left[12(1+\sqrt{2})\Expect{\|Y\|_1}\right] \> \frac{n(t)}{\sqrt{m(t)}}
    + \Expect{\|Y\|_1} \> \exp\left(-4C''m(t) \right)\notag\\
    &+ \sqrt{\Expect{\|Y\|_1^2}}\left[
    1 + \left(\Expect{\exp \left(\frac{4 \supnorm{Y|_{m(t)}(\infty)}}{A m(t)(m(t)+1)} \right)}\Expect{\exp \left(\frac{4 \supnorm{X|_{m(t)}(\infty)}}{A m(t)(m(t)+1)} \right)}\right)^{1/4}\> \right]\notag\\
    &\hspace{10cm} \times \exp\left(-\frac{4 C''}{\delta' A} m(t) \right).
\end{align}
In the second inequality, we used the Cauchy-Schwarz inequality and the observation that for any $m \in \{1,\dots,d\}$, $\supnorm{\gamma(u)|_{m}} = \max_{1 \le i \le m}\left| X_i(\infty) + u(X^Y_i(\infty) - X_i(\infty))\right|  \le \supnorm{Y|_m(\infty)} + \supnorm{X|_m(\infty)}$ for $u \in [0, 1]$.

As $Y \in \mathcal{P}(P_1,P_2,\delta)$, taking $A = 4\max\{A_0, 4\delta^{-1}\}$, where $A_0$ is given in \eqref{eqn:atlas_perturbation2},
\begin{align}\label{eqn:atlas_perturbation11}
\Expect{\|Y\|_1} \le \sqrt{\Expect{\|Y\|_1^2}} \le \sqrt{P_1}, \nonumber \\ 
\Expect{\exp \left(\frac{4 \supnorm{Y|_{m(t)}(\infty)}}{A m(t)(m(t)+1)} \right)} \le P_2.
\end{align}
Moreover, for the same choice of $A$, we obtain along the same lines as \eqref{eqn:stat_moments6} using the explicit product form distribution of $X|_{m(t)}(\infty)$ (see \eqref{eqn:atlas_stationary}),
\begin{align}\label{eqn:atlas_perturbation12}
    \Expect{\exp \left(\frac{4 \supnorm{X|_{m(t)}(\infty)}}{A m(t)(m(t)+1)} \right)} 
    &\le  \Expect{\exp \left(\frac{ \supnorm{X|_{m(t)}(\infty)}}{A_0 m(t)(m(t)+1)} \right)} \nonumber \\
    &\le 1 + \frac{m(t)}{A_0m(t)(m(t)+1) - 1} \le 2.
\end{align}
Note that we cannot refer to Lemma \ref{lemma:stat_moments} here since Assumption \ref{assump:main} does not hold for the Atlas model. Using the above estimates in \eqref{eqn:atlas_perturbation10}, we obtain for $2n(t) \vee \floor{t_0^{3/16}} \le m(t) \le d$,
\begin{align}\label{eqn:atlas_perturbation13}
&\Expect{\|X(X^Y(\infty), t) - X(X(\infty), t) \|_1} \le \sqrt{P_1}\left[12(1+\sqrt{2}) \>\frac{n(t)}{\sqrt{m(t)}} + \exp\left(-4C''m(t) \right)\right] \nonumber\\
& + \sqrt{P_1}\left(1+(2P_2)^{1/4}\right)\exp\left(-\frac{C''}{\delta' \max\{A_0, 4\delta^{-1}\}} m(t) \right) + \alpha^Y(n(t)).
\end{align}
This proves the first bound in \eqref{eqn:atlas_perturbation_result1} upon noting that $\frac{C''}{\delta' \max\{A_0, 4\delta^{-1}\}} \ge \frac{C''}{\delta'A_0}\frac{\delta}{\delta + 4}$, and for $t^{(n)}_0 \le t < d^{16/3}$ (with $t^{(n)}_0$ as defined in the theorem statement), $2n(t) \vee \floor{t_0^{3/16}} \le m(t) \le d$.

We now address the case when $t$ is large relative to $d$ by applying results from \cite{banerjeebudhiraja}. Using equation (44) of that reference, plugging in the Standard Atlas model parameter estimates calculated in equation \eqref{eqn:atlas_perturbation1} here (with $d = d'$) and, in the reference, equation (12) and parameters given prior to Theorem 1, we have for any $x, \tilde{x} \in \mathbb{R}^d_+$ with $x > 0$,
\begin{multline}\label{eqn:atlas_perturbation14}
    \Expect{\|X(x, t) - X(\tilde{x}, t) \|_1}
    \le \Expect{\|X(x, t) - X(0, t) \|_1} + \Expect{\|X(\tilde{x}, t) - X(0, t) \|_1} \\
   \le C_1\left(\|x\|_1 \exp\left(\frac{C_0'\supnorm{x}}{A'd^4} \right) + \|\tilde{x} \|_1\exp\left(\frac{\supnorm{C_0'\tilde{x}}}{A'd^4} \right) \right) \exp \left(-\frac{C_0}{A'} \frac{t}{d^6\log (2d)} \right),
\end{multline}
for all $t \ge t_0''d^4 \log(2d), \ A' \ge A_0'$, where $C_0, C_0', C_1, t_0'', A_0' \in (0, \infty)$ are dimension-independent constants. Applying \eqref{eqn:atlas_perturbation14} conditional on $x = X(\infty) > 0$ and $\tilde{x} = X^{Y}(\infty) \ge 0$ and taking expectations we have
\begin{multline}\label{eqn:atlas_perturbation15}
    \Expect{\|X(X^Y(\infty), t) - X(X(\infty), t) \|_1} \\
    \le C_1\left(\Expect{\|X(\infty)\|_1 \exp\left(\frac{C_0'\supnorm{X(\infty)}}{A'd^4} \right)} + \Expect{\|X^Y(\infty) \|_1\exp\left(\frac{\supnorm{C_0'X^Y(\infty)}}{A'd^4} \right)} \right)\\
    \times \exp \left(-\frac{C_0}{A'} \frac{t}{d^6\log (2d)} \right)
\end{multline}
for all $t \ge t_0''d^4 \log(2d), \ A' \ge A_0'$. From the explicit distribution of $X(\infty)$ in \eqref{eqn:atlas_stationary}, for any $A' \ge \max\{A_0',4C_0'\}$,
\begin{align}\label{eqn:atlas_perturbation16}
    \Expect{\|X(\infty)\|_1 \exp\left(\frac{C_0'\supnorm{X(\infty)}}{A'd^4} \right)} 
    &\le \sqrt{\Expect{\|X(\infty)\|_1^2}}\sqrt{\Expect{\exp\left(\frac{2C_0'\supnorm{X(\infty)}}{A'd^4} \right)}} 
    \le 2d,
\end{align}
Moreover, as $Y \in \mathcal{P}(P_1,P_2,\delta)$, using $\|X^Y(\infty)\|_1 \le \|X(\infty)\|_1 + \|Y\|_1$ and $\supnorm{X^Y(\infty)} \le \supnorm{X(\infty)} + \supnorm{Y|_d}$, we obtain for any $A' \ge \max\{A_0', 2C_0'\delta^{-1},4C_0'\}$,
\begin{align}\label{eqn:atlas_perturbation17}
\Expect{\|X^Y(\infty)\|_1 \exp\left(\frac{C_0'\supnorm{X^Y(\infty)}}{A'd^4} \right)} \nonumber \\
\le \sqrt{\Expect{\|X^Y(\infty)\|_1^2}}\sqrt{\Expect{\exp\left(\frac{2C_0'\supnorm{X^Y(\infty)}}{A'd^4} \right)}} \nonumber \\
\le \sqrt{2\Expect{\|X(\infty)\|_1^2} + 2\Expect{\|Y\|_1^2}}\sqrt{\Expect{\exp\left(\frac{2C_0'\supnorm{X(\infty)}}{A'd^4} \right)}\Expect{\exp\left(\frac{2C_0'\supnorm{Y|_d}}{A'd^4} \right)}} \nonumber \\
\le \sqrt{4d^2 + 2P_1}\sqrt{2P_2}.
\end{align}
Using \eqref{eqn:atlas_perturbation16} and \eqref{eqn:atlas_perturbation17} in \eqref{eqn:atlas_perturbation15}, fixing $A' = \max\{A_0', 2C_0'\delta^{-1},4C_0'\}$, we obtain
\begin{equation}\label{eqn:atlas_perturbation18}
    \Expect{\|X(X^Y(\infty), t) - X(X(\infty), t) \|_1} \le 2C_1 \sqrt{4d^2 + 2P_1}\sqrt{2P_2}\exp \left(-\frac{C_0}{A'} \frac{t}{d^6\log (2d)} \right),
\end{equation}
for $t \ge t_0''d^4 \log(2d),$ which proves the second bound in \eqref{eqn:atlas_perturbation_result1}, and completes the proof of the theorem.
\end{proof}

%
%

\section*{Acknowledgements}
SB was supported in part by the NSF CAREER award DMS-2141621.

The authors acknowledge Soumik Pal for suggesting a version of the perturbation problem for the Symmetric Atlas model that initiated this work. They also thank Amarjit Budhiraja and Andrey Sarantsev for numerous insightful discussions.

The authors also thank two anonymous referees and an associate editor for their careful reading and valuable feedback that greatly improved the readability of the article.
 


\bibliographystyle{imsart-number} 
\bibliography{refs.bib}       

\begin{thebibliography}{37}

\bibitem{andres2009diffusion}
\begin{barticle}[author]
\bauthor{\bsnm{Andres},~\bfnm{S.}\binits{S.}}
(\byear{2009}).
\btitle{Diffusion processes with reflection}.
\bjournal{PhD thesis}
\bvolume{TU Berlin}.
\end{barticle}
\endbibitem

\bibitem{atarbudhiraja}
\begin{barticle}[author]
\bauthor{\bsnm{Atar},~\bfnm{Rami}\binits{R.}},
  \bauthor{\bsnm{Budhiraja},~\bfnm{Amarjit}\binits{A.}} \AND
  \bauthor{\bsnm{Dupuis},~\bfnm{Paul}\binits{P.}}
(\byear{2001}).
\btitle{On positive recurrence of constrained diffusion processes}.
\bjournal{The Annals of Probability}
\bvolume{29}
\bpages{979--1000}.
\end{barticle}
\endbibitem

\bibitem{banerjeebudhiraja}
\begin{barticle}[author]
\bauthor{\bsnm{Banerjee},~\bfnm{Sayan}\binits{S.}} \AND
  \bauthor{\bsnm{Budhiraja},~\bfnm{Amarjit}\binits{A.}}
(\byear{2020}).
\btitle{Parameter and dimension dependence of convergence rates to stationarity
  for Reflecting {B}rownian Motions}.
\bjournal{The Annals of Applied Probability}
\bvolume{30}
\bpages{2005-2029}.
\end{barticle}
\endbibitem

\bibitem{banerjee2021domains}
\begin{barticle}[author]
\bauthor{\bsnm{Banerjee},~\bfnm{Sayan}\binits{S.}} \AND
  \bauthor{\bsnm{Budhiraja},~\bfnm{Amarjit}\binits{A.}}
(\byear{2021}).
\btitle{Domains of attraction of invariant distributions of the infinite Atlas
  model}.
\bjournal{The Annals of Probability (to appear)}.
\end{barticle}
\endbibitem

\bibitem{banerjee2020rates}
\begin{barticle}[author]
\bauthor{\bsnm{Banerjee},~\bfnm{Sayan}\binits{S.}} \AND
  \bauthor{\bsnm{Burdzy},~\bfnm{Krzysztof}\binits{K.}}
(\byear{2021}).
\btitle{Rates of convergence to equilibrium for potlatch and smoothing
  processes}.
\bjournal{The Annals of Probability}
\bvolume{49}
\bpages{1129--1163}.
\end{barticle}
\endbibitem

\bibitem{banerjee2020ergodicity}
\begin{barticle}[author]
\bauthor{\bsnm{Banerjee},~\bfnm{Sayan}\binits{S.}} \AND
  \bauthor{\bsnm{Sankararaman},~\bfnm{Abishek}\binits{A.}}
(\byear{2021}).
\btitle{Ergodicity and steady state analysis for {I}nterference {Q}ueueing
  Networks}.
\bjournal{AMS Contemporary Mathematics: Special volume in honor of M. M. Rao,
  \emph{to appear}}.
\end{barticle}
\endbibitem

\bibitem{blanchet-chen}
\begin{barticle}[author]
\bauthor{\bsnm{Blanchet},~\bfnm{J.}\binits{J.}} \AND
  \bauthor{\bsnm{Chen},~\bfnm{X.}\binits{X.}}
(\byear{2020}).
\btitle{{Rates of Convergence to Stationarity for Reflected Brownian Motion}}.
\bjournal{Mathematics of Operations Research}
\bvolume{45}
\bpages{660-681}.
\end{barticle}
\endbibitem

\bibitem{blanchent2020efficient}
\begin{barticle}[author]
\bauthor{\bsnm{Blanchet},~\bfnm{Jose}\binits{J.}},
  \bauthor{\bsnm{Chen},~\bfnm{Xinyun}\binits{X.}},
  \bauthor{\bsnm{Si},~\bfnm{Nian}\binits{N.}} \AND
  \bauthor{\bsnm{Glynn},~\bfnm{Peter~W}\binits{P.~W.}}
(\byear{2021}).
\btitle{Efficient steady-state simulation of high-dimensional stochastic
  networks}.
\bjournal{Stochastic Systems}
\bvolume{11}
\bpages{174--192}.
\end{barticle}
\endbibitem

\bibitem{budhiraja_lee}
\begin{barticle}[author]
\bauthor{\bsnm{Budhiraja},~\bfnm{A.}\binits{A.}} \AND
  \bauthor{\bsnm{Lee},~\bfnm{C.}\binits{C.}}
(\byear{2007}).
\btitle{{Long time asymptotics for constrained diffusions in polyhedral
  domains}}.
\bjournal{Stochastic Processes and their Applications}
\bvolume{117}
\bpages{1014–1036}.
\end{barticle}
\endbibitem

\bibitem{cabezas2019brownian}
\begin{barticle}[author]
\bauthor{\bsnm{Cabezas},~\bfnm{Manuel}\binits{M.}},
  \bauthor{\bsnm{Dembo},~\bfnm{Amir}\binits{A.}},
  \bauthor{\bsnm{Sarantsev},~\bfnm{Andrey}\binits{A.}} \AND
  \bauthor{\bsnm{Sidoravicius},~\bfnm{Vladas}\binits{V.}}
(\byear{2019}).
\btitle{Brownian Particles with Rank-Dependent Drifts: Out-of-Equilibrium
  Behavior}.
\bjournal{Communications on Pure and Applied Mathematics}
\bvolume{72}
\bpages{1424--1458}.
\end{barticle}
\endbibitem

\bibitem{dembo_rwre}
\begin{barticle}[author]
\bauthor{\bsnm{Dembo},~\bfnm{A.}\binits{A.}},
  \bauthor{\bsnm{Gantert},~\bfnm{N.}\binits{N.}} \AND
  \bauthor{\bsnm{Zeitouni},~\bfnm{O.}\binits{O.}}
(\byear{2004}).
\btitle{Large Deviations for Random Walk in Random Environment with Holding
  Times}.
\bjournal{The Annals of Probability}
\bvolume{32}
\bpages{996-1029}.
\end{barticle}
\endbibitem

\bibitem{dembo2019infinite}
\begin{barticle}[author]
\bauthor{\bsnm{Dembo},~\bfnm{A.}\binits{A.}},
  \bauthor{\bsnm{Gantert},~\bfnm{N.}\binits{N.}} \AND
  \bauthor{\bsnm{Zeitouni},~\bfnm{O.}\binits{O.}}
(\byear{2019}).
\btitle{The intinite Atlas process: Convergence to equilibrium}.
\bjournal{Annales de l'Institut Henri Poincar\'{e}, Probabilit\'{e} et
  Statistiques}
\bvolume{55}
\bpages{607-619}.
\end{barticle}
\endbibitem

\bibitem{dembo2017equilibrium}
\begin{barticle}[author]
\bauthor{\bsnm{Dembo},~\bfnm{Amir}\binits{A.}} \AND
  \bauthor{\bsnm{Tsai},~\bfnm{Li-Cheng}\binits{L.-C.}}
(\byear{2017}).
\btitle{Equilibrium fluctuation of the {A}tlas model}.
\bjournal{The Annals of Probability}
\bvolume{45}
\bpages{4529--4560}.
\end{barticle}
\endbibitem

\bibitem{fernholz2002stochastic}
\begin{bincollection}[author]
\bauthor{\bsnm{Fernholz},~\bfnm{E.~R.}\binits{E.~R.}}
(\byear{2002}).
\btitle{Stochastic portfolio theory}.
In \bbooktitle{Stochastic Portfolio Theory}
\bpages{1--24}.
\bpublisher{Springer}.
\end{bincollection}
\endbibitem

\bibitem{harrisonreiman}
\begin{barticle}[author]
\bauthor{\bsnm{Harrison},~\bfnm{J.}\binits{J.}} \AND
  \bauthor{\bsnm{Reiman},~\bfnm{M.}\binits{M.}}
(\byear{1981}).
\btitle{{Reflected Brownian motion on an orthant}}.
\bjournal{Ann. Probab.}
\bvolume{9}
\bpages{302-308}.
\end{barticle}
\endbibitem

\bibitem{harrisonwilliams}
\begin{barticle}[author]
\bauthor{\bsnm{Harrison},~\bfnm{J.}\binits{J.}} \AND
  \bauthor{\bsnm{Williams},~\bfnm{R.}\binits{R.}}
(\byear{1987}).
\btitle{{Brownian models of open queueing networks with homogeneous customer
  populations}}.
\bjournal{Stochastics}
\bvolume{22}
\bpages{77-115}.
\end{barticle}
\endbibitem

\bibitem{harrisonwilliams_expo}
\begin{barticle}[author]
\bauthor{\bsnm{Harrison},~\bfnm{J.}\binits{J.}} \AND
  \bauthor{\bsnm{Williams},~\bfnm{R.}\binits{R.}}
(\byear{1987}).
\btitle{{Multidimensional reflected Brownian motions having exponential
  stationary distributions}}.
\bjournal{The Annals of Probability}
\bvolume{15}
\bpages{115-137}.
\end{barticle}
\endbibitem

\bibitem{ichiba_karatzas}
\begin{barticle}[author]
\bauthor{\bsnm{Ichiba},~\bfnm{T.}\binits{T.}},
  \bauthor{\bsnm{Karatzas},~\bfnm{I.}\binits{I.}} \AND
  \bauthor{\bsnm{Shkolnikov},~\bfnm{M.}\binits{M.}}
(\byear{2013}).
\btitle{{Strong solutions of stochastic equations with rank-based
  coefficients}}.
\bjournal{Probab. Theory Relat. Fields}
\bvolume{156}
\bpages{229–248}.
\end{barticle}
\endbibitem

\bibitem{pal}
\begin{barticle}[author]
\bauthor{\bsnm{Ichiba},~\bfnm{T.}\binits{T.}},
  \bauthor{\bsnm{Pal},~\bfnm{S.}\binits{S.}} \AND
  \bauthor{\bsnm{Shkolnikov},~\bfnm{M.}\binits{M.}}
(\byear{2013}).
\btitle{{Convergence rates for rank-based models with applications to portfolio
  theory}}.
\bjournal{Probab. Theory Relat. Fields}
\bvolume{156}
\bpages{415-448}.
\end{barticle}
\endbibitem

\bibitem{ichiba}
\begin{barticle}[author]
\bauthor{\bsnm{Ichiba},~\bfnm{Tomoyuki}\binits{T.}},
  \bauthor{\bsnm{Papathanakos},~\bfnm{Vassilios}\binits{V.}},
  \bauthor{\bsnm{Banner},~\bfnm{Adrian}\binits{A.}},
  \bauthor{\bsnm{Karatzas},~\bfnm{Ioannis}\binits{I.}} \AND
  \bauthor{\bsnm{Fernholz},~\bfnm{Robert}\binits{R.}}
(\byear{2011}).
\btitle{Hybrid {A}tlas models}.
\bjournal{The Annals of Applied Probability}
\bvolume{21}
\bpages{609--644}.
\end{barticle}
\endbibitem

\bibitem{jourdain2008propagation}
\begin{barticle}[author]
\bauthor{\bsnm{Jourdain},~\bfnm{Benjamin}\binits{B.}} \AND
  \bauthor{\bsnm{Malrieu},~\bfnm{Florent}\binits{F.}}
(\byear{2008}).
\btitle{Propagation of chaos and {P}oincar{\'e} inequalities for a system of
  particles interacting through their CDF}.
\bjournal{The Annals of Applied Probability}
\bvolume{18}
\bpages{1706--1736}.
\end{barticle}
\endbibitem

\bibitem{jourdain2013propagation}
\begin{barticle}[author]
\bauthor{\bsnm{Jourdain},~\bfnm{Benjamin}\binits{B.}} \AND
  \bauthor{\bsnm{Reygner},~\bfnm{Julien}\binits{J.}}
(\byear{2013}).
\btitle{Propagation of chaos for rank-based interacting diffusions and long
  time behaviour of a scalar quasilinear parabolic equation}.
\bjournal{Stochastic partial differential equations: analysis and computations}
\bvolume{1}
\bpages{455--506}.
\end{barticle}
\endbibitem

\bibitem{karatzas_skewatlas}
\begin{binproceedings}[author]
\bauthor{\bsnm{Karatzas},~\bfnm{Ioannis}\binits{I.}},
  \bauthor{\bsnm{Pal},~\bfnm{Soumik}\binits{S.}} \AND
  \bauthor{\bsnm{Shkolnikov},~\bfnm{Mykhaylo}\binits{M.}}
(\byear{2016}).
\btitle{Systems of Brownian particles with asymmetric collisions}.
In \bbooktitle{Annales de l'Institut Henri Poincar{\'e}, Probabilit{\'e}s et
  Statistiques}
\bvolume{52}
\bpages{323--354}.
\bpublisher{Institut Henri Poincar{\'e}}.
\end{binproceedings}
\endbibitem

\bibitem{kella}
\begin{barticle}[author]
\bauthor{\bsnm{Kella},~\bfnm{Offer}\binits{O.}} \AND
  \bauthor{\bsnm{Ramasubramanian},~\bfnm{Sundareswaran}\binits{S.}}
(\byear{2012}).
\btitle{Asymptotic irrelevance of initial conditions for Skorohod reflection
  mapping on the nonnegative orthant}.
\bjournal{Mathematics of Operations Research}
\bvolume{37}
\bpages{301--312}.
\end{barticle}
\endbibitem

\bibitem{mcmt}
\begin{bbook}[author]
\bauthor{\bsnm{Levin},~\bfnm{David~A}\binits{D.~A.}} \AND
  \bauthor{\bsnm{Peres},~\bfnm{Yuval}\binits{Y.}}
(\byear{2017}).
\btitle{Markov chains and mixing times}
\bvolume{107}.
\bpublisher{American Mathematical Soc.}
\end{bbook}
\endbibitem

\bibitem{ramanan_mandelbaum}
\begin{barticle}[author]
\bauthor{\bsnm{Mandelbaum},~\bfnm{Avi}\binits{A.}} \AND
  \bauthor{\bsnm{Ramanan},~\bfnm{Kavita}\binits{K.}}
(\byear{2010}).
\btitle{Directional derivatives of oblique reflection maps}.
\bjournal{Mathematics of Operations Research}
\bvolume{35}
\bpages{527--558}.
\end{barticle}
\endbibitem

\bibitem{meyn2012markov}
\begin{bbook}[author]
\bauthor{\bsnm{Meyn},~\bfnm{Sean~P}\binits{S.~P.}} \AND
  \bauthor{\bsnm{Tweedie},~\bfnm{Richard~L}\binits{R.~L.}}
(\byear{2012}).
\btitle{Markov chains and stochastic stability}.
\bpublisher{Springer Science \& Business Media}.
\end{bbook}
\endbibitem

\bibitem{pitman_pal}
\begin{barticle}[author]
\bauthor{\bsnm{Pal},~\bfnm{S.}\binits{S.}} \AND
  \bauthor{\bsnm{Pitman},~\bfnm{J.}\binits{J.}}
(\byear{2008}).
\btitle{{One-dimensional Brownian particle systems with rank-dependent
  drifts}}.
\bjournal{The Annals of Applied Probability}
\bvolume{18}
\bpages{2179–2207}.
\end{barticle}
\endbibitem

\bibitem{pal2019note}
\begin{barticle}[author]
\bauthor{\bsnm{Pal},~\bfnm{Soumik}\binits{S.}} \AND
  \bauthor{\bsnm{Sarantsev},~\bfnm{Andrey}\binits{A.}}
(\byear{2019}).
\btitle{A note on transportation cost inequalities for diffusions with
  reflections}.
\bjournal{Electronic Communications in Probability}
\bvolume{24}.
\end{barticle}
\endbibitem

\bibitem{reiman_jackson_net}
\begin{barticle}[author]
\bauthor{\bsnm{Reiman},~\bfnm{Martin~I}\binits{M.~I.}}
(\byear{1984}).
\btitle{Open queueing networks in heavy traffic}.
\bjournal{Mathematics of operations research}
\bvolume{9}
\bpages{441--458}.
\end{barticle}
\endbibitem

\bibitem{sarantsev2015triple}
\begin{barticle}[author]
\bauthor{\bsnm{Sarantsev},~\bfnm{A.}\binits{A.}}
(\byear{2015}).
\btitle{Triple and simultaneous collisions of competing {B}rownian particles}.
\bjournal{Electronic Journal of Probability}
\bvolume{20}.
\end{barticle}
\endbibitem

\bibitem{sarantsev_RBM_tail}
\begin{barticle}[author]
\bauthor{\bsnm{Sarantsev},~\bfnm{Andrey}\binits{A.}}
(\byear{2017}).
\btitle{Reflected Brownian motion in a convex polyhedral cone: tail estimates
  for the stationary distribution}.
\bjournal{Journal of Theoretical Probability}
\bvolume{30}
\bpages{1200--1223}.
\end{barticle}
\endbibitem

\bibitem{sarantsev2017two}
\begin{barticle}[author]
\bauthor{\bsnm{Sarantsev},~\bfnm{Andrey}\binits{A.}}
(\byear{2017}).
\btitle{Two-sided infinite systems of competing {B}rownian particles}.
\bjournal{ESAIM: Probability and Statistics}
\bvolume{21}
\bpages{317--349}.
\end{barticle}
\endbibitem

\bibitem{sarantsev_infinite}
\begin{binproceedings}[author]
\bauthor{\bsnm{Sarantsev},~\bfnm{Andrey}\binits{A.}}
(\byear{2017}).
\btitle{Infinite systems of competing Brownian particles}.
In \bbooktitle{Annales de l'Institut Henri Poincar{\'e}, Probabilit{\'e}s et
  Statistiques}
\bvolume{53}
\bpages{2279--2315}.
\bpublisher{Institut Henri Poincar{\'e}}.
\end{binproceedings}
\endbibitem

\bibitem{sarantsev}
\begin{barticle}[author]
\bauthor{\bsnm{Sarantsev},~\bfnm{Andrey}\binits{A.}}
(\byear{2019}).
\btitle{Comparison techniques for competing Brownian particles}.
\bjournal{Journal of Theoretical Probability}
\bvolume{32}
\bpages{545--585}.
\end{barticle}
\endbibitem

\bibitem{sarantsev2017stationary}
\begin{barticle}[author]
\bauthor{\bsnm{Sarantsev},~\bfnm{Andrey}\binits{A.}} \AND
  \bauthor{\bsnm{Tsai},~\bfnm{Li-Cheng}\binits{L.-C.}}
(\byear{2017}).
\btitle{Stationary gap distributions for infinite systems of competing
  {B}rownian particles}.
\bjournal{Electronic Journal of Probability}
\bvolume{22}.
\end{barticle}
\endbibitem

\bibitem{sznitman}
\begin{bbook}[author]
\bauthor{\bsnm{Sznitman},~\bfnm{A-S}\binits{A.-S.}}
(\byear{2004}).
\btitle{Topics in Random Walks in Random Environment}.
\bpublisher{School and Conference on Probability Theory, ICTP Lecture Notes
  Series}.
\end{bbook}
\endbibitem

\end{thebibliography}


\end{document}